\documentclass[a4paper]{amsart}

\usepackage{amsmath,amsfonts,amssymb}
\usepackage{mathtools}

\usepackage{graphicx}

\usepackage[utf8]{inputenc}
\usepackage[T1]{fontenc}
\usepackage{latexsym}

\usepackage{algorithm}
\usepackage{algorithmic}
\usepackage{caption}
\usepackage{subcaption}

\usepackage{cmap}  
\usepackage{microtype}  
\usepackage{stmaryrd}
\usepackage{graphics}
\usepackage{color}
\usepackage{soul}
\usepackage{cancel}
\usepackage{url,float}
\usepackage{bbding}
\usepackage{booktabs}
\usepackage{comment}
\usepackage{enumerate}
\usepackage{bbm}
\usepackage[normalem]{ulem}
\usepackage[hidelinks]{hyperref}
\hypersetup{colorlinks=true,citecolor=blue,linkcolor=blue,
	filecolor=blue,urlcolor=blue}
\usepackage{cleveref}

\newtheorem{theorem}{Theorem}[section]
\newtheorem{lemma}[theorem]{Lemma}
\newtheorem{corollary}[theorem]{Corollary}
\newtheorem{proposition}[theorem]{Proposition}
\newtheorem{problem}[theorem]{Problem}

\theoremstyle{definition}
\newtheorem{definition}[theorem]{Definition}
\newtheorem{notation}[theorem]{Notation}

\theoremstyle{remark}
\newtheorem{remark}[theorem]{Remark}

\newtheorem{thmsimpl}{Simplified version of Theorem~\ref{thm:cplx2D}\hspace{-1.5mm}} 
\newtheorem{thmsimplfinord}{Simplified version of Theorem~\ref{thm:ordre_fini}\hspace{-1.5mm}} 

\numberwithin{equation}{section}

\makeatletter
\newcommand{\removelatexerror}{\let\@latex@error\@gobble}
\makeatother

\newcommand{\R}{\mathbbm{R}}
\newcommand{\ZZ}{\mathbbm{Z}}
\newcommand{\Z}{\mathbbm{Z}}
\newcommand{\rel}{\mathbbm{Z}}
\newcommand{\nat}{\mathbbm{N}}
\newcommand{\ree}{\mathbbm{R}}
\newcommand{\comp}{\mathbbm{C}}

\newcommand{\NN}{\mathbbm{N}}

\newcommand{\rM}{{\mathsf{M}}}
\newcommand{\Muniv}{\rM_1}
\newcommand{\Mbiv}{\rM_2}

\newcommand{\cDun}{{{\mathcal D}_{a_1,b_1,K_1}}}
\newcommand{\Prec}{{\mathsf{tprec}}}
\newcommand{\PracPrec}{{\mathsf{tprec_{comp}}}}
\newcommand{\InterPrec}{\mathfrak{p}}
\newcommand{\PracMc}{M_{c,\mathsf{comp}}}
\newcommand{\PracMr}{M_{r,\mathsf{comp}}}
\newcommand{\PrachatA}{\hat{A}_{\mathsf{comp}}}
\newcommand{\cM}{{\mathcal{M}}}
\newcommand{\tO}{{\tilde{O}}}
\newcommand{\M}{{\mathcal M}}

\newcommand\C{{\mathbf C}}

\newcommand{\vol}{{\mathrm{vol\,}}}

\renewcommand{\le}{\leqslant}
\renewcommand{\ge}{\geqslant}
\renewcommand{\leq}{\leqslant}
\renewcommand{\geq}{\geqslant}
\renewcommand{\phi}{\varphi}

\newcommand\textred[1]{{\color{red}\bfseries#1}}
\newcommand\mathred[1]{{\color{red}\boldsymbol{#1}}}
\newcommand\red[1]{\relax\ifmmode\mathred{#1}\else\textred{#1}\fi}

\newcommand\textblue[1]{{\color{blue}\bfseries#1}}
\newcommand\mathblue[1]{{\color{blue}\boldsymbol{#1}}}
\newcommand\blue[1]{\relax\ifmmode\mathblue{#1}\else\textblue{#1}\fi}

\usepackage{xparse}
\NewDocumentCommand{\transp}{sm}{%
  \IfBooleanTF{#1}{(#2)^{t}}{#2^{t}}%
}

\newcommand{\sumprime}{\ensuremath{\sideset{}{'}\sum}}

\newcommand\incircrel[1]{\mathrel{%
    \ooalign{\small{$#1$}\cr$\bigcirc$}%
}}
\newcommand\circledless{\incircrel{\kern.2ex<}}
\newcommand\circledgtr {\incircrel{\kern.45ex>}}
\newcommand\circledleq {\incircrel{\kern.15ex\leq}}
\newcommand\circledgeq {\incircrel{\kern.6ex\geq}}

\newcommand\augment {\mathbin |}

\DeclareMathOperator{\sgn}{sgn}

\def\K{{\mathcal{K}}}
\DeclareMathOperator{\card}{card}
\newcommand\Mult{{\mathsf{M}}}

\providecommand{\normun}[1]{\lVert#1\rVert_1}
\providecommand{\normeucl}[1]{\lVert#1\rVert_2}
\providecommand{\norminf}[1]{\lVert#1\rVert_{\infty}}
\providecommand{\norminfmunpun}[1]{\lVert#1\rVert_{\infty,[-1,1]}}
\providecommand{\norminfmunpund}[1]{\left \lVert#1 \right \rVert_{\infty,[-1,1] \times [-1,1]}}

\providecommand{\norminfabab}[1]{\lVert#1\rVert_{\infty,[a_1,b_1]\times[a_2,b_2]}}

\begin{document}

\author[N. Brisebarre]{Nicolas Brisebarre}
\address{CNRS, ENS de Lyon, Inria, Université Claude-Bernard Lyon 1, Laboratoire LIP (UMR 5668), Lyon, France.}
\email{Nicolas.Brisebarre@cnrs.fr}

\author[G. Hanrot]{Guillaume Hanrot}
\address{CryptoLab, Inc. \& Université de Lyon, CNRS, ENS de Lyon, Inria, Université Claude-Bernard Lyon 1, Laboratoire LIP (UMR 5668), Lyon, France.}
\email{Guillaume.Hanrot@ens-lyon.fr}

\title[Integer points close to a transcendental curve]{Integer points close to a transcendental curve: an algorithmic approach}

\begin{abstract}
In this article, we propose an algorithmic approach to determine the integer points located near a transcendental curve. This approach is closely related to a celebrated work by Bombieri and Pila and to the so-called Coppersmith method. We establish the underlying theoretical foundations, prove the correctness of the algorithms, study their complexity and present practical experiments; we also compare our approach with previously existing ones.  From a practical point of view, we focus on an instance of our general problem, called the Table Maker's Dilemma, whose solving makes it possible to evaluate a given function with correct rounding. Our experiments show a significant speedup,  up to three orders of magnitude. In particular, our results show that the development of a correctly rounded mathematical library for the binary128 format is now possible at a much smaller cost than with previously existing approaches.
\end{abstract}

\subjclass{Mathematics Subject Classification 2020: Primary 11Y99, 41A05, 65G50;  Secondary 11D75, 11J25}
\keywords{Algorithmic number theory, lattice basis reduction, approximation theory, LLL algorithm, Chebyshev nodes, integer points, computer arithmetic, correct rounding}

\thanks{This work was partly supported by the TaMaDi, FastRelax and NuSCAP projects of the French \emph{Agence Nationale de la Recherche}.}

\maketitle

\section{Introduction}
\label{sec:introduction}

In this paper, we study the problem of explicitly determining rational points with fixed denominator close to a transcendental curve. Formally: 
\begin{problem} \label{probgen} 
Let $ a,  b \in \ree $, $a < b$,  $f : [a,b] \rightarrow \ree $ be a transcendental function analytic in a (complex) neighborhood of $[a,b]$. Let $u, v \in \nat \setminus \{ 0 \}$, $w >0$. Determine the integers $X$, $a \le X/u \le b$   for which  there exists $Y \in \rel$ satisfying
\[
\left | f\left( \frac{X}{u} \right) - \frac{Y}{v} \right | < \frac{1}{w}.
\]
\end{problem}

Our work, mixing ingredients from approximation theory and algorithmic
number theory, follows a strategy similar to the approach initiated by
Bombieri and Pila~\cite{BombieriPila1989} to bound the number of
solutions, an algorithmic version of which has been proposed by Stehlé
et al.~\cite{SLZ2005, Stehle2006} (without knowledge of Bombieri and
Pila's prior work).

Our goal is to improve on those latter algorithmic approaches, both
from a theoretical and a practical point of view. Our main ingredient is
to work with an algorithmic representation of the function $f$ itself
rather than reducing the problem to the polynomial case by replacing
$f$ by an approximation (Taylor) polynomial for $f$.  The difference
may seem subtle, but it raises significant difficulties, while
bringing two major improvements: smaller matrices and the prereduction
 trick (see Sections~\ref{sec:prered}, \ref{sec:comparison}
 and~\ref{sec:expresults}).

\paragraph{\textbf{Our main results.}}
Our main result is Theorem~\ref{thm:cplx2D}, which we present in a simplified version (in particular, with a slightly suboptimal exponent). We introduce the piecewise affine function:
\begin{equation}\label{eq:g}
  g : x \in [1,+\infty) \mapsto \frac{1}{3} \left ( \frac{2x}{1 + \left \lfloor  \sqrt{1/4+2x} -1/2 \right \rfloor}  + \left \lfloor  \sqrt{1/4+2x} -1/2 \right \rfloor \right ) - 1/6.
\end{equation}
Note that $g(1) = 1/2$, $g$ is continuous and increasing and  $g(x) = (2\sqrt{2x}/3) (1 + o(1))$ for $x\rightarrow + \infty$. 
\begin{thmsimpl}
Let $d \in \nat, N = (d+1)(d+2)/2$.  Given fixed $f$ and two fixed real numbers $a < b$, Problem~\ref{probgen}  can heuristically\footnote{We refer to p.~\pageref{page:heuristic} for a discussion of this (mild) heuristic character.}  be solved for $u, v \rightarrow \infty$, $d\rightarrow \infty$, with $d = o(\log(uv))$, $\gamma \in [3,N]$, and
\[
  w = (uv)^{\frac{d}{3 g(N/\gamma) }(1 +o(1))}
\]
  over $[a, b]$ using
  \[
    (b-a) (uv)^{\frac{d}{3 g(N/\gamma) \gamma}(1 +o(1))} 
  \]
  calls to Algorithm~\ref{algo:2variables} with parameter $d$.
\end{thmsimpl} 

We give an intuition regarding the role of the parameter $\gamma$. Our
algorithm subdivides $[a,b]$ into subintervals $[a_1, b_1]$ of length
$h = b_1 - a_1$, and each call to Algorithm~\ref{algo:2variables} looks
for rational values of $f$ with fixed denominator $v$ in a region of
``length'' $h$ and ``width'' $2/w$ around the curve $y = f(x)$. The
parameter $\gamma$ describes the ``shape'' of this elementary region: up
to a factor $2^{O(\gamma)}$, we have $w = h^{-\gamma}$, i.e.,
$\gamma = (-\log w/\log h)(1 + o(1))$ as $h \rightarrow 0$. In
particular, the number of calls is $(b-a)/h$, that is,
$(b-a)\, w^{1/\gamma}$ up to a constant factor, as in the previous
theorem.

The parameter $\gamma$ thus measures the way we balance our efforts
between the width and the length. A large value of $\gamma$ means that
each call covers a long subinterval compared to the width of the strip,
hence fewer calls; but the smallest $w$ that the method can handle,
namely $w = (uv)^{\frac{d}{3g(N/\gamma)}(1+o(1))}$ in the previous
theorem, increases with $\gamma$ since $g$ is increasing: we only detect
points very close to the curve. Conversely, a small value of $\gamma$
yields information in a wider strip around the curve, but requires more
calls. The value $\gamma \approx 1$ corresponds to subintervals of length
of the order of $1/w$; and, when $f$ is monotonic, $\gamma \ll 1$
heuristically corresponds to the case $\gamma \gg 1$ for $f^{-1}$, with
the roles of $u$ and $v$ exchanged. These last two regimes are not
covered by our analysis, which assumes $\gamma \in [3, N]$ (see however
Remark~\ref{rem:gamma13} for $1 < \gamma < 3$).

For $\gamma = N$, the previous theorem yields essentially Bombieri--Pila's estimate~\cite{BombieriPila1989}, whereas for $\gamma = o(N)$, we recover Stehlé's results~\cite{SLZ2005, Stehle2006}. For $\gamma \asymp N$, our analysis improves on Stehlé's and bridges the gap between the two approaches.

If $f$ is further assumed to be entire with moderate growth at infinity,
we prove: 
\begin{thmsimplfinord}
  Let $f$ be an entire function of finite order $\le \theta$, where $\theta > 0$. Let $\lambda \geq 1$ and 
  $\nu > \frac{2\lambda \theta}{3g(\lambda)}$ be real numbers. Then
  Problem~\ref{probgen} can heuristically be solved
  over $[a, b]$ for $uv \rightarrow \infty$
  and
  \[
  w = (uv)^{\frac{\nu^2}{2\theta \lambda} \frac{\log(uv)}{\log \log (uv)}(1 + o(1))},
  \]
  using a single call to Algorithm~\ref{algo:2variables} with parameter
  $d = \nu \frac{\log(uv)}{\log \log (uv)} (1 + o(1))$. 
\end{thmsimplfinord}

This statement improves on~\cite{Stehle2006} which obtained the weaker bound $w = 2^{\Theta(\log^2(uv))}$, for $d = O(\log(uv))$. 

\paragraph{\textbf{Practical improvements and the TMD.}}
Our method also shows significant practical impact on the resolution of Problem~\ref{probgen}. In order to illustrate this impact, we have used it on a particular instance of Problem~\ref{probgen} that 
motivated our work.

Problem~\ref{probdir} is often referred to as the Table Maker's
Dilemma (TMD). We motivate it and describe it briefly, referring to
\cite{BHMZ2026} for a more detailed discussion and extensive
bibliography.

For various reasons, it may be desirable that numerical computations
produce well-defined, \emph{correctly rounded} implementations of the classical functions. Given a rounding function (see~\cite[p. 4]{BHMZ2026}, either to nearest\footnote{with ties to even or ties away from zero.} or directed\footnote{towards zero, towards $\pm \infty$.}), mapping real numbers to machine numbers, this means that the implementation returns the rounding of the actual mathematical value of the function.

In order to evaluate a function $f$ such as the cube root or the
exponential function, one usually evaluates a very good, typically
polynomial, approximation $P$ of it. Building a correctly rounded
implementation of $f$ then boils down to guaranteeing that the rounded
value of $f(x)$ coincides with the rounded value of $P(x)$. This leads
to the following problem:
\begin{problem}[TMD (for directed rounding functions)]
 \label{probdir}
Let $e_1, e_2 \in \rel$, $f : [2^{e_1},2^{e_1+1})\rightarrow [2^{e_2},2^{e_2+1})$.

Let  $p \in \nat$. Determine
\begin{itemize}
\item the set of exact cases
  \[
    EC =\left  \{ X \in \llbracket 2^{p-1},2^p-1 \rrbracket \textrm{ such that } 2^{p-1-e_2} f\left (X/2^{-e_1+ p-1}\right) \in \rel \right \} 
  \]
where $\llbracket 2^{p-1},2^p-1 \rrbracket = \rel \cap [2^{p-1},2^p-1]$;
\item the minimum $\mu (p) \in \ZZ$ such that, for $X \in \llbracket 2^{p-1},2^p-1 \rrbracket \setminus EC$ and for $Y \in \llbracket 2^{p-1},2^p\rrbracket$, we have
\[
\left |2^{p-1-e_2} f \left(\frac{X}{2^{-e_1+p-1}}\right) - Y \right | \ge \frac{1}{2^{\mu(p)}}.
\]
\end{itemize}
\end{problem}

Note that Problem~\ref{probgen} encompasses the TMD: consider $a =
2^{e_1}, b = 2^{e_1 +1}- 2^{e_1 + 1 -p}, u= 2^{p -e_1- 1}$ and $ v = 2^{p -e_2- 1}$.    It also includes a generalization of the question addressed
in~\cite{BombieriPila1989}, which corresponds to the case $a = 0, b =
1, u= v $.

Regarding Problem~\ref{probdir}, experiments using our implementation
demonstrate very significant speedups with respect to the state of the art, up
to three orders of magnitude compared to the implementation of~\cite{Stehle2006}. In particular, we give the first significant results for the
binary128~\cite{BHMZ2026} format. We prove, for instance, that~for $f = \exp$,
$e_1 = -2$, $e_2=0$, one has $\mu(113) \le 11\cdot 113 + 1$ in a few days of computation
using a single-thread CPU.

This work paves the way for the first development of an efficient
correctly rounded mathematical library in the three fundamental
formats binary32, binary64 and binary128, extending work initiated
by the libraries
\texttt{CRlibm}\footnote{\url{https://github.com/taschini/crlibm}}~\cite{CRLIBM2006,Lauter2008}
and
\texttt{CORE-MATH}\footnote{\url{https://core-math.gitlabpages.inria.fr/}}~\cite{SZG2022}
for the binary64 precision C99 standard elementary functions.

\subsection{A rough view of our approach and an outline of the paper}\label{subsec:goal}

We follow the strategy of subdividing the interval under
study into small subintervals, in order to be able to replace the
function under study by a polynomial approximation. However, from that
point we proceed differently. Over a small interval $I$, classical
solutions \emph{reduce} (in the complexity-theoretic sense) to
the polynomial case by noticing that, if $\max_{x\in I} |P(x) - f(x)|\le
\varepsilon$, then 
\begin{multline*}
  \left\{ \textrm{Solutions of Problem~\ref{probgen}}\right.   \left.  \textrm{for}
  f, I=[a, b], u, v, w \right\}
\subset \\
\left\{
\textrm{Solutions of Problem~\ref{probgen} for $P, I=[a, b], u, v, (w^{-1}+\varepsilon)^{-1}$}
\right\}.
\end{multline*}
This is the path taken, for instance, by~\cite{Lef2000,
  LefevreMuller2001a,Lefevre2005,SLZ2005, Stehle2006}. Rather than doing
this, we approximate
$f$ by an algebraic function using uniform approximation: if we assume
$f:[a_1,b_1]\rightarrow \ree$, we search for $P_0$ and $P_1 \in \Z[X_1,
  X_2]$ that are small on a strip around the ``weighted'' curve $(u x, v
f(x))$. This property then implies that the solutions to
Problem~\ref{probgen} are common roots of $P_0$ and $P_1$. Finally, we
use a heuristic coprimality\label{page:heuristic} assumption on $P_0$ and $P_1$, analogous
to the one used in~\cite{BonehDurfee2000,SLZ2005,Stehle2006} to obtain
these bad cases. 

In order to compute $P_0$ and $P_1$, we use ideas and techniques developed by the first author and S.~Chevillard~\cite{BrisebarreChevillard2007}.  If we still assume $f:[a_1,b_1]\rightarrow \ree$, the key idea is to find $P \in \Z[X_1,X_2]$ that is small at some points  $(u x_i,v(f(x_i)+y_i))$ of a strip around the ``weighted'' curve. If the pairs $(x_i,y_i)$ are carefully chosen, these discrete constraints imply uniform smallness over the strip around the curve, cf. \Cref{subsec:overview:lebesgue}. 
 The discrete constraints can be reformulated as the fact that the values of $P$ at  the $(u x_i,v(f(x_i)+y_i))$  are the coordinates of a certain short vector in a Euclidean lattice. The celebrated LLL algorithm, cf. \Cref{subsec:overview:lattices}, then allows for computing a reasonable candidate for $P$.
 This approach, which forces smallness on a strip around this curve, is somehow analogous to~\cite{SLZ2005, Stehle2006}.

 Another important ingredient is a prereduction trick that gives rise
 to a significant speedup. It relies on the hunch that contiguous
 intervals should lead to fairly identical lattices to reduce, hence
 the LLL matrix used to address an interval should quasi-reduce the
 lattice built from the contiguous interval. This holds in our
 situation as we continuously work with the function $f$, so that the
 lattices on two contiguous intervals are closely related. This does
 not work for the previous approaches, which replace $f$ by a different
 polynomial approximation over each interval, yielding lattices which are
 ``close'', but not enough for the idea to work.

We give a state of the art in \Cref{sec:state}. Theoretical
results are presented in \Cref{subsec:state:theory},  while
the existing algorithmic approaches are sketched in
\Cref{subsec:state:algorithmic}. Our approach relies on tools
from Approximation Theory and Euclidean lattice basis reduction and an
idea presented in~\cite{BrisebarreChevillard2007,ChevillardPhDThesis};
we recall them in \Cref{sec:overview}.  Our complete approach
is presented and analyzed in \Cref{sec:two-var}; it is followed
by \Cref{sec:prered}, where we introduce and study the prereduction
idea. We present a comparison with previous work in
\Cref{sec:comparison} and conclude with experimental results
in \Cref{sec:expresults}.

A preliminary version \cite{BH2023} of this work also  describes a simpler, but less powerful, univariate version of our method.

\section{State of the art}
\label{sec:state}

We now recall previous theoretical results about Problems~\ref{probgen} and~\ref{probdir}. Regarding Problem~\ref{probdir}, one can find a more complete state of the art in~\cite[Chap. 12]{MullerEtAl2010} and in~\cite{BHMZ2026}. 

Since we have tested our algorithms on Problem~\ref{probdir} instances, we focus  on TMD-related state-of-the-art for algorithmic approaches in this section.

\subsection{Theoretical results} \label{subsec:state:theory}

Problem~\ref{probgen} is connected to a number of topics from Diophantine Geometry. We focus on references that appear to be the most closely connected to our work. The seminal result by Jarn\'{\i}k~\cite{Jarnik1926} and the extremely fruitful work by Bombieri and Pila~\cite{BombieriPila1989}, including their \textit{determinant method} extended later by Heath-Brown~\cite{HeathBrown2002}, yield upper bounds for the number of integer points on a curve. This corresponds to a modified version of Problem~\ref{probgen} where $w\rightarrow \infty$. As the reader will see in the sequel, our approach shares an important feature with the determinant method: in~\cite{BombieriPila1989}, the key point lies in controlling the size of an alternant determinant, which leads to an estimate of the number of integer points on the curve. 
In the present paper, we estimate the volume of a lattice related to  a similar alternant determinant. For a suitable choice of parameters,  this volume is small enough for the lattice to have small vectors, leading to polynomial constraints on the solutions to Problem~\ref{probgen}. These constraints can be used for the explicit determination of these solutions. 

Tackling Problem~\ref{probdir} can be viewed as a quest for a uniform irrationality measure with restricted denominators. There are partial results in this direction~\cite{Beukers1980,Beukers1981,BB2002,Rivoal2007} but more interestingly in our context,  Khémira and Voutier, following the works~\cite{NW95} and~\cite{KhemiraPhD}, proved in~\cite{KheVou2011} a lower bound (called transcendence measure)  for  the expression~$\left|e^{\beta} - \alpha \right|$, where~$\alpha$ and~$\beta$ are algebraic numbers,  $\beta \neq 0$. In particular, when specialized to floating-point numbers, their result provides interesting (but unfortunately overestimated) upper bounds for the TMD~\cite{BHMZ2026}.

Note that relevant information on the exponential function yields relevant information as well on trigonometric and hyperbolic functions, and their respective inverses, the logarithm function and inverse trigonometric functions, see~\cite{BHMZ2026} for an illustration of this remark in the context of the TMD.

\subsection{Algorithmic approaches to Problem~\ref{probdir}}
\label{subsec:state:algorithmic}

In view of the lack of practicality (or in order to improve on it) of fundamental results discussed in the previous subsection, algorithmic approaches have been developed and used in an extensive way since the late 90s.

A common feature of the existing approaches is that one starts by
splitting the domain into subintervals and replacing the function
(assumed to be sufficiently smooth) by a polynomial, in practice a
Taylor approximation, over the interval under consideration; one is
then reduced to studying the problem in the polynomial case.

Lefèvre, together with Muller~\cite{Lef2000,LefevreMuller2001a,Lefevre2005}, studied reduction to degree 1 polynomials; in this case, the remaining Diophantine problem is to find two integers $x, y$, $|x|\leq X, |y| \le Y$ such that $|\alpha x + \beta - y|$ is minimal, which is solved by elementary Diophantine arguments, either
the three distance theorem, or continued fractions (see, e.g.,~\cite{BertheImbert09}). These ideas lead to an algorithm of complexity\footnote{Here, we use the $\tO(\cdot)$ notation defined as $f(n) = \tO(g(n))$ if and only if there exists a nonnegative integer $k$ such that $f(n) = O(g(n) \log^k g(n))$ ($g$ is implicitly assumed to tend to $+\infty$ at $\infty$).} $\tilde{O}(2^{2p/3})$  as $p\rightarrow \infty$. 

Highly optimized and parallel implementations of this method have proved
invaluable for finding optimal values of $\mu(53)$ for several functions of the
standard C mathematical library. This was a key step towards the development of
\texttt{CRlibm}, the first mathematical library ensuring correctly rounded evaluation with proven worst-case bounds.

Higher degree approximations give rise to more complicated Diophantine
problems. Stehlé, Lefèvre and Zimmermann~\cite{SLZ2005}, further
refined by Stehlé~\cite{Stehle2006}, make use of a technique due to
Coppersmith~\cite{Coppersmith1997, Coppersmith2001} and based on
lattice basis reduction to solve it. We recall Corollaries 4 \& 5
of~\cite{Stehle2006}, adapted to our context.
\begin{theorem}[Stehlé~\cite{Stehle2006}]\label{thm:slz}
Let $f: [1,2] \rightarrow [1,2]$, $f\in C^{\infty}([1,2])$, and assume that there exists $K>0$ such that $|f^{(i)}(x)|\le i!\, K$ for all $x\in [1, 2]$ and all $i \ge 0$. Assume moreover that the derivatives of $f$ can be evaluated efficiently (see~\cite[\S 2.2]{Stehle2006} for a precise statement). For every $\varepsilon > 0$, there exists a heuristic algorithm which, given $p \in \nat$, returns all $X$ and $Y \in \llbracket 2^{p-1},2^p-1 \rrbracket$ such that
\[
\left |2^{p-1} f \left(\frac{X}{2^{p-1}}\right) - Y \right | \le \frac{1}{2^{p}}
\]
in time $O(2^{p(1 + \varepsilon)/2})$ as $p \rightarrow \infty$, where the implied constant depends on $\varepsilon$ and $f$. There also exists a polynomial-time heuristic algorithm which returns all $X$ and $Y \in \llbracket 2^{p-1},2^p-1 \rrbracket$ such that
\[
\left |2^{p-1} f \left(\frac{X}{2^{p-1}}\right) - Y \right | \le \frac{1}{2^{4p^2}}.
\]
This second algorithm reduces a lattice of dimension $O(p^2)$ in $\ree^m$, for some $m = O(p^4)$.
\end{theorem}
As noticed by Stehlé in~\cite{Stehle2006}, this can be seen as an algorithmic counterpart of the theoretical result of Nesterenko and Waldschmidt~\cite{NW95}. 

The heuristic character of the algorithm is in practice rather mild
(i.e., the algorithm works as expected on almost all inputs).

We follow a slightly different approach: the algorithmic content of our
method remains based on lattice basis reduction, but instead of
reducing the problem to a polynomial problem, we keep the problem
connected to the function, thanks to rigorous uniform approximation
techniques based on Chebyshev interpolation. In order to develop our
approach, we therefore give a brief overview of Chebyshev
interpolation/approximation and lattice basis reduction in the next
section.

Before concluding this section, we point to the works~\cite{Elkies2000,CHS2009,AKR2018} which address Hall's
conjecture. The latter can be viewed as a specific instance of a variant of
Problem~\ref{probdir} for an algebraic function. Note that Elkies' approach also consists of
first subdividing the initial interval and then using lattice
reduction to determine interesting cases for Hall's conjecture.

\section{A short overview of uniform approximation and lattice basis reduction}
\label{sec:overview}

We shall require various tools from uniform approximation theory~\cite{FoxParker1968, Rivlin1974,Powell1981,Cheney1982, Boyd2001,Mason2002,Trefethen2013} and algorithmic geometry of numbers~\cite{Lov1986,GruLek1987,Cohen1993,Cas1997,LLL+25,vzGG2013}. 

\subsection{Relation between uniform approximation and interpolation}
\label{subsec:overview:lebesgue}

Let $n\in \NN$, the $n$\mbox{-}th Chebyshev polynomial of the first kind is defined by  $T_n \in \mathbb{R}_n [x]$ and $T_n (\cos t) = \cos (nt) \textrm{ for all } t \in [0, \pi].$ The $T_n$'s can also be defined by 
\[
T_0 (x) = 1, T_1 (x) = x, T_{n + 2} (x) = 2 x T_{n + 1} (x) - T_n (x), \forall n \in \mathbb{N}.
\]

\subsubsection{Interpolation at pairs of Chebyshev nodes}\label{subsec:lebconst}

The zeros of $T_{n+1}$ are
\[
{\mu}_{k,n} = \cos \left( \frac{(n - k + 1/2) \pi}{n + 1} \right), k = 0, \ldots, n.
\]
They are called $(n+1)$-Chebyshev nodes of the first kind. Polynomials interpolating univariate functions at this family give rise to very good uniform approximations over $[-1,1]$ to these functions~\cite{Mason2002,Trefethen2013}.
To be able to work on an interval $[a,b]$, one needs scaled versions of Chebyshev polynomials and nodes. We then define, for $n \in\nat$,
\[
  T_{n,[a,b]}(x) := T_n \left (\frac{2x-b-a}{b-a} \right ), \mu_{k,n,[a,b]} := \frac{(b-a)\mu_{k,n}+a+b}{2}, k =0, \ldots, n.
\]
When there is no ambiguity, we denote the nodes as $\mu_{k,[a,b]}$.
 Note that $T_{n,[a,b]} \left ( \mu_{k,n,[a,b]} \right ) = T_n (\mu_{k,n})$.

Let $N_1, N_2 \in \nat, N_1, N_2 \ge 1$. We now introduce the two dimensional extension of the DCT-II\footnote{This denotes the discrete cosine transform of type 2~\cite{Strang1999,PPST2018}. We can take advantage of fast algorithms~\cite[\S 6.3]{PPST2018} that make it possible to compute it in at most $\mathcal{O} (N_1N_2\log (N_1N_2))$ operations for a fixed and given precision.}
\[
\begin{array}{rrcl}
  \textrm{2D-DCT-II}: & {\mathcal M}_{N_1,N_2} (\ree) & \rightarrow  &  {\mathcal M}_{N_1,N_2} (\ree)\\
                      & (x_{\ell_1,\ell_2})_{\substack{0 \le \ell_1 \le N_1 -1,\\0 \le \ell_2 \le N_2 -1}} & \mapsto & (X_{k_1,k_2})_{\substack{0 \le k_1 \le N_1 -1,\\0 \le k_2 \le N_2-1}}
\end{array}
\]
with
\[
X_{k_1,k_2} = \sum_{0 \le \ell_1 \le N_1 -1}\sum_{0 \le \ell_2 \le N_2 -1} x_{\ell_1,\ell_2} \cos \left ( \frac{k_1(\ell_1 + 1/2)\pi}{N_1} \right )\cos \left ( \frac{k_2(\ell_2 + 1/2)\pi}{N_2} \right ),
\]
for $ k_1= 0, \ldots, N_1-1,  k_2= 0, \ldots, N_2-1$.

Let $a_1, a_2, b_1, b_2 \in \ree$, $a_1 < b_1$ and $a_2 < b_2$. Let
$f$ be a function defined over $[a_1,b_1]\times [a_2,b_2]$. Let
$\ree_{N_1-1, N_2-1}[x, t]$ denote the set of bivariate polynomials $P(x,t)$
with real coefficients and $\deg_x P < N_1, \deg_t P < N_2$. If we
interpolate $f$ by $P (x,t) \in \ree_{N_1-1,N_2-1}[x,t]$ at pairs of
Chebyshev nodes 
 $ (\mu_{k_1,N_1-1,[a_1,b_1]},\mu_{k_2,N_2-1,[a_2,b_2]})_{\substack{0
    \le k_1 \le N_1 -1,\\ 0 \le k_2 \le N_2-1}}$ $ =: (\mu_{k_1,k_2})_{\substack{0
    \le k_1 \le N_1 -1,\\ 0 \le k_2 \le N_2-1}}$, we have the
following expressions for the interpolation polynomials (the proof is
identical to the one variable case~\cite[Chap. 6]{Mason2002}):
\[
P (x,t) = \sumprime_{k_1 = 0}^{ N_1-1} \sumprime_{k_2 = 0}^{N_2 -1} c_{k_1,k_2} T_{k_1,[a_1,b_1]} (x)T_{k_2,[a_2,b_2]} (t) \in \ree_{N_1-1,N_2-1} [x,t],
\]
with 
\begin{equation}\label{eq:coeffs-interpol2D}
  ( c_{k_1,k_2} )_{\substack{0 \le k_1 \le N_1 -1\\0 \le k_2 \le N_2 -1}} = \frac{4}{N_1 N_2} \textrm{2D-DCT-II} \left ( (f(\mu_{N_1 - 1 - \ell_1,N_2 - 1 - \ell_2}) )_{\substack{0 \le \ell_1 \le N_1 -1\\0 \le \ell_2 \le N_2 -1}}\right ). 
\end{equation}
The symbol $\sumprime$ means that the coefficient of index $k_1 = 0$ or $k_2 =0$ has to be halved (hence, $c_{0,0}$ has to be divided by four in particular). Note that, if we introduce
\[
  \widehat{f}: \widehat{f}(z_1,z_2) = f \left ( z_1 \frac{b_1-a_1}{2} + \frac{a_1+b_1}{2}, z_2 \frac{b_2-a_2}{2} + \frac{a_2+b_2}{2} \right ),
\]
the coefficients $c_{k_1,k_2}$ are also the coefficients of the interpolation polynomial in $\ree_{N_1-1,N_2-1} [x,t]$ of $\widehat{f}$ at pairs of Chebyshev nodes of the first kind.

\subsubsection{Uniform approximation using interpolation polynomials}\label{subsec:unifapprox}

Let $\rho > 1$, if $a < b$ are two real numbers, we define the ellipse
\[
  \mathcal{E}_{\rho,a,b} = \left \{ \frac{b-a}{2} \frac{\rho e^{i\theta} + \rho^{-1} e^{-i\theta}}{2} +\frac{a+b}{2}, \theta \in [0,2\pi] \right \}
\]
and let ${E_{\rho,a,b}}$ be the compact  region bounded by the ellipse $\mathcal{E}_{\rho,a,b}$.

Let  $N\in \nat$, $N \ge 1$, we also define 
\[
\eta_{\rho,0} = 1 \textrm{ and }  \eta_{\rho,k} =   \frac{\rho^2+1}{\rho^{2}-1} \textrm{ for } k = 1, \ldots, N-1.
\]
The following proposition establishes Cauchy's inequalities for interpolation polynomials at pairs of (scaled) Chebyshev nodes.
Let $\rho_1, \rho_2 > 1$, $a_1<b_1$, $a_2 < b_2$, we define $\mathcal{E}_{\rho_1,a_1,b_1,\rho_2,a_2,b_2} =\mathcal{E}_{\rho_1,a_1,b_1} \times \mathcal{E}_{\rho_2,a_2,b_2}$  
  and  ${E_{\rho_1,a_1,b_1,\rho_2,a_2,b_2}} = E_{\rho_1,a_1,b_1}\times E_{\rho_2,a_2,b_2}$. 
  
\begin{notation}\label{not:M1M2}
Let $F_1$ (resp. $F_2$) be a function analytic in a neighborhood of $E_{\rho_1,a_1,b_1}$ (resp. $E_{\rho_1,a_1,b_1,\rho_2,a_2,b_2}$). We define
\[
\Muniv (F_1) := \max_{z \in E_{\rho_1,a_1,b_1}}|F_1(z)|, \qquad
\Mbiv (F_2) := \max_{(z_1, z_2) \in E_{\rho_1,a_1,b_1, \rho_2, a_2, b_2}}|F_2(z_1, z_2)|.
    \]
\end{notation}
  
\begin{proposition}\label{prop:ineqcauchy2D}
Let $\rho_1, \rho_2 > 1$,  $a_1<b_1$, $a_2 < b_2$, let  $N_1, N_2 \in \nat, N_1, N_2 \ge 2$, $F$ be a function analytic in a neighborhood of $E_{\rho_1,a_1,b_1,\rho_2,a_2,b_2}$, the coefficients $c_{k_1,k_2}, k_1 =0, \ldots, N_1 -1, k_2 =0, \ldots, N_2 -1$ of the interpolation polynomial  $P_{N_1-1,N_2-1}$ of $F$ at pairs of Chebyshev nodes of the first kind satisfy, for $k_1 = 0, \ldots, N_1 -1$, $k_2 = 0, \ldots, N_2 -1$,
\[
\left | c_{k_1,k_2} \right |  \le   4 \frac{\Mbiv(F)}{\rho_1^{k_1}\rho_2^{k_2}}\eta_{\rho_1,k_1} \eta_{\rho_2,k_2}.
\]
Moreover, we have 
\[
\norminfabab{ F - P_{N_1-1,N_2-1}}
\le \frac{8 \rho_1 \rho_2 \Mbiv(F)}{(\rho_1 -1)(\rho_2 -1)} \left (\frac{1}{\rho_1^{N_1} } + \frac{1}{\rho_2^{N_2} }\right ).
\]
\end{proposition}
\begin{proof}
See \Cref{app:proofscheby}. 
\end{proof}

Note that when $\rho_1 \ge 2, \rho_2 \ge 2$, the term $\frac{\rho_1 \rho_2}{(\rho_1-1)(\rho_2-1)}$ is upper bounded by 4; we shall make use of this simpler version of the bound in the rest of the paper.

\subsection{Euclidean lattices and the LLL algorithm}
\label{subsec:overview:lattices}
Lattices and lattice basis reduction algorithms have become, over the
last 40 years, a major tool in algorithmic number theory; their main
use is to provide algorithms for finding ``small'' integer linear combinations
of arbitrary vectors in $\R^M$. 

\begin{definition}
Let $M \in \NN,$ $M\ge 1$,  a lattice of $\R^M$ is a discrete subgroup of $\R^M$; equivalently,
  a lattice $L \subset \R^M$ is the set of integer linear combinations
  of a family $(b_1, \dots, b_N)$ of $\R$-linearly independent vectors
  of $\R^M$.  We shall then say that $(b_i)_{1\leq i\leq N}$ is a
  basis of $L$, and that $N \leqslant M$ is the dimension (or the rank) of $L$. 
\end{definition}

\begin{proposition} The families $B = (b_i)_{1\leq i\leq N}$,  $C = (c_i)_{1\leq i\leq N}$ are two bases of the same lattice, given   in (row) matrix form if and only if there exists $U\in \mathcal{M}_N(\Z)$,  $\det U \in \{\pm 1\}$,  such that $C = UB$. As a consequence, the quantity $(\det C \transp{C})^{1/2} = (\det B \transp{B})^{1/2}$ is independent of the basis and is associated with the lattice itself -- we shall call it the {\em volume} of the lattice and denote it by $\vol L$.
\end{proposition}

Given a basis $(b_1, \dots, b_N)$ of $L$ as input, finding a shortest nonzero vector in $L$
is called the {\em shortest vector problem}. The decision version of
this problem has been shown \cite{Ajtai1998} to be hard under
randomized reductions; in practice, one thus has to content oneself
with approximation algorithms, such as the LLL algorithm~\cite{LLL1982}:
\begin{theorem}[Lenstra, Lenstra, Lov\'asz, 1982]\label{thm:lll82}
  The LLL algorithm,  given a basis    $(b_1, \dots, b_N)$ of a lattice $L \subset \rel^M$, returns
  a basis $(c_1, \dots, c_N)$ such that $\normeucl{c_1} \leq 2^{(N-1)/4} (\vol L)^{1/N}$, and $\normeucl{c_1}  \leq \left(2^{(N-1)/4}\right)^2 \min_{x\in L-\{0\}} \normeucl{x}$. One also has  $\normeucl{c_2}  \leq 2^{N/4} (\vol L)^{1/(N-1)}$.

  The time complexity of the LLL algorithm is polynomial in the maximal bit-length of the coefficients of the $b_i$’s, the lattice rank $N$, and the embedding dimension $M$. 
\end{theorem}
\begin{proof}
See Theorems 9 and 10 from~\cite[Chap. 2]{LLL+25}, except for the last inequality on $\normeucl{c_2}$ which is a consequence of the proof of Fact 3.3 in~\cite{BonehDurfee2000}.
\end{proof}

The constant 2 in the terms $2^{(N-1)/4}$ and $2^{N/4}$ of the inequalities of the theorem is arbitrary, and could be replaced by any real number $> 4/3$.

We now discuss briefly an improvement due to Akhavi \& Stehlé~\cite{AkhaviStehle2008} to the LLL algorithm in the case where $N$ is much smaller than $M$, the dimension of the ambient space. Let $A$ be an $N\times M$ matrix, the rows of which generate the lattice~$L$; the idea is to reduce a smaller $N\times N$ matrix obtained by a random projection (i.e., multiplying $A$ on the right by a random $M\times N$ matrix), and apply the same transformation to the original matrix. 
\begin{theorem}[Akhavi \& Stehlé, 2008] \label{thm:AkhaviStehle} 
 For all $N$, there is an $n_0(N)$ such that for $M \geq n_0(N)$,  if $P$ is an $M\times N$ matrix whose columns are independent random vectors  picked  independently and uniformly inside the $M$ dimensional  unit ball, and $A' = \mathrm{LLL}(A\cdot P)$;  then, with probability  $\geq 1 - 2^{-N}$, the first row vector of the matrix  $A'(A\cdot P)^{-1} A$ is a vector of $L$ of norm $\leq 2^{4N} (\vol L)^{1/N}$.
\end{theorem}
S.~Torres~\cite{TorresPhD} showed that this idea indeed improves the practical outcome of~\cite{Stehle2006}. As for us, cf. \Cref{sec:expresults}, we noticed that in practice this idea works even for simpler models for the random matrix $P$ such as random, uniform $\{0, \pm 1\}$ coefficients, which give equally good results.

The following result will be useful when estimating volume of lattices
generated by the rows of non-square matrices.
It is a corrected and slightly improved version of~\cite[Theorem 2]{Stehle2006}\footnote{Note that there is an inaccuracy in the statement of~\cite[Theorem 2]{Stehle2006}: $\sqrt{n}$ should be replaced with $\binom{n}{d}^{1/2} d^{d/2}$, where $n$ and $d$ are Stehlé's notations.}. Let $B$ be an $N \times M$ matrix whose rows span a lattice $L$. 
We assume that the entries of $B$ satisfy: $|B_{i,j} | \leq {\mathfrak r}_i \cdot {\mathfrak c}_j$, for some nonnegative ${\mathfrak r}_i$ and ${\mathfrak c}_j$,  $0 \leq i \leq N-1$ and $0 \leq j \leq M -1$. As mentioned in~\cite{Stehle2006}, this is typical of Coppersmith-type lattice bases.
\begin{theorem}\label{thm:stehle-ants} Let $B$ be a rank-$N$ $N \times M$ matrix (hence $M \geq N$), the entries of which are bounded by the product of some quantities ${\mathfrak r}_i$’s and ${\mathfrak c}_j$’s as described above. Let $L$ be the lattice spanned by the rows of the matrix $B$, and $\mathfrak{P}$ the product of the $N$ largest ${\mathfrak c}_j$’s. We have:
  \[
  \vol L = ( \det B \transp{B} )^{1/2} \le \binom{M}{N}^{1/2} N^{N/2} \left ( \prod_{i = 0}^{N-1} {\mathfrak r}_i \right ) \mathfrak{P}.
  \]
\end{theorem}
\begin{proof} 
  Let us denote $\C_0, \dots, \C_{M-1}$ the columns of $B$. The classical Lagrange's identity, which is a particular case of Cauchy--Binet formula~\cite{Gramain1984,Knill2014}, then states that
\begin{multline}\label{eq:cauchybinet}
  \det B \transp{B} = \sum_{0 \leq j_1 < \dots < j_N\leq M-1} \det (\C_{j_1}, \dots, \C_{j_N})^2
\\  \leq  \binom{M}{N} \max_{0\leq j_1 < \dots < j_N\leq M-1} \det (\C_{j_1}, \dots, \C_{j_N})^2.
\end{multline}
  As   $B$ has rank $N$, we have $\left ( \prod_{i = 0}^{N-1} {\mathfrak r}_i \right ) \neq 0$.

  Now, for a given $0 \leq j_1 < \dots < j_N\leq M -1$, if one of the ${\mathfrak c}_j$ is zero, it follows that at least one column of $(\C_{j_1}, \dots, \C_{j_N})$ is zero, hence   $\det (\C_{j_1}, \dots, \C_{j_N}) = 0$. Otherwise, we consider the matrix  $(\C'_{j_1} \dots \C'_{j_N})$ obtained from  $(\C_{j_1} \dots \C_{j_N})$ after having divided the $i$-th row by ${\mathfrak r}_i$ for all $i = 0, \ldots, N-1$ and the $j_k$-th column by ${\mathfrak c}_{j_k}$ for all $k = 1, \ldots, N$. All the coefficients of  $(\C'_{j_1} \dots \C'_{j_N})$ have an absolute value less than or equal to $1$. Hadamard's inequality then implies $ \det (\C'_{j_1}, \dots, \C'_{j_N})^2 \le N^N$. 
  It follows that
\[  \det (\C_{j_1}, \dots, \C_{j_N})^2 =  \left ( \prod_{i = 0}^{N-1} {\mathfrak r}_i \right )^2 \left ( \prod_{k = 1}^{N} {\mathfrak c}_{j_k} \right )^2     \det (\C'_{j_1}, \dots, \C'_{j_N})^2  \le   \left ( \prod_{i = 0}^{N-1} {\mathfrak r}_i \right )^2 \mathfrak{P}^2 N^N.
\]
We conclude by combining the last inequality with~\eqref{eq:cauchybinet}.  
\end{proof}

\section{Description and analysis of the algorithms} \label{sec:two-var}

We consider $u, v \in \nat \setminus \{ 0 \}$, $a_1 < b_1$ and $a_2 < b_2$, and  $f : [a_1,b_1] \rightarrow \ree$ a function that is analytic in a neighborhood of $[a_1,b_1]$. In this section, we develop a heuristic algorithmic approach to determine the integers $X, Y$ such that
\begin{equation}\label{eq:prob2var}
 X/u \in [a_1, b_1] \textrm{ and } a_2 <  Y/v - f(X/u) <  b_2.
\end{equation}
Note that this is a mere reformulation of Problem~\ref{probgen} on the interval $[a_1,b_1]$: for $w > 0$, we set  $b_2 = - a_2 = 1/w$. 

Our approach aims at building a trap for these pairs $(X,Y)$, which
consists of two polynomials $P_0, P_1 \in \rel [X_1,X_2]$ such that,
for $i = 0, 1$, for all $x \in [a_1,b_1]$, $t \in [a_2,b_2]$, we have
$| P_i(ux,v(f(x)+t)) | <1$.

Let $X \in \rel$ be such that $X/u =: x_0 \in [a_1,b_1]$ and let  $Y \in \rel$ be such that $Y/v =: f(x_0) + t_0 $ with $t_0 \in [a_2,b_2]$. Then $P_i(ux_0,v(f(x_0)+t_0)) = P_i (X,Y) \in \rel \cap (-1,1) = \{ 0 \}$, that is to say $(X,Y)$ is a common root of $P_0$ and $P_1$. We make the heuristic assumption that $P_0$ and $P_1$ have no nonconstant common factor; in that case the system $P_0(X, Y) = P_1(X,Y) = 0$ has only finitely many solutions which can be found via elimination. We obtain the list of all the integers $X, Y$ that satisfy~\eqref{eq:prob2var}.

The two polynomials $P_0$, $P_1$ are obtained by computing two short vectors
in a lattice built from $f$ using two-dimensional Chebyshev interpolation
theory; this construction is performed by Algorithm~\ref{algo:build_two_var},
described in \Cref{ssec:algo1}. Algorithm~\ref{algo:2variables}
is our main algorithm; it organizes the whole computation,
using~Algorithm~\ref{algo:build_two_var} as a subroutine. We
describe it in \Cref{ssec:algo2}, prove its correctness
in \Cref{sssec:correct2D} and analyze its complexity in \Cref{subsec:complexity2D}.

\textit{Throughout this section, $N_1, N_2 \ge 2$, $d\ge 2$ will be three integers. We define $N = (d+1)(d+2)/2$.  In order to avoid degenerate situations and trivial output, we shall always assume $N_1 N_2 \ge   N$. We also introduce two positive real parameters $\rho_1, \rho_2$, which we shall assume, for the sake of simplicity\footnote{This assumption causes a small loss of generality regarding functions with singularities very close to the real line.}, to be $\ge 2$. Finally, we shall make use of the following notation: for any $x \in \ree$,  $[ x ]_0 = \lfloor x \rfloor$ if $x \geq 0$ and $\lceil x \rceil$ otherwise. }

\subsection{Algorithm~\ref{algo:build_two_var}: building the auxiliary matrix}
\label{ssec:algo1}
 Let $a_1<b_1$, $a_2 < b_2$, $f$ be analytic in the neighborhood of  $E_{\rho_1,a_1,b_1}$.  We set  $(f_i)_{0\leq i\leq N-1} = (u^k  x^k v^\ell (f(x) + t)^\ell)_{0\leq k+\ell \leq d}$.   Each  $f_i$ is analytic over $E_{\rho_1,a_1,b_1,\rho_2,a_2,b_2}$.  If we interpolate $f_i$ by $Q_i (x,t) \in \ree_{N_1-1,N_2-1} [x,t]$ at pairs of Chebyshev nodes $\mu_{k_1,k_2} =(\mu_{k_1,N_1-1,[a_1,b_1]},$ $\mu_{k_2,N_2-1,[a_2,b_2]})_{\substack{0 \le k_1 \le N_1 -1,\\  0 \le k_2 \le N_2-1}} $, cf. \Cref{subsec:overview:lebesgue},  the corresponding coefficients
$ ( c_{k_1,k_2,i} )_{\substack{0 \le k_1 \le N_1 -1,\\0 \le k_2 \le N_2 -1}}$ are given by
\begin{equation*}
  ( c_{k_1,k_2,i} )_{\substack{0 \le k_1 \le N_1 -1,\\0 \le k_2 \le N_2 -1}} = \frac{4}{N_1 N_2} \textrm{2D-DCT-II} \left ( (f_i(\mu_{N_1 - 1 - \ell_1,N_2 - 1 - \ell_2}) )_{\substack{0 \le \ell_1 \le N_1 -1,\\0 \le \ell_2 \le N_2 -1}}\right )
\end{equation*}
from~\eqref{eq:coeffs-interpol2D}.

Let  $A = \left(A_1 | A_2\right)$ 
 be the $N\times (N_1N_2+N)$ matrix defined by 
 \begin{equation}
   A_1  = \left ( \frac{c_{k_1,k_2,i}}{2^{\delta_{0k_1}+\delta_{0k_2}}} \right )_{\substack{0 \le i \le N - 1,\\ 0 \le k_1 \le N_1 -1,\, 0 \le k_2 \le N_2-1}},  A_2 = \left( \delta_{ij}  R_{i}   \right)_{ 0\leq i, j\leq N-1}, \label{eq:A1A2spec2D} 
 \end{equation}
where $[R_{i}, i = 0, \ldots, N -1]$ is the ordered list
\begin{multline*}
  \left [ 64 u^k v^\ell \Muniv(x)^k \Mbiv(f(x)+t)^\ell \left ( \frac{1}{\rho_1^{N_1}}+ \frac{1}{\rho_2^{N_2}}\right); 
       \ell = 0, \ldots, d, k = 0, \ldots, d - \ell \right].
\end{multline*}
The definitions of $\Muniv$ and $\Mbiv$ are given in Notation~\ref{not:M1M2}.
Recall from Proposition~\ref{prop:ineqcauchy2D} that $\norminf{ f_i - Q_i  } \le A_2[i,i],$ $ i = 0, \ldots, N-1.$
\medskip

The following lemma gathers the main properties of the matrix $A$. The
first part follows from the linearity of Chebyshev interpolation, while the
second is a consequence of the last part of Proposition~\ref{prop:ineqcauchy2D}, assuming $\rho_1, \rho_2 \ge 2$; it is proved in more detail in the proof of~Theorem~\ref{thm:alternants_interpolants2D}.
\begin{lemma}\label{le:prec}
  Let $\mathfrak{w} = (\mathfrak{w}_i)_{0\le i \le N-1} \in \rel^{N}$, $\mathfrak{w}^{(1)} = \left(\mathfrak{w}^{(1)}_{k_1, k_2}\right)_{0\le k_1 \le N_1-1, 0\le k_2 \le N_2-1} := \mathfrak{w}A_1$,
  $\mathfrak{w}^{(2)} = (\mathfrak{w}^{(2)}_j)_{0\le j\le N-1}  := \mathfrak{w}A_2$. Let $Q$ be the Chebyshev
  interpolant of degrees (at most) $N_1-1, N_2-1$ of $F_\mathfrak{w} = \sum_{i=0}^{N-1} \mathfrak{w}_i f_i$
  over $[a_1, b_1]\times [a_2, b_2]$. Then, we have 
    \begin{align*}
      Q & = \sum_{\substack{0\le k_1 \le N_1-1, \\ 0\le k_2 \le N_2 - 1} } \mathfrak{w}^{(1)}_{k_1, k_2}
      T_{k_1, [a_1, b_1]} T_{k_2, [a_2, b_2]}, \\
      \max_{(x, t)\in [a_1,  b_1]\times [a_2, b_2]} & |F_\mathfrak{w}(x, t) - Q(x, t)|  \le
      \sum_{j=0}^{N-1} |\mathfrak{w}^{(2)}_j|.
    \end{align*}
\end{lemma}

In order to avoid numerical instability, we replace the matrix $A$ by
a rounded version of it. We introduce
$\hat{A} = (\hat{A}_1 |\hat{A}_2 )$:
\[
  \hat{A}_1  = \left ( \left [ 2^{\Prec} {A_1}[i,j]  \right ]_0/ 2^{\Prec} \right )_{\substack{0\leq i\leq N-1, \\ 0\leq j\leq N_1N_2-1  }},
   \hat{A}_2 = \left( \lfloor 2^{\Prec}  {A_2}[i,j] \rfloor/ 2^{\Prec}  \right)_{ 0\leq i, j\leq N-1},
\]
where $\Prec = \lceil - \log_2 (\min_{0 \le  i \le N-1} {A_2}[i,i])   + \log_2 (N) \rceil + 2$. The integer $\Prec$ guarantees that $\hat{A}$ preserves the important properties from $A$. In particular, as $2^{\Prec}  {A_2}[i,i] \geq 2^{2} N $,
we have
\begin{equation}\label{eq:a2tprec}
  {\hat{A}_2}[i,i]  \geq 2^{2-\Prec} N > 0.
\end{equation}
This shows that the rows of $\hat{A}$ are linearly independent and define an $N$-dimensional lattice in $\R^{N_1N_2+N}$.

We now derive a simpler equivalent expression for $\Prec$. 
In order to do so, we assume that the set $u [a_1,b_1]$,
resp. $v( f([a_1,b_1]) + [a_2,b_2])$, contains at least one nonzero
integer $n_x$, resp. $n_f$.  Again, this assumption is made without
loss of generality with respect to our problem, since if the
assumption does not hold the problem is trivial.
 \begin{lemma}\label{lem:prec2D}  We have
   \[
     \Prec  = \lceil - \log_2 (R_{0})   + \log_2 (N) \rceil + 2  = \left \lceil - \log_2(\rho_1^{-N_1} + \rho_2^{-N_2}) + \log_2 (N)  \right \rceil - 4
\] 
  and
\[
   \frac{8}{N}(\rho_1^{-N_1} + \rho_2^{-N_2})
     \le 2^{-\Prec} \le \frac{16}{N} (\rho_1^{-N_1}  + \rho_2^{-N_2}).
     \]
\end{lemma}
\begin{proof}
  Under our assumption, we get
  $v ( \Mbiv(f(x)+t) )\ge  |n_f| \ge 1$ and  $ u \Muniv(x) \geq  |n_x| \geq 1$. It then follows that $R_{i} \ge 64 (\rho_1^{-N_1} + \rho_2^{-N_2}) = R_{0}$ for all $i$. Therefore, we get  $\Prec = \lceil - \log_2 (R_{0})   + \log_2 (N) \rceil + 2.$
\end{proof}

\paragraph{\textbf{Convention.}} \Cref{algo:build_two_var} and \Cref{algo:2variables} shall be described in terms of a slightly different matrix $A = (A_1 \mid A_2)$
 in which the diagonal entries of $A_2$ are the quantities actually used by
Algorithm~\ref{algo:build_two_var}, namely
\[
  A_2[i,i] = R_0\, (u B_x)^k \bigl(v (B_f + B_t)\bigr)^\ell ,
  \qquad i \leftrightarrow (k,\ell),
\]
where $B_x$ and $B_t$ are defined at Step~\ref{step:BxBt} and $B_f$ is
any certified upper bound for $\max_{z \in \mathcal{E}_{\rho_1, a_1,
b_1}} |f(z)|$ (see~\Cref{sssec:overst2D}). Our purpose here is to give a truly implementable version of those algorithms. 

The analysis that follows, however, remains based on the matrix $A_2$ as defined by~\Cref{eq:A1A2spec2D}.  This discrepancy does not affect the applicability of the analysis to~\Cref{algo:build_two_var} and \Cref{algo:2variables}. Indeed: 
\begin{itemize}
\item Statements regarding the validity of the algorithms (Theorem~\ref{thm:alternants_interpolants2D} and its corollaries) remain true. Indeed, as  $B_x \ge
\Muniv(x)$ and $B_f + B_t \ge \Mbiv(f(x)+t)$, we still have $\|f_i -
Q_i\|_\infty \le A_2[i,i]$ for all $i$, which is the only property of
$A_2$ used in the proof of Theorem~\ref{thm:alternants_interpolants2D}. Further, 
 we have $\min(uB_x, v(B_f + B_t)) \ge 1$, hence $\min_i A_2[i,i]=R_0$, so that the formulas of
Lemma~\ref{lem:prec2D} remain valid for this modified matrix. 
\item Statements regarding number of intervals / complexity (Proposition~\ref{thm:1/(b-a)_2D} and its corollaries) remain true as long as overestimation is moderate, namely $B_f \le C\cdot \Muniv(f)$ for a fixed constant $C$, an assumption that we make in  the sequel; see Remark~\ref{rem:upperbound}.
\end{itemize}

This leads to the description of Algorithm~\ref{algo:build_two_var}. 
\begin{algorithm}[htp]
	\centering
	\begin{algorithmic}[1]
				
		\REQUIRE Four real numbers $ a_1 <  b_1, a_2 < b_2 $, $f$ a transcendental function analytic in a complex neighborhood of $E_{\rho_1,a_1,b_1}$, $B_f$ an upper bound for $\Muniv(f)$, five positive integers $d, N_1, N_2, u, v$, two real numbers $\rho_1, \rho_2 \ge 2$ such that $N_1, N_2 \geq 2$, $N_1 N_2\geq N := (d+1)(d+2)/2$.
		
		\ENSURE Two matrices $M_c\in \mathcal{M}_{N,N_1N_2} (\rel), M_r \in \mathcal{M}_N (\rel)$, where $N= (d+1)(d+2)/2$, respectively storing scaled values of the coefficients and the remainders, an integer $\Prec$ which is the truncation precision.
		
                \STATE $R_0 \leftarrow   64 \left(\rho_1^{-N_1} + \rho_2^{- N_2} \right)$,  $\Prec \leftarrow \lceil - \log_2 (R_0) + \log_2 (N)  \rceil + 2$ \label{step:prec2D}
                
                \COMMENT Computation of the Chebyshev nodes, listed in reverse order

                \STATE $L_{cheb,x} \leftarrow \left[\frac{b_1-a_1}{2} \cos \left( (j+1/2)\frac{\pi}{N_1} \right) +  \frac{a_1 + b_1}{2}\right]_{0\leq j\leq N_1-1}$

                \STATE $L_{cheb,t} \leftarrow \left[\frac{b_2-a_2}{2} \cos \left( (j+1/2)\frac{\pi}{N_2} \right) +  \frac{a_2 + b_2}{2}\right]_{0\leq j\leq N_2-1}$
                    
                \STATE $M_c \leftarrow [0]_{N \times N_1 N_2} ; M_r \leftarrow  [0]_{N \times N} $

                \STATE $B_x \leftarrow \left | \frac{a_1+b_1}{2} \right | +\frac{b_1-a_1}{4} (\rho_1 + \rho_1^{-1}) $, $B_t \leftarrow  \left | \frac{a_2+b_2}{2} \right | +\frac{b_2-a_2}{4} (\rho_2 + \rho_2^{-1})$ \label{step:BxBt}  
                
                \STATE $i\leftarrow 0$                
                               
                \FOR{$\ell =0$ \TO $d$} \label{step:for_ell2D}
                    \FOR{$k=0$ \TO $d - \ell $} \label{step:for_k2D}
 
                        \STATE $\varphi \leftarrow  ( (x,t) \mapsto (u x)^k (v (f(x) + t))^\ell)$  \label{step:varphi2D}
                    
                        \COMMENT We compute the coefficient matrix: for each function, we compute its value at points of $L_{cheb,x} \times L_{cheb,t}$, use DCT and scale.

                        \STATE $ U \leftarrow \frac{4}{N_1 N_2} \textrm{2D-DCT-II}\left (  (\varphi ( L_{cheb,x}[\ell_1], L_{cheb,t} [\ell_2] ))_{\substack{0 \le \ell_1 \le N_1 -1, \\ 0 \le \ell_2 \le N_2 -1} }\right)$, \label{step:dct2D}

                        \FOR{$k_1 = 0$ \TO $N_1 -1$}

                            \FOR{$k_2 = 0$ \TO $N_2 -1$}

                               \STATE $ M_c[i,k_2 + k_1 N_2] \leftarrow U[k_1,k_2]$.
         
                               \ENDFOR
                               \ENDFOR
                       \FOR{$k_1 = 0$ \TO $N_1 -1$}
                            \STATE $ M_c[i,k_1 N_2] \leftarrow \frac{1}{2} M_c[i,k_1 N_2]$         
                        \ENDFOR
     \FOR{$k_2 = 0$ \TO $N_2 -1$}
     \STATE $ M_c[i,k_2] \leftarrow \frac{1}{2} M_c[i,k_2]$
     \ENDFOR

                        \FOR{$j = 0$ \TO $N_1 N_2 -1$}

                            \STATE $M_c[i,j] \leftarrow \left [  2^{\Prec} M_c[i,j] \right ]_0$\label{step:mc2D}                 
                               
                        \ENDFOR
                            
                        \COMMENT We compute the scaled remainder matrix.                

                        \STATE   $M_r [i,i] \leftarrow \left  \lfloor 2^{\Prec} R_0  (u B_x)^k (v (B_f + B_t))^\ell \right \rfloor$, $i \leftarrow i+1$ \label{step:mr2D} 
                    \ENDFOR                        
                \ENDFOR\label{step:end_for_ell2D}

                \STATE Return $M_c, M_r, \Prec$
	\end{algorithmic}
	\caption{Computation of the lattice to be reduced}
	\label{algo:build_two_var}
\end{algorithm}

\subsection{Algorithm~\ref{algo:2variables}: solving the problem}
\label{ssec:algo2}
As discussed previously, Algorithm~\ref{algo:2variables} solves the
problem by calling Algorithm~\ref{algo:build_two_var} as a
subroutine. If the lattice obtained from the latter contains two
sufficiently short vectors, we (attempt to) use elimination theory to
solve the corresponding bivariate polynomial system. The details are
described in~Algorithm~\ref{algo:2variables}.

\begin{algorithm}[htp]
	\centering
	\begin{algorithmic}[1]
				
                \REQUIRE Four real numbers $ a_1 <  b_1, a_2 < b_2 $, $f$ a transcendental function analytic in a complex neighborhood of $E_{\rho_1,a_1,b_1}$, $B_f$ an upper bound on $\Muniv(f)$, five positive integers $d, N_1, N_2, u, v$, two real numbers $\rho_1, \rho_2 \ge 2$  such that $N_1, N_2 \geq 2$, $N_1 N_2\geq N := (d+1)(d+2)/2$.	  
		\ENSURE If successful, return a list $\mathcal{L}$ such that $\mathcal{L} \supset \big \{ X \in \rel$ such that $a_1 \le X/u \le b_1$ and there exists $Y \in \rel$, $\frac{Y}{v} \in \left [ f\left (\frac{X}{u} \right ) + a_2, f\left (\frac{X}{u} \right ) + b_2 \right ]\big \}$. 
		\STATE $(M_c,M_r, \Prec) \leftarrow$ Algorithm~\ref{algo:build_two_var} ($a_1, b_1, a_2, b_2, f, B_f, d, N_1, N_2, u, v, \rho_1, \rho_2$), \label{step:call_algo3}

               \STATE $M_{LLL} \leftarrow$ LLL-reduce the rows of $( M_c\augment M_r)$                \label{step:lll2D}
                
               \STATE $U\leftarrow  M_{LLL,r} M_r^{-1}$ // This is the LLL change of basis matrix; $M_{LLL,r}$ is the right part of the matrix $M_{LLL}$. Note that $M_r$ is diagonal. \label{step:changebasismat2D}

                \IF{$\max ( \normeucl{( M_{LLL}[0,j] )_{0\le j \le N+ N_1N_2 -1}} , \normeucl{( M_{LLL}[1,j] )_{0\le j \le N+ N_1N_2-1}}) \leq 2^{\Prec} /(N+N_1N_2)$} \label{step:cond_lll2D}

                \STATE $L_m \leftarrow [X_1^k X_2^\ell$ for $k = 0$ to $d-\ell$ for $\ell = 0$ to $d]$ \label{step:monomials2D}
                \COMMENT List of monomials, ordered in a way compatible with Algorithm~\ref{algo:build_two_var}, Steps~\ref{step:for_ell2D}--\ref{step:varphi2D}.
                \STATE $P_0 \leftarrow \sum_{j=0}^{N-1} U[0,j]L_m[j]$, $P_1 \leftarrow \sum_{j=0}^{N-1} U[1,j]L_m[j]$ \label{step:aux_poly2var}
                \IF{$\deg_{X_2}(P_0) = 0$ and $\deg_{X_2}(P_1) = 0$}
                    \STATE $R(X_1) \leftarrow \gcd(P_0, P_1)$
                \ELSE 
                    \STATE $R(X_1) \leftarrow \mathrm{Res}_{X_2}(P_0(X_1, X_2), P_1(X_1, X_2))$ 
                \ENDIF
                \IF{$R(X_1) \neq 0$}
                       \STATE $\mathcal{L} \leftarrow \{ t \in \rel; R(t) = 0 \}$ 
                    
                       \STATE  return $\mathcal{L}$

                    \ELSE

                       \STATE return ``FAIL'' \label{step:fail_coprime2D}
                    \ENDIF

                \ELSE
                    
                    \STATE return ``FAIL'' \label{step:fail_small2D}

                \ENDIF

	\end{algorithmic}
	\caption{Computation of the solutions to Problem~\ref{probgen}}
	\label{algo:2variables}
\end{algorithm}

We point out that the implementation of Algorithms~\ref{algo:build_two_var}
and~\ref{algo:2variables} requires considering several practical issues (see~\Cref{sssec:rounding2D}). Several optimizations
should also be considered, such as computing fewer DCTs (see~\Cref{sssec:eff2D})
or the use of Newton polynomials (see~\Cref{sssec:newton2D}) leading to denser
lattices, hence shorter vectors.

Finally, the input to \Cref{algo:build_two_var} requires an upper bound $B_f$ for $\Muniv(f)$. This requirement is discussed in \Cref{sssec:overst2D}. 
\textbf{Throughout the paper, our asymptotic success estimates
assume that the supplied bound satisfies $B_f = O(\Muniv(f))$,
with an implied constant independent of the varying parameters
and of the subintervals considered.}

\subsection{Proof of correctness}
\label{sssec:correct2D}
We shall now prove the correctness of Algorithm~\ref{algo:2variables}.

\subsubsection{Uniformly small polynomials in the vicinity of a transcendental analytic curve}

The following theorem is central in the proof of correctness of~Algorithm~\ref{algo:2variables}.
\begin{theorem}\label{thm:alternants_interpolants2D}
  Let $d \geq 1$ be an integer, $N= (d+1)(d+2)/2$, $u, v >0$.
  Let $\rho_1 \ge 2$, $\rho_2 \ge 2$, $a_1<b_1$, $a_2 < b_2$, $f$ be analytic in the neighborhood of $E_{\rho_1,a_1,b_1}$, $N_1, N_2 \ge 2$, and $N \leq N_1 N_2$.    Let $\Lambda = (\lambda_{k,\ell})_{0\leq k+\ell \leq d} 
 \in \Z^N$ be such that $\normeucl{\Lambda \hat{A}} \leq 1/(N+N_1N_2)$, and let $P(X_1,X_2) = \sum_{0\leq k+\ell \leq d} \lambda_{k,\ell} X_1^k X_2^\ell$, we have 
\[
   \max_{\substack{x\in [a_1, b_1], \\t\in [a_2, b_2]}} |P(u x, v (f(x) + t))| < 1.
   \]   
\end{theorem}
\begin{proof}
  We first turn the assumption on $\|\Lambda \hat{A}\|_2$ to a fact on
  $\|\Lambda A\|_1$. For $j =0, \ldots, N + N_1 N_2 -1$, we have
  \begin{align*}
    | (\Lambda A)[j] -  (\Lambda \hat{A})[j] | & \leq \sum_{0\leq k+\ell \leq d}  |\lambda_{k,\ell}| 2^{-\Prec}\\
    & \stackrel{(\ref{eq:a2tprec})}{\le} \sum_{0\leq k + \ell \leq d}     |\lambda_{k,\ell}| \min_i \hat{A}_2[i,i] \frac{2^{-2}}{N}\\
    & \le \frac{1}{4N} \normun{\Lambda\hat{A}_2}.
  \end{align*}
  Thus, we have
  \begin{align}
    \normun{\Lambda A} & \leq \normun{\Lambda \hat{A}} + (N+N_1N_2)  \frac{\normun{\Lambda\hat{A}_2}}{4N} \notag \\ 
& \le (N + N_1 N_2)^{1/2} \| \Lambda \hat{A}\|_2 + \frac{(N + N_1 N_2)}{4\sqrt{N}} \| \Lambda \hat{A}_2\|_2 \notag \\
    & \le \frac{1}{(N+N_1N_2)^{1/2}} + \frac{1}{4N^{1/2}} 
     < 1 \textrm{ for $N \ge 3$, $N_1, N_2 \ge 2$, } \label{ineq:un}
    \end{align}
thanks to the assumption  $\normeucl{\Lambda\hat{A}} \leq 1/(N+N_1N_2)$.
  
  Let now $P$ be as in the statement of the theorem, and
  \[
  Q(x, t) = \sum_{\substack{0 \le j_1 \le N_1 - 1, \\ 0 \le j_2 \le N_2 - 1}} q_{j_1,j_2} T_{j_1,[a_1,b_1]} (x) T_{j_2,[a_2,b_2]} (t)
  \]
  be the interpolation polynomial for $P(u x, v (f(x) + t))$ at the order $(N_1, N_2)$ pairs of Chebyshev nodes of the first kind. Then, the coordinates of $\Lambda A_1$ are exactly $q_{j_1,j_2}$, $0\leq j_1\leq N_1-1, 0\leq j_2\leq N_2-1$.

Proposition~\ref{prop:ineqcauchy2D} shows that
\begin{align*}
  \max_{\substack{x\in [a_1, b_1], \\t\in [a_2, b_2]}} | Q(x,t) & - P(ux, v(f(x)+t))| \leq  32\Mbiv(P(ux, v(f(x)+t)))\left( \frac{1}{\rho_1^{N_1}} + \frac{1}{\rho_2^{N_2}} \right)\\
  & \leq 64 \sum_{0\leq k+\ell\leq d} |\lambda_{k,\ell}| u^k v^\ell \Muniv(x)^k \Mbiv(f(x)+t)^\ell \left( \frac{1}{\rho_1^{N_1}} + \frac{1}{\rho_2^{N_2}} \right) \\
  & = \|\Lambda A_2 \|_1.
\end{align*}
As $\max_{x\in [a_i,b_i]} |T_{k,[a_i,b_i]}(x)| = 1$ for all $k$, we have
\[
\max_{\substack{x\in[a_1,b_1], \\ t\in [a_2, b_2]}} |Q(x,t)| \le \sum_{\substack{0\leq j_1\leq N_1 - 1, \\ 0\leq j_2\leq N_2-1}} |q_{j_1,j_2}|,
  \]
  which leads to 
\begin{align*}
  \max_{\substack{x\in[a_1,b_1], \\ t\in [a_2, b_2]}} |P(ux, v(f(x)+t))|  & \leq \max_{\substack{x\in[a_1,b_1],\\t\in[a_2, b_2]}}|Q(x,t)| + \|\Lambda A_2 \|_1 \\ 
   & \leq \sum_{\substack{0\leq j_1\leq N_1 - 1, \\ 0\leq j_2\leq N_2-1}} |q_{j_1,j_2}|+ \|\Lambda A_2\|_1\\
  & =  \normun{\Lambda A} < 1, 
   \end{align*}
 which concludes the proof.
\end{proof}

We deduce the following corollary.
\begin{corollary}\label{cor:main2D}
  Under the assumptions and notations of the previous theorem, for all
  $x\in[a_1, b_1], y$ such that $ux, vy \in \Z$, we have either $P(ux, vy) = 0$, or
  $y \not\in [f(x) + a_2, f(x) + b_2]$.
\end{corollary}

\begin{remark}\label{rem:reste_f2D} The proof of Theorem~\ref{thm:alternants_interpolants2D} yields in particular that
  \[
  1 > \normun{\Lambda A} \geq \sum_{0\leq k+\ell\leq d} |\lambda_{k,\ell}| \underbrace{64  u^k \Muniv(x)^kv^\ell \Mbiv(f(x)+t)^\ell  \left( \frac{1}{\rho_1^{N_1}} + \frac{1}{\rho_2^{N_2}} \right)}_{=: Q_{k,\ell}}.
  \]
 It follows that $\lambda_{k,\ell} = 0$ for any $k, \ell$ such that $ Q_{k,\ell} \geq 1 $.

 As the proof of Lemma~\ref{lem:prec2D} shows in particular that
 \[ R_{i} \ge    64 v \Mbiv(f(x)+t) \left( \frac{1}{\rho_1^{N_1}} + \frac{1}{\rho_2^{N_2}} \right) = R_{d+1}\] for all $i\geq d+1$, we see that if $R_{d+1} \ge 1$, we must have $R_{i} \ge 1$ for all $i \geq d+2$ so that $\lambda_{k,\ell} = 0$ for any $1 \leq \ell \leq d$, $0 \leq k \leq d - \ell$. In that case, the only functions taken into account are the $u^k x^k$'s and  the method only considers univariate polynomials and is likely to fail. 
\end{remark}  

\subsubsection{The auxiliary function $\psi$}
In order to state our theorems in a precise way, we need to introduce an auxiliary function $\psi$. This function appears in the study of integral points in regions ${\mathcal K}_{s} = \{ (x, y) / x + \gamma y \le s, x \ge 0, y \ge 0\}$,
in~\Cref{app:phipsi} where we undertake a  detailed study of $\psi$. 

\begin{definition}
Let $\phi$ be the bijection 
\begin{align*}
\phi: [1, +\infty) & \rightarrow [1, +\infty)\\
x & \mapsto (1 + \lfloor x \rfloor) (x - \lfloor x \rfloor /2). 
\end{align*}

For $\lambda \in [1, \infty)$, we define $\psi(\lambda)$ as
\[
  \psi(\lambda) = \frac{1 + \lfloor \phi^{-1}(\lambda) \rfloor}{12 \lambda}
  \left( 6\phi^{-1}(\lambda)^2 - \lfloor \phi^{-1}(\lambda) \rfloor - 2
  \lfloor \phi^{-1}(\lambda) \rfloor^2\right) \textrm{ for all } \lambda \geq 1.
\]
\end{definition}

In order to try to achieve a balance between optimality and intuition, we have introduced in~\Cref{eq:g} the function $g$ defined by
\begin{equation*}
  g : x \in [1,+\infty) \mapsto \frac{1}{3} \left ( \frac{2x}{1 + \left \lfloor  \sqrt{1/4+2x} -1/2 \right \rfloor}  + \left \lfloor  \sqrt{1/4+2x} -1/2 \right \rfloor \right ) - 1/6.
\end{equation*}

\begin{proposition}\label{prop:asymppsi}
  We have $\psi(x) \ge g(x)$ for $x\in [1, +\infty)$, $\psi(1) = g(1)$ and $\psi(\lambda) - g(\lambda) = O(\lambda^{-1/2})$ for $\lambda \rightarrow \infty$. As a consequence, $\psi(\lambda) = 2\sqrt{2\lambda}/3 + O_{+\infty}(1)$. 
\end{proposition}

Note that all the results stated in terms of $\psi$ in the sequel also
hold by replacing $\psi$ by $g$ in a slightly degraded version -- as
$\max_{x\ge 1} (\psi(x) - g(x)) \le (2-\sqrt{3})/6 \approx 0.045$, see~\Cref{app:phipsi}.

\subsubsection{Proof of success of Algorithm~\ref{algo:2variables}}
\label{sssec:success2D}

If we apply the LLL lattice basis reduction algorithm to $\hat{A}$, we obtain:
\begin{corollary}\label{cor:LLLapplies2D}
  Assume that $\det (\hat{A} \transp{\hat{A}})^{1/(2(N-1))} \leq \frac{2^{-N/4 - \Prec/(N-1)}}{N+N_1N_2}$, and let $U \in \mathrm{GL}_N(\Z)$ be such that $U\hat{A}$ is an LLL-reduced matrix. Then   Theorem~\ref{thm:alternants_interpolants2D} applies with $\Lambda$ any of the first two rows of $U$.
\end{corollary}
\begin{proof} 
  Let $\tilde{A}= 2^{\Prec} \hat{A} \in \mathcal{M}_{N, N+N_1N_2} (\rel)$. Let $U$ be a matrix such that $U\tilde{A}$ is LLL-reduced. 
  Theorem~\ref{thm:lll82} 
 then implies    that if $w_1$ and $w_2$ denote the first two rows of $U\hat{A}$, we have     
\[
\normeucl{2^{\Prec} w_i} \leq 2^{N/4} \max \left( \det ( \tilde{A} \transp{\tilde{A}})^{1/(2(N-1))}, \det (\tilde{A} \transp{\tilde{A}})^{1/(2N)}\right ), i = 1, 2.
\]
As $\tilde{A}$ is a rank-$N$ integer matrix, the determinant of $\tilde{A} \transp{\tilde{A}}$ is a positive integer, so that
$\det( \tilde{A} \transp{\tilde{A}})^{1/(2(N-1))}\ge \det (\tilde{A} \transp{\tilde{A}})^{1/(2N)}$; we thus have
\[
2^{\Prec}  \normeucl{w_i} \leq  2^{N/4}  2^{N\Prec/(N-1)} \det (\hat{A} \transp{\hat{A}})^{1/(2(N-1))}, 
\]
hence
\[
\normeucl{w_i} \leq  2^{N/4}2^{\Prec/(N-1)} \det (\hat{A} \transp{\hat{A}})^{1/(2(N-1))} \leq  \frac{1}{N+N_1N_2}.
\]
\end{proof}

In order to understand which precise role the size of the intervals plays in our algorithms, we now study the case $\rho_1 = K_1/(b_1-a_1)$ and similarly $\rho_2 = K_2/(b_2 - a_2)$, where $K_1 \ge 2(b_1 - a_1)$ and $K_2 \ge 2(b_2 - a_2)$ are fixed real numbers (note that $\rho_1, \rho_2 \geq 2$).  We further
assume $\rho_1^{N_1} \leq \rho_2^{N_2}$.

We introduce $K'_1 = K_1 + |a_1 + b_1|$, $K'_2 = K_2 + |a_2 + b_2|$, and note, for helping the intuition of the reader, that $K'_1, K'_2$ play the role of $B_x, B_t$ in Algorithm~\ref{algo:build_two_var} and thus roughly correspond to upper bounds on $x$, $t$.

For the sake of brevity, in the following proof we abbreviate $\Delta_{N,N_1, N_2, [a_1,b_1],[a_2, b_2], \rho_1, \rho_2} := (\det A \transp{A})^{1/2}$ (cf. Corollary~\ref{cor:bnd_with_alpha2D}) as $\Delta$.

\begin{proposition}\label{thm:1/(b-a)_2D}
  Let $f$ be analytic in a neighborhood of the closed disc $\cDun = \{ z \in \comp: |z -  (a_1+b_1)/2| \le  K_1/2 \}$, $d$ be an integer $\geq 2$, $N = (d+1)(d+2)/2$,  $\rho_1 = K_1/(b_1-a_1) \geq 2$, $\rho_2 = K_2/(b_2 - a_2) \geq 2$,    $\gamma = \log \rho_2 / \log \rho_1\in [3,N]$, $N_1 =  \lfloor \sqrt{2 \gamma N} \rfloor, N_2 = \lceil \sqrt{2N/\gamma} \rceil$ two integers. Let $M_{\cDun}(f) := \max_{z\in \cDun}|f(z)|$. Assume $\min(u\Muniv(x), v\Mbiv(f(x)+t)) \ge 1.$

Then, for $d\rightarrow \infty$, if 
\begin{equation}\label{eq:1/(b-a)_2D}
b_1-a_1 < K_1 2^{-O\left(\sqrt{N/\gamma}\right)}
\left( uvK'_1 \left(M_{\cDun} (f) + K'_2 \right)\right)^{-\frac{d}{3\psi (N/\gamma) \gamma} (1 + O(1/d))},
\end{equation}
we have $\Delta^{1/(N-1)}  < \frac{2^{-N/4 - \Prec/(N-1)}}{N + N_1N_2}$.
\end{proposition}
\begin{proof}
  Since $\rho_1 = K_1/(b_1-a_1) \geq 2$, we have ${E_{\rho_1, a_1, b_1}} \subset
  \cDun$. 
  Thanks to Corollary~\ref{cor:bnd_with_alpha2D},  we have
\[
    \Delta^{1/(N-1)} \le  2^{O(1)} N^{1/2+o(1)} 
 \frac{(uvK'_1)^{d/3 + O(1)}}{\rho_1^{\psi(N/\gamma) \gamma + O(1) }}  \left(M_{\cDun} (f) + K'_2\right)^{d/3 + O(1)}.
\]
Note that, using Lemma~\ref{lem:prec2D}, as $\rho_1^{N_1} \le \rho_2^{N_2}$,
\begin{multline*}
  2^{-\Prec}  \ge \frac{8}{N}\left( \rho_1^{-N_1} + \rho_2^{-N_2} \right) \ge \frac{16}{N} \rho_2^{-N_2} = \frac{16}{N} \rho_1^{-\gamma N_2} \\
   \ge 2^{-o(N)} \rho_1^{-O(N)}, \textrm{ as } \gamma N_2 < \sqrt{2N\gamma} + \gamma \le N(1 + \sqrt{2}).
\end{multline*}

Thus, for $\Delta^{1/(N-1)} < 2^{-N/4-\Prec/(N-1)} / (N + N_1N_2)$, it suffices that $\Delta^{1/(N-1)} < 2^{-O(N)} \rho_1^{-O(1)}$, or again that
\[
\rho_1  > 2^{O\left( \frac{N}{\psi(N/\gamma)\gamma}\right)}
 \left( uvK'_1 \left(M_{\cDun} (f) + K'_2\right)\right)^{\frac{d}{3\psi (N/\gamma) \gamma} (1 + O(1/d))}.
\]
\end{proof}

\begin{remark}\label{rem:upperbound}
Proposition~\ref{thm:1/(b-a)_2D} treats the case of exact bounds on $f$, while Algorithm~\ref{algo:build_two_var} is based on $B_f$. Under our assumption $B_f = O(\Muniv(f))$, using $B_f$
in the volume estimates introduces at most a factor $2^{O(d)}$
in the bound for
$(\det \hat{A}\transp{\hat{A}})^{1/(2(N-1))}$.
This factor is absorbed by the terms already allowed in the proof
of Proposition~\ref{thm:1/(b-a)_2D}, so that the same asymptotic
sufficient condition applies.
\end{remark}

\begin{corollary}\label{cor:1/(b-a)_2D}
  Under the assumptions of Proposition~\ref{thm:1/(b-a)_2D},  Algorithm~\ref{algo:2variables} over $[a_1, b_1]$ and  $[a_2, b_2]$ produces at Step~\ref{step:aux_poly2var} two polynomials $P_0$, $P_1$ such that
  \[
  \max_{x\in [a_1, b_1],\, t\in [a_2, b_2]} |P_i(ux, v(f(x)+t))| < 1 \textrm{ for } i\in \{0,1\}.
  \]
  In particular,  Algorithm~\ref{algo:2variables} never  executes Step~\ref{step:fail_small2D} and its output is valid.
\end{corollary}
\begin{proof}
By Proposition~\ref{thm:1/(b-a)_2D} and the preceding observation, the upper bound of
Corollary~\ref{cor:bnd_with_alpha2D} is less than
$2^{-N/4-\Prec/(N-1)}/(N+N_1N_2)$; indeed, the proof of that
proposition only uses this upper bound, not the exact value of~$\Delta$.
By Remark~\ref{rem:hatA2D}, the same upper bound holds for
$(\det \hat{A}\transp{\hat{A}})^{1/(2(N-1))}$, so that
$\det(\hat{A}\transp{\hat{A}})^{1/(2(N-1))}
  < \frac{2^{-N/4-\Prec/(N-1)}}{N+N_1N_2}$.
Corollary~\ref{cor:LLLapplies2D} then applies, and
Theorem~\ref{thm:alternants_interpolants2D} gives the result. 

\end{proof}

Note again that $P_0$ and $P_1$ may not be coprime, in which case the algorithm returns ``FAIL'' at Step~\ref{step:fail_coprime2D}.
This is what makes the algorithm heuristic. 

\subsection{Complexity analysis}\label{subsec:complexity2D}
In this subsection, we deduce estimates for the complexity of our algorithm applied to a fixed interval $[a, b]$. This actually requires several things:
\begin{itemize}
\item An evaluation of the complexity of the basic blocks, namely Algorithms~\ref{algo:build_two_var} and~\ref{algo:2variables}, done in \Cref{sssec:basicplx}. 
\item The use of Corollary~\ref{cor:1/(b-a)_2D} to evaluate the size of a subinterval   $[a_1, b_1]$ which can be treated at once by those algorithms. This is the subject of \Cref{sssec:numint}, building on this corollary.
\end{itemize}

\subsubsection{Complexity of Algorithms~\ref{algo:build_two_var} and~\ref{algo:2variables}}\label{sssec:basicplx}

We let $\Mult(n)$ denote the complexity of multiplying two $n$-bit integers (or two precision $n$ floating-point numbers). Recall that  $\Mult(n) = O(n \log n) = \tilde{O}(n)$, see~\cite{HvdH2021}, but the proofs keep the dependency on $\Mult(n)$ explicit. 
Similarly, the Propositions assume typical choices of parameters for our algorithm, namely $N_1 N_2 = O(d^2)$ and $d^2 = O({\InterPrec})$, where ${\InterPrec}$ is the working precision defined in Proposition~\ref{prop:cplxalgo1}, but the general case can be found in the proofs in \Cref{app:compalgo}.
\begin{proposition}\label{prop:cplxalgo1}
  Let $\cM = \max (|a_1|, |b_1|, |a_2|, |b_2|, \max_{x \in [a_1,b_1]} |f'(x)|,$ $ u, $ $ v,$   $ \rho_1,$  $ \rho_2, B_f)$ and ${\InterPrec} = \max(N_1 \log \rho_1, N_2 \log \rho_2) + O(d \log \cM)$. Assume that $N_1 N_2 = O(d^2)$ and $d^2 = O({\InterPrec})$, that   evaluating $f$ in precision ${\InterPrec}$ costs   $C_{f,{\InterPrec}} = \tilde{O}({\InterPrec})$, and that a DCT of size   $n$ in precision $q$ costs $O(n^2 \Mult(q))$. Then, on input   $a_1, b_1, a_2, b_2, f, B_f, d, N_1,$ $N_2,$ $ u,$ $ v, \rho_1, \rho_2$, Algorithm~\ref{algo:build_two_var} has complexity $\tilde{O}(d^6 {\InterPrec})$.
\end{proposition}
\begin{proof}
  See \Cref{app:proofalgo1}.
\end{proof}

We now turn to the analysis of Algorithm~\ref{algo:2variables}. We
shall limit ourselves to the analysis of
Steps~\ref{step:call_algo3}--\ref{step:aux_poly2var}, which compute
the two auxiliary polynomials. This is, in any case, the core of the
algorithm, but also the choice made in previous papers, and thus
allows for a better comparison.
\begin{proposition}\label{prop:cplxalgo2}
  On input $a_1, b_1, a_2, b_2, f, B_f, d, N_1, N_2, u, v, \rho_1, \rho_2$, under the  assumptions 
  $C_{f,{\InterPrec}} = \tO({\InterPrec})$,  $N_1 N_2 = O(d^2)$, $d^2 = O({\InterPrec})$, Steps~\ref{step:call_algo3}--\ref{step:aux_poly2var} of Algorithm~\ref{algo:2variables} have complexity
  $\tilde{O}(d^8{\InterPrec}^2)$. 
\end{proposition}
\begin{proof}
  See \Cref{app:proofalgo2}.
\end{proof}    

\subsubsection{Number of subintervals}\label{sssec:numint}
Thanks to the results of the previous subsection, given a value
$\gamma$, we can estimate the maximum size of an interval $[a_1, b_1]
\subset [a, b]$, with $a, b$ fixed, for which
Algorithm~\ref{algo:2variables} succeeds (in the sense of
Corollary~\ref{cor:1/(b-a)_2D}) and yields an upper bound of the
order of magnitude $w = O(|b_1-a_1|^{-\gamma})$. 

This follows from Proposition~\ref{thm:1/(b-a)_2D}, and yields at the same
time the number of subintervals to be considered if one wants to deal
with a full interval $[a, b]$.

\begin{theorem}\label{thm:cplx2D}
  Given fixed $f$ and two fixed real numbers $a < b$, Problem~\ref{probgen}  can heuristically be solved for $u, v \rightarrow \infty$, $d\rightarrow \infty$, $\gamma \in [3,N]$, and
  \[
    w = 2^{O({\sqrt{\gamma N}})}(uv)^{\frac{d}{3\psi(N/\gamma)}(1 +O(1/d))}
  \]
  over $[a, b]$ using
  \[
    (b-a) 2^{O\left ( \sqrt{N/\gamma} \right)} (uv)^{\frac{d}{3\psi(N/\gamma)\gamma}(1+O(1/d))}
    \]
  calls to Algorithm~\ref{algo:2variables} with parameter $d$.
\end{theorem} 
\begin{proof}
Under our assumptions, in Corollary~\ref{cor:1/(b-a)_2D},    we note that $a_1, a_2, b_1, b_2$ are bounded, and we choose $K_1$ to be a constant such that $f$ is analytic in $\cDun$. This constant exists in view of the assumptions on $f$ in Problem~\ref{probgen}. We choose $K_2 = 1$  and $\rho_2 = \rho_1^{\gamma}$, and obtain the sufficient heuristic success condition :
\[
b_1 - a_1 \le 2^{-O\left ( \sqrt{N/\gamma} \right)}(uv)^{-\frac{d}{3\psi(N/\gamma)\gamma}(1+O(1/d))};
\]
we assume equality in the sequel.

  The heuristic nature of this result comes from the possibility that   the two polynomials obtained in Algorithm~\ref{algo:2variables} are   not coprime, in which case one cannot recover the solutions $X, Y$   from those two polynomials.

  Finally, we  take $1/w = (b_2 - a_2)/2$, thus the upper bound on $w$ is $O((b_1 - a_1)^{-\gamma})$. The number of calls to Algorithm~\ref{algo:2variables} is  $\lceil (b-a)/(b_1-a_1) \rceil$.
\end{proof}
The careful reader might have noticed powers of two appearing here
compared to the simplified version of the introduction. These powers
of two are related to the approximation factor of the LLL algorithm. In most applications, these terms are asymptotically negligible, for instance as soon as $d = o(\log(uv))$.

As explained in Section~\ref{sec:introduction}, a large $\gamma$ means that we only obtain information on points very close to the curve $y = f(x)$ but that we can handle large intervals at once. In an opposite direction, small $\gamma$ yields information in a wider region around the curve, but requires more computations. 

In terms of asymptotic regimes, 
\begin{itemize}
    \item 
if $\gamma$ tends to infinity as $N/\kappa$, $\kappa \ge 1$, we deduce as an upper bound for the number of intervals
  $O\left((b-a) (uv)^{\frac{2\kappa}{3d\psi(\kappa)}(1 + O(1/d))}\right)$,
    with $w = 2^{O(d^2)} (uv)^{\frac{d}{3\psi(\kappa)}(1 + O(1/d))}$. For $\kappa = 1$, this is essentially Bombieri--Pila's estimate~\cite{BombieriPila1989}.
\item if, on the other hand, $\gamma = o(N)$, $N/\gamma$ tends to infinity and we can use the asymptotic estimate (see Corollary~\ref{cor:estim_sumKs}) $\psi(N/\gamma) = \frac{2}{3} \sqrt{\frac{2N}{\gamma}} + O(1)$ to get a bound on the number $n_I$ of intervals and on $w$ of the respective forms
    \begin{align*}
    n_I & = (b-a) 2^{O(d/\sqrt{\gamma})} (uv)^{1/(2\sqrt{\gamma})(1 + O(\sqrt{\gamma}/d))},\\
    w & = 2^{O(d\sqrt{\gamma})} (uv)^{\sqrt{\gamma}/2 (1 + O(\sqrt{\gamma}/d))}.
    \end{align*}
\end{itemize}
For $u = v = 2^p$, $\gamma = 4 + o(1)$, $d = o(p)$, we recover, and slightly refine, Stehlé's result~\cite[Cor.~4]{Stehle2006}: we can solve the TMD (i.e., get the bound $1/w = 2^{-2p}$) in time $2^{(p/2)(1+o(1))}$.

\subsubsection{Taking into account the growth of $f$}
The previous analysis ignores the dependency on $f$ by choosing $\rho_1(b_1-a_1)$ as a constant and considering, then, the term $M_{\cDun} (f)$ as a constant.

In this subsection, we shall slightly improve on this analysis if $f$
is entire and the growth of $f$ at infinity is moderate, by letting
$\rho_1(b_1 - a_1)$ tend to infinity with $d$.

We start by reformulating Proposition~\ref{thm:1/(b-a)_2D} in the
present context; the following statement clearly shows that the region
of validity of the method is controlled by the relationship between
$K_1$ and $M_{\cDun}(f)$, namely the growth of $f$ at $\infty$.

In the sequel, if $H:\R_{>0} \rightarrow \R_{>0}$ we denote
by $\omega(H(d))$ any function $\tilde{H}: \R_{>0} \rightarrow \R_{>0}$ such that
$\tilde{H}(d)/H(d) \rightarrow_{d\rightarrow +\infty} +\infty$.

\begin{proposition}
Under the assumptions of Proposition~\ref{thm:1/(b-a)_2D}, assume that $a_1 < b_1$ are fixed, that $(a_2 + b_2)/2$ is fixed, and let $\gamma = N/\lambda$ for some
  constant $\lambda \ge 1$. We assume that $\omega(1) \le K_1 = 2^{O(d)}$, $K_2 = O(1)$ and that $f$ is a nonconstant entire function.
  Then, a sufficient condition for the method to succeed over $[a_1, b_1]$
  when $d \rightarrow \infty$ is  
  \begin{equation}\label{eq:asymp_cond}
    K_1 (M_{\cDun}(f)uv)^{-\frac{2\lambda}{3d\psi(\lambda)}} = 2^{\omega(1)}.
  \end{equation}
\end{proposition}
\begin{proof}
  Under our assumptions, we have $K'_1 = K_1(1 + o(1))$, $M_{\cDun}(f) = \omega(1)$, and $K'_2 = K_2 + O(1) = O(1)$. We have $\rho_1 = K_1/(b_1 - a_1) \ge 2$ for $d$ large enough, $\rho_2 = \rho_1^\gamma \ge 2$ for $d$ large enough,
so that
  the right hand side of~(\ref{eq:1/(b-a)_2D}) is 
  \begin{multline*}
  K_1 2^{O(1)}
  \left( uv \cdot K_1(1 + o(1))\left(M_{\cDun}(f)(1 + o(1))\right)\right)^{-d/(3\gamma\psi(\lambda)) (1 + O(1/d))}\\ = K_1^{-O(1/d)} 2^{O(1)} 2^{\omega(1)}.
  \end{multline*}
  
  As we have assumed $K_1 = 2^{O(d)}$, the first term can be merged   into the second one. Overall, this quantity tends to $\infty$ as $d\rightarrow \infty$. 
  
  The method succeeds as soon as this upper bound is larger than $|b_1 - a_1|$;   this is guaranteed to happen for $d$ large enough. 
\end{proof}

Depending on the growth of $f$ at infinity it is now
quite easy to produce a condition on $|b_2-a_2|/|b_1-a_1|^\gamma$,
$u$, $v$ under which the method works: one just has to find an
optimal choice for $K_1$.

The conditions on $K_1$ and $K_2$ are chosen to obtain a statement which is simple and works for practical functions; it is sufficient  unless the function tends really very slowly to infinity. Proposition~\ref{thm:1/(b-a)_2D} still allows us to handle the general case. 

We say that an entire function $f : \comp \rightarrow \comp$ has finite order $\le \theta$ if 
\[
\limsup_{\rho\rightarrow +\infty} \frac{\log \log \max_{|z|=\rho} |f(z)|}{\log \rho} \le \theta.
\]

\begin{theorem}\label{thm:ordre_fini} 
    Let $f$ be an entire function of finite order $\le \theta$, where $\theta > 0$. Let 
    $\lambda \ge 1$ be a constant and
    $\gamma = N/\lambda$, and assume that  $a_1 < b_1$, $(a_2+b_2)/2$ are fixed. Let 
  $\nu > \frac{2\lambda \theta}{3\psi(\lambda)}$ be a real number; then
  the method succeeds over $[a_1, b_1]$ for $uv \rightarrow \infty$ as soon as
  \[
  d = \nu \frac{\log(uv)}{\log \log (uv)} (1 + o(1)), 
  \]
  and
  \[
  |b_2 - a_2| \le (uv)^{-\frac{\nu^2}{2\theta \lambda} \frac{\log(uv)}{\log \log (uv)}(1 + o(1))}.
  \]
\end{theorem}
\begin{proof}
  We take $\theta' > \theta$ to be chosen later on,
  and $K_1 = d^{1/\theta'}$; then $M_{\cDun}(f) = 2^{O(d)}$ and 
  the condition of (\ref{eq:asymp_cond})  rewrites as 
  \[
   d \log d = \frac{2\theta'\lambda}{3\psi(\lambda)} \log (uv) + \omega(d). 
   \]
   Let $\nu> \frac{2\lambda \theta}{3\psi(\lambda)}$; we can always choose $\theta'$ so that $\nu > \frac{2\lambda \theta'}{3\psi(\lambda)}$. Then,
   the previous condition is realized for $d = \nu \frac{\log(uv)}{\log \log (uv)} (1 + o(1))$, for $d$ large enough.

   Finally, from $K_2 = |b_2 - a_2| (K_1 / |b_1 - a_1|)^\gamma$, $|b_2 - a_2| = O(d^{-\gamma/\theta})$ implies 
   \[
   K_2 = O\left(\exp\left(-\gamma (\theta^{-1}-\theta'^{-1})\log d - \gamma \log |b_1 - a_1|\right)\right) = O(1), 
   \]
   as required. The result
   then follows, as $\gamma = d^2/(2\lambda)(1+o(1))$. 
\end{proof}

As a consequence for Problem~\ref{probdir}, we can (heuristically) check in polynomial time that $\mu(p) = O(p^2/\log p)$ for a function such as $\exp$. This improves on~\cite{Stehle2006} by a factor $\log p$.

Similar results can be obtained for various types of behavior of $f$
at infinity; overall, slower growth leads to improved results. 

\section{Prereduction} \label{sec:prered}

Most of the execution time of our algorithm consists, in practice, in
calls to the LLL algorithm on lattices built from values of the
function on consecutive intervals. In this section, we explain how to
significantly reduce  those costs.

Let us consider two consecutive intervals, $[a_1, b_1]$ and $[b_1, c_1
  = 2b_1-a_1]$. Roughly speaking, the purpose of the LLL step
of~\Cref{algo:2variables} is to compute a unimodular matrix $U$ such
that $UM$ is ``small''.  The output of this algorithm is actually more
than one or two basis vectors; it is a full matrix $U$ such that $UM$
has ``small'' coefficients.

As the matrices $M$ and $M'$ corresponding to two consecutive intervals are built using values of $f$ (and its derivatives in the case of Stehlé's approach) over those intervals, $M$ and $M'$ are close to one another. As $UM$ is small, we might thus expect $UM'$ to be small too; if this holds, the cost of LLL-reduction applied to $UM'$ will be much lower than LLL-reduction applied to $M'$.
In this section, we give theoretical and experimental arguments that validate this intuition: our experiments (see~Table~\ref{tab:prered_exper}) show that the prereduction trick works very well for our algorithm and we outline in~\Cref{subsec:anatrick} a heuristic justification of that fact.

\def\row{\textrm{row}}

\subsection{Analysis of the prereduction trick}\label{subsec:anatrick}
The discussion that follows is somewhat technical. For the sake of
simplicity, up to a linear transform of each coordinate, we shall now
restrict ourselves to the case $[a_1, b_1] = [a_2, b_2] = [-1, 1]$ so
that prereduction applies to $[1, 3] \times [-1, 1]$ the matrix found
for the product of intervals $[-1, 1]\times[-1, 1]$. 
We modify the construction of $M_r$ so that the remainder matrix is
the same on both intervals---by taking for each coefficient the
largest value among the two intervals. We finally make the mild
assumption $\rho_1 \ge 8$.

The following lemma explains the algebraic meaning of the prereduction idea, by mathematically formulating what the matrix $UM'$ represents. It is a mere reformulation of the definitions.
For simplicity, we ignore the scaling by $2^{\mathrm{tprec}}$ and the rounding of matrix entries. Thus, $M_c$ and $M_r$ denote the real matrices obtained before Step \ref{step:mc2D} and without the floor at Step~\ref{step:mr2D} of  Algorithm~\ref{algo:build_two_var}.
\begin{lemma}\label{lem:prered}
Let $M_c$ (resp. $M'_c$) $\in \mathcal{M}_{N, N_1N_2}(\R)$ be the matrix computed by
Algorithm~\ref{algo:build_two_var}, with truncation removed at steps \ref{step:mc2D} and \ref{step:mr2D}, on input $a_1, b_1, a_2, b_2$ 
(resp. $b_1$, $2b_1 - a_1$, $a_2$ , $b_2$), with $a_1 < b_1, a_2 < b_2$, 
  and let $\mathfrak{w} = (\mathfrak{w}_{k\ell})_{0 \le k+\ell \le d}$ be a vector in $\R^{N}$.

Then, $\mathfrak{w}\cdot M_c$ (resp. $\mathfrak{w}\cdot M'_c$) is a vector containing the
coefficients (in the Chebyshev basis) of the Chebyshev interpolant of
order $(N_1-1, N_2-1)$ over $[a_1, b_1] \times [a_2, b_2]$ (resp. $[b_1,
  2b_1 - a_1] \times [a_2, b_2]$) of the function $W(ux, v(f(x)+t))$,
where $W = \sum_{0\le k+\ell \le d} \mathfrak{w}_{k\ell} x^k y^\ell$ is the polynomial
whose coefficients are the coordinates of $\mathfrak{w}$.

Similarly, the absolute values of the coordinates of $\mathfrak{w}\cdot M_r$ and $\mathfrak{w}\cdot M'_r$ are upper bounds for the approximation error of $\mathfrak{w}_{k,\ell} f_{k,\ell}$, $0\le k+\ell\le d$, by its Chebyshev interpolants over the same intervals. 
\end{lemma}

Let $\mathfrak{w} \in \R^N$ be a vector such that $\|\mathfrak{w} \cdot (M_c \, | \, M_r)\|_1 < 1$. In
case of success of Algorithm~\ref{algo:2variables}, there are at least two such
$\mathfrak{w}$ among the rows of the change of basis matrix used by the LLL call of
Step~\ref{step:lll2D}. To explain why prereduction works, we need to argue that
whenever $\mathfrak{w}\cdot M_c$ is small, $\mathfrak{w} \cdot M'_c$ is much smaller than
$M'_c$. As we have forced $M_r = M'_r$, we need not take this part
into account.

We notice that the coefficients of the Chebyshev interpolant of order
$(N_1-1, N_2-1)$ of a function $F$ over $[a_1, b_1] \times [a_2, b_2]$
are obtained from values of $F$ over $[a_1, b_1] \times [a_2, b_2]$ by
a DCT, which is an orthogonal (up to diagonal normalization) linear transform. In view of this and
Lemma~\ref{lem:prered}, in order to bound $\|\mathfrak{w} \cdot M'_c\|_2$ it
suffices to bound the values of the function $W(ux, v(f(x)+t))$ over
$[1, 3]\times [-1, 1]$.

Let $Q_{N_1-1, N_2-1}(x,t)$ be the Chebyshev interpolant of orders $N_1-1, N_2-1$ of $W(ux, v(f(x)+t))$ over $[-1,1] \times [-1, 1]$. We then have, for any $x \in [1, 3]$, $t\in [-1, 1]$, assuming $\rho_1\geq 8$ and $\rho_2 \ge 2$, thanks to Lemma~\ref{le:cheb_extrapol}:
\begin{multline*}
\left|W(ux,v(f(x)+t)) - Q_{N_1 - 1, N_2 - 1}(x, t)
\right|  \le \\
8 \, \Mbiv (W(ux, v(f(x)+t)))  \left(
\left(\frac{3 + 2\sqrt{2}}{\rho_1}\right)^{N_1} +
\frac{1}{\rho_2^{N_2}} \right) \frac{\rho_1 \rho_2}{(\rho_1 - (3+2\sqrt{2}))(\rho_2-1)}. 
\end{multline*}

Now, success over $[-1,1]\times [-1,1]$ implies that
\[
64\, \Mbiv (W(ux, v(f(x)+t))) \left(\frac{1}{\rho_1^{N_1}} + \frac{1}{\rho_2^{N_2}}
  \right) < 1, 
  \]
  so that we have, for $(x, t)\in [1, 3] \times [-1, 1]$ 
  \[
    \left|W(ux,v(f(x)+t)) - Q_{N_1 - 1, N_2 - 1}(x, t)\right| \le (3+2\sqrt{2})^{N_1}.
  \]

Finally, recall that $Q_{N_1 - 1, N_2-1}(x,t) = \sum a_{k_1, k_2} T_{k_1}(x) T_{k_2}(t)$
with $\sum |a_{k_1, k_2}| \le 1$. This shows that, for $x \in [1, 3]$, $t\in [-1, 1]$, 
\[
|Q_{N_1 - 1, N_2 - 1}(x, t)| \le (|x| + \sqrt{x^2-1})^{N_1 - 1}. 
\]

Overall, we find that, for $x \in [1, 3]$, $t \in [-1, 1]$, 
\[
\left|W(ux, v(f(x) + t))\right|
 \le 7\cdot (3+2\sqrt{2})^{N_1-1}. 
\]

This suggests that prereduction starts with a matrix with coefficients
of size $(uv)^{O(d)}$ and turns it, at the cost of a matrix-matrix product, into a matrix with coefficients of order $2^{O(N_1)}$. 

The heuristic analysis shows that, intimately, the success of the prereduction
technique is closely related to the design of our algorithm, in
particular the fact that we approximate the function itself rather
than replacing it by a polynomial at an early stage. Both this
heuristic study and the experiments that follow show that
prereduction is an important optimization and constitutes a very
significant practical improvement over the state of the art.

\subsection{Experimental support}
We have experimented with the exponential function, for $x \approx 1/4$.
Experiments are reported in Table~\ref{tab:prered_exper}.

Table~\ref{tab:prered_exper} shows an excellent qualitative accordance
between the output of Algorithm~\ref{algo:2variables} and the
heuristics presented in this section.

\begin{remark}
For the SLZ algorithm, our experiments showed that the
prereduction idea slightly degrades, instead of improving, the input
of the LLL step. 
\end{remark}

\begin{table}
\begin{center}
\begin{tabular}{|c|c|c|c|c|} 
  \hline
  \hspace*{0cm} $(d, p, N_1)$ & $(6,  113, 28)$ &  $(8,  113, 39)$ &   $(10, 113, 56)$ &
  $(12, 113, 80)$\\
  \hline
  $\log_2(\|M\|_\infty)$ & 668 & 889 & 1109 & 1328\\
  \hline
  $pd$ & 678 & 904 & 1130 & 1356\\
  \hline
  \hline
  $\log_2(\|UM'\|_\infty)$ & 50 & 83 & 120 & 174 \\
  \hline
  $(N_1-1) \log_2(3+2\sqrt{2})$ & 69 & 97 & 140 & 201\\
  \hline
\end{tabular}
\end{center}
\caption{Effects of the prereduction for our method for the $\exp$ function near $1/4$, and comparison with the behavior predicted by the heuristic.}\label{tab:prered_exper}
\end{table}

\section{Comparison with previous work}\label{sec:comparison}

Our method bears a strong resemblance to Stehlé's
work~\cite{Stehle2006}. From an algorithmic point of view, our approach
mostly differs by the use of more intrinsic techniques.

Firstly, Stehlé reduces Problem~\ref{probgen} to the case of a polynomial $P$. In
order to do this, he introduces $P$ such that $|P(x) - f(x)| \le
\varepsilon$, for $x\in [a, b]$, and replaces $1/w$ by $1/w + \varepsilon$; the
solutions to this modified problem contain the solutions to the initial
one. Compared to this, we work with the function all along. To this end,
following~\cite{BriJol10}, we introduce an explicit representation of
functions in the form of a polynomial approximation plus a remainder.

Secondly, Stehlé uses Taylor approximation for computing $P$; we
rather use the sharper Chebyshev interpolation.

As already discussed, the first point is key to the success of the
prereduction trick; it also allows us to use a dense representation of
our auxiliary polynomials, which leads us to manipulate almost square
matrices, whereas Stehlé's matrices are inherently more rectangular,
because he has to represent all coefficients of the somewhat sparse
bivariate polynomials he manipulates.

Our analysis is also more precise in the case of large $w$, which,
for large $uv$, is needed for the problem to be tractable in practice. 
Table~{\ref{tab:comparison}} shows that if one wants $w$ to be of the
order of $(uv)^{c\cdot d}$ for $c < 1$, 
our analysis is sharper than Stehlé's. 
Going further in this direction,
Stehlé proved that Problem~\ref{probgen} could be solved
in polynomial time for $w \ge (uv)^{O(\log(uv))}$ (a similar result is
obtained by purely mathematical means, for the exponential function,
by Nesterenko and Waldschmidt \cite{NW95}), a bound which we improve 
to $w \ge (uv)^{O(\log(uv)/\log\log(uv))}$ for moderate-growth entire functions
(see Theorem~\ref{thm:ordre_fini}) by using complex-analytic techniques.

\begin{table}[htbp]
\begin{center}
\begin{tabular}{|l|c|c|c|c|c|} 
\hline
& References   & Parameters   & Matrix       & Complexity & Bound\\ 
&              & to be chosen & dimensions   &  &       on $w$ \\
\hline
BH$_{d\rightarrow \infty}$ & Thm.~\ref{thm:cplx2D}& $\gamma = \frac{(d+1)(d+2)}{2\kappa}$,  &  $\approx d^2/2 \times \theta d^2/2$& $\frac{2\kappa}{3d \psi(\kappa)}$ & $\frac{d}{3\psi(\kappa)}$ \\
  & +Prop.~\ref{prop:asymppsi}  & $\kappa \ge 1$ & $\theta = \sqrt{\frac{2}{\kappa}}\lceil \sqrt{2\kappa}\rceil + 1$& &\\
\hline
  BH$_{d\rightarrow \infty}$ & Thm.~\ref{thm:cplx2D} & $\gamma = o(d^2)$ &  $\approx d^2/2 \times 3d^2/2$ & $\frac{1}{2\sqrt{\gamma}}$ & $\sqrt{\gamma}/2$ \\
    & +Prop.~\ref{prop:asymppsi}  & & & &\\
\hline
\hline
  S$_{d \rightarrow \infty}$ &  \cite[Thm. 3]{Stehle2006} & $1 \le \xi = o(d^2)$ & $d^2/2 \times O(\xi d^2)$ & $\frac{1}{2\sqrt{\xi}}$ & $\sqrt{\xi}/2$\\
  \hline
\end{tabular}
\end{center}
\caption{Comparison of the main methods of this paper and~\protect\cite{Stehle2006}. For the last two columns, the value given is the exponent of $uv$ in the complexity or the bound (up to the $(1+o(1))$ which has been omitted here). For an easier comparison, we introduce $\xi = (n_1+n_2)^2/(4(n_1 - t)^2)$ (in Stehlé's notation) and use $d$ for Stehlé's $\alpha$.  }
  \label{tab:comparison}
\end{table}

Table~\ref{tab:comparison} only estimates the exponential part of the
complexity, while the polynomial part, which is dominated by the cost
of lattice basis reduction, also plays an important practical role.
Akhavi--Stehlé's random projection trick somewhat reduces the influence
of the dimension (mentioned in Table~\ref{tab:comparison}), but the
size of the integers involved in the two methods differs. Roughly
speaking, and ignoring the dependency on $f, a, b$ which is similar
for the two methods:
  \begin{itemize}
  \item in the case of our method, the size of the integers
  involved is $\approx d\log (\max(u,v)) + \Prec \approx d \log \max(u, v) + \log \min(\rho_1^{N_1}, \rho_2^{N_2})$; we expect,
    for optimal choices of parameters, that $N_1 \log \rho_1$ and $N_2 \log \rho_2 \approx (d/\sqrt{\gamma}) \cdot O(\sqrt{\gamma} \cdot \log (uv))$
    have the same order of magnitude. This size is thus
    \[
    O(d \left( \log \max(u, v)\right));
    \]
  \item in Stehlé's paper\footnote{The notations $\alpha_{\textrm{Ste}}$, $d_{\textrm{Ste}}$ and $t_{\textrm{Ste}}$ refer to the parameters $\alpha, d$ and $t$ from~\cite{Stehle2006}.}, the integers involved are of the order of
    $(M N_2N_1^{d_{\textrm{Ste}}})^{\alpha_{\textrm{Ste}}}$, which, in our notations, means that their size is of the order of $O(d(\log \gamma + \log \max(u, v)))$. 
  \end{itemize}

We may thus also have a significant improvement on the size of the integers
involved when $\gamma$ is large.

\section{Experimental results}\label{sec:expresults}

We have implemented TMD-oriented (cf. Problem~\ref{probdir}) versions of our algorithms in C. Our codes are  available from~\url{https://hal.science/hal-04474530/file/LACoR.zip}. The tests hereafter were executed on an Intel Xeon E5620 2.40GHz CPU, using a single thread.

In the example that we address, we cut the binades under consideration into subintervals of the same size and we apply the algorithms to each subinterval.
 The subinterval will correspond to the interval $[a_1,b_1]$, while $[a_2,b_2]$ will be equal to $[-1/w,1/w]$, cf. Problem~\ref{probgen}.

Currently, the most expensive part of our algorithms is the LLL
reduction. In our implementation, we used a random projection trick
inspired by~\cite{AkhaviStehle2008},
cf. Theorem~\ref{thm:AkhaviStehle}; this trick had a significant
impact on our results. 
 
Another optimization comes from the use of Newton polynomials instead
of monomials (as pointed out in \Cref{sssec:newton2D}): for given
values of $d, N, N_1, N_2$, it makes it possible to process larger
subintervals.

The timings and the values of $\log_2(w)$ presented with a $^*$ are estimated ones: we performed our computations on a subinterval and then extrapolated the timing to address the whole binade, and the corresponding value of $\log_2(w)$.

We  chose to limit the evaluation of our algorithms to feasible computations in   binary128, namely computations that could be performed in real   life, possibly using a large cluster. In terms of the bound on $w$,   we have thus excluded the optimal case of the TMD, namely $w \approx   2^{2p}$, and have started at $w \approx 2^{4p}$.

\subsection{Algorithms~\ref{algo:build_two_var} and~\ref{algo:2variables} in action:  the TMD for the exponential function in binary128} \label{subsec:exper_algo3and4}

We used our implementation of Algorithms~\ref{algo:build_two_var}
and~\ref{algo:2variables} to address the TMD for the exponential
function over $[1/4, 1/2)$, for directed rounding functions and for
  the precision $p = 113$. 
More precisely, we address the following question, for various
values of the parameter $w$: determine the integers $X$, $1/4 \le
X/2^{p+1} < 1/2$ for which there exists $Y \in \rel$ satisfying
\[
\left | \exp\left( \frac{X}{2^{p+1}} \right) - \frac{Y}{2^{p-1}} \right | < \frac{1}{w}.
\]

We mention that, for $p = 113$, \cite{KheVou2011} implies\footnote{For the choice $(K,L,E) = (84,59,109.44...)$} that there is no solution for $w = 2^{297.11p}$; to the best of our knowledge, this is the best theoretical result on this question.

We report in Table~\ref{tab:exp} our results. 
We first set $p = 113, u = 2^{p+1}, v = 2^{p-1}, b_2 = - a_2 = 1/w$. 
The choice of the parameters $d, N_1, N_2, \rho_1, b_1-a_1$ is then
made in order to minimize the time for treating the whole binade $[1/4, 1/2)$.
We finally fix $\rho_2 = \min(\rho_1^{N_1/N_2},1/b_2)$.

\begin{table}[htbp]
\begin{center}
\begin{tabular}{|c|c|c|c|c|c|c|c|c|} 
\hline
$\log_2(w)$ & $d$  & $N$ & $N_1$ & $N_2$ & $\rho_1$ & $b_1-a_1$ & Timing & \% LLL \\
\hline
$4p^{*}$ & 6 & 28 & 20 & 3 & $ 2^{39.8}$ & $2^{-36.75}$ & $193^{*}$ years & 78\%\\
\hline
$6p^{*}$ &  8 & 45  & 45  & 2 & $ 2^{28.4}$ & $2^{-26.55}$ & $1.56^{*}$ years & 82\% \\
\hline
$8p$ & 9 & 55  & 50  & 2 & $2^{24.3}$ & $2^{-23.55}$ & 36.3 days & 89\%\\
\hline
$10p$& 10 & 66  & 70  & 2 & $2^{25}$ & $7/2^{23}$ & 11.5 days & 89\% \\
\hline
$12p$ & 12 & 91 & 80  & 2 & $2^{21}$ & $5/2^{19}$ &  5.0 days & 96\% \\
\hline
\end{tabular}
\end{center}
\caption{Algo.~\ref{algo:build_two_var} and~\ref{algo:2variables}: $\exp$ over the binade $[1/4,1/2)$}\label{tab:exp}
\end{table}

\subsection{Comparison with BaCSeL}
We have compared our results with BaCSeL-4.0\footnote{\url{https://gitlab.inria.fr/zimmerma/bacsel},
this version integrates our port of Akhavi--Stehlé's trick.}, which
implements~\cite{Stehle2006}.  For the comparison to be fair, we have
included the Akhavi--Stehlé trick
(cf. Theorem~\ref{thm:AkhaviStehle}) in BaCSeL. We have also tried to
include the prereduction trick from \Cref{sec:prered} but the latter, in the setting
of~\cite{Stehle2006}, seems to make the reduction more costly. In this
case, the complete timings are merely estimates for the cost of
treating a whole binade.

The results are reported in Table~\ref{tab:exp_bacsel}. 
\begin{table}[htbp]
\begin{center}
\begin{tabular}{|c|c|c|c|c|c|} 
\hline
$\log_2(w)$ & $\alpha_{\textrm{Ste}}$  & {$d_{\textrm{Ste}}$} & {$t_{\textrm{Ste}}$} &
Timing  & Comparison with\\
            & & & & & this paper \\
  \hline
  $4p^{*}$ & 5 & 10 & 71 & $\approx 19200^*$ years & $\times 100$\\
  \hline
$6p^{*}$ &  8 & 20 & 83.7 & $422^*$ years & $\times 270$ \\
\hline
$8p^{*}$ &  9 & 30 & 87.4  & $90^*$ years & $\times 906$ \\ 
\hline
$10p^{*}$&  10 & 42 & 91 & $30^*$ years & $\times 953$\\
\hline
$12p^{*}$ & 12 & 56 & 94.3 & $31^*$ years & $\times 2264$ \\
\hline
\end{tabular}
\end{center}
\caption{Stehlé's BaCSeL parameters and timings for the exponential function over the binade $[1/4,1/2)$}\label{tab:exp_bacsel}
\end{table}

We observe that we gain two to three orders of magnitude, increasing with
the value of $\alpha_{\textrm{Ste}}$ (which is our $d$); our method allows for
slightly larger intervals (by a factor around 2, which seems to
decrease with $d$), but the main factors explaining the difference are
the fact that the entries of our lattices are smaller and already close to being reduced thanks to our prereduction trick.

\section*{Acknowledgements.} We wish to thank Jean-Michel Muller for so many invaluable discussions, Martin Albrecht for his kind and effective help with fplll and Paul Zimmermann for his thorough and extremely helpful rereading of this paper.

\bibliographystyle{abbrv}
\bibliography{nbrisebarre}

@string{ACM="ACM"}

@string{IEEE="IEEE Computer Society Press"}

@string{Springer="Springer-Verlag"}

@article {AKR2018,
    AUTHOR = {Aanderaa, St{\aa}l and Kristiansen, Lars and Ruud, Hans
              Kristian},
     TITLE = {Search for good examples of {H}all's conjecture},
   JOURNAL = {Math. Comp.},
  FJOURNAL = {Mathematics of Computation},
    VOLUME = {87},
      YEAR = {2018},
    NUMBER = {314},
     PAGES = {2903--2914},
      ISSN = {0025-5718,1088-6842},
   MRCLASS = {11Y50 (11D25 65A05)},
  MRNUMBER = {3834691},
MRREVIEWER = {Masanari\ Kida},
       DOI = {10.1090/mcom/3298},
}

@article {BB2002,
    AUTHOR = {Bauer, Mark and Bennett, Michael A.},
     TITLE = {Applications of the hypergeometric method to the generalized
              {R}amanujan-{N}agell equation},
   JOURNAL = {Ramanujan J.},
  FJOURNAL = {Ramanujan Journal. An International Journal Devoted to the
              Areas of Mathematics Influenced by Ramanujan},
    VOLUME = {6},
      YEAR = {2002},
    NUMBER = {2},
     PAGES = {209--270},
      ISSN = {1382-4090,1572-9303},
   MRCLASS = {11D45 (11D61 11J82 11J86)},
  MRNUMBER = {1908198},
MRREVIEWER = {Yann\ Bugeaud},
       DOI = {10.1023/A:1015779301077},
}

@article{BertheImbert09,
  author    = {Val{\'{e}}rie Berth{\'{e}} and
               Laurent Imbert},
  title     = {Diophantine Approximation, {O}strowski Numeration and the
  Double-Base
               Number System},
    JOURNAL = {Discrete Math. Theor. Comput. Sci.},
   FJOURNAL = {Discrete Mathematics \& Theoretical Computer Science. DMTCS.},
  volume    = {11},
  number    = {1},
  pages     = {153--172},
  year      = {2009},
  url       = {http://dmtcs.episciences.org/450},
  bibsource = {dblp computer science bibliography, https://dblp.org}
}

@article {Beukers1980,
    AUTHOR = {Beukers, F.},
     TITLE = {On the generalized {R}amanujan-{N}agell equation. {I}},
   JOURNAL = {Acta Arith.},
  FJOURNAL = {Polska Akademia Nauk. Instytut Matematyczny. Acta Arithmetica},
    VOLUME = {38},
      YEAR = {1980-1981},
    NUMBER = {4},
     PAGES = {389--410},
      ISSN = {0065-1036},
   MRCLASS = {10B99 (10F05)},
  MRNUMBER = {621008},
MRREVIEWER = {Masahiko\ Fujiwara},
       DOI = {10.4064/aa-38-4-389-410},
}

@article {Beukers1981,
    AUTHOR = {Beukers, F.},
     TITLE = {On the generalized {R}amanujan-{N}agell equation. {II}},
   JOURNAL = {Acta Arith.},
  FJOURNAL = {Polska Akademia Nauk. Instytut Matematyczny. Acta Arithmetica},
    VOLUME = {39},
      YEAR = {1981},
    NUMBER = {2},
     PAGES = {113--123},
      ISSN = {0065-1036},
   MRCLASS = {10B99 (10F05)},
  MRNUMBER = {639621},
MRREVIEWER = {Masahiko\ Fujiwara},
       DOI = {10.4064/aa-39-2-113-123},
}

@article{BombieriPila1989,
    AUTHOR = {Bombieri, E. and Pila, J.},
     TITLE = {The number of integral points on arcs and ovals},
   JOURNAL = {Duke Math. J.},
  FJOURNAL = {Duke Mathematical Journal},
    VOLUME = {59},
      YEAR = {1989},
    NUMBER = {2},
     PAGES = {337--357},
      ISSN = {0012-7094},
   MRCLASS = {11P21 (11D99)},
MRREVIEWER = {Ulrich Rausch},
}

@ARTICLE{BonehDurfee2000,
  AUTHOR = {Boneh, Dan and Durfee, Glenn},
     TITLE = {Cryptanalysis of {RSA} with private key {$d$} less than
              {$N^{0.292}$}},
   JOURNAL = {IEEE Trans. Inform. Theory},
  FJOURNAL = {Institute of Electrical and Electronics Engineers.
              Transactions on Information Theory},
    VOLUME = {46},
      YEAR = {2000},
    NUMBER = {4},
     PAGES = {1339--1349},
      ISSN = {0018-9448},
   MRCLASS = {94A60 (11Y05 68Q25)},
MRREVIEWER = {Marc M. Gysin},
       DOI = {10.1109/18.850673},
}

@article{BHMZ2026,
  author       = {Nicolas Brisebarre and Guillaume Hanrot and Jean{-}Michel Muller and Paul Zimmermann},
  title        = {Correctly-Rounded Evaluation of a Function: Why, How, and at What Cost?},
      fjournal =  {ACM Computing Surveys},
  journal      = {{ACM} Comput. Surv.},
  volume       = {58},
  number       = {1},
  pages        = {27:1--27:34},
  year         = {2026},
  url          = {https://doi.org/10.1145/3747840},
  doi          = {10.1145/3747840},
  bibsource    = {dblp computer science bibliography, https://dblp.org}
}

@article{Coppersmith1997,
  author    = {Don Coppersmith},
  title     = {Small Solutions to Polynomial Equations, and Low Exponent {RSA} Vulnerabilities},
  journal   = {J. Cryptology},
  volume    = {10},
  number    = {4},
  pages     = {233--260},
  year      = {1997},
  doi       = {10.1007/s001459900030},  
  bibsource = {dblp computer science bibliography, https://dblp.org}
}

@article{SLZ2005,
    author={D.~Stehl{\'e} and V.~Lef{\`e}vre and P.~Zimmermann},
    title={{S}earching {W}orst {C}ases of a {O}ne-{V}ariable {F}unction {U}sing {L}attice {R}eduction},
   JOURNAL = {IEEE Trans. Comput.},
   fjournal={{IEEE} Transactions on Computers},
    volume={54},
    year={2005},
    pages={340--346},
    month= mar,
    number={3},
    key={SLZ2005},
    language = {English},
}

@article{Gramain1984,
    AUTHOR = {Gramain, F.},
     TITLE = {{Sur le lemme de Siegel (d’après E. Bombieri et J. Vaaler)}},
   JOURNAL = {Publications mathématiques de l’Univ. P. et M. Curie},
  FJOURNAL = {Publications mathématiques de l’Université Pierre et Marie Curie},
    VOLUME = {64},
      YEAR = {1983-1984},
     PAGES = {fascicule 1},
}

@article {HvdH2021,
    AUTHOR = {Harvey, David and van der Hoeven, Joris},
     TITLE = {Integer multiplication in time {$O(n \log n)$}},
   JOURNAL = {Ann. of Math. (2)},
  FJOURNAL = {Annals of Mathematics. Second Series},
    VOLUME = {193},
      YEAR = {2021},
    NUMBER = {2},
     PAGES = {563--617},
      ISSN = {0003-486X},
   MRCLASS = {11Y16 (68W30)},
  MRNUMBER = {4224716},
MRREVIEWER = {Markus Hittmeir},
       DOI = {10.4007/annals.2021.193.2.4},
}

@article {HeathBrown2002,
    AUTHOR = {Heath-Brown, D. R.},
     TITLE = {The density of rational points on curves and surfaces},
   JOURNAL = {Ann. of Math. (2)},
  FJOURNAL = {Annals of Mathematics. Second Series},
    VOLUME = {155},
      YEAR = {2002},
    NUMBER = {2},
     PAGES = {553--595},
      ISSN = {0003-486X,1939-8980},
   MRCLASS = {11G35 (11G50 14G05 14G40)},
  MRNUMBER = {1906595},
MRREVIEWER = {Carlo\ Gasbarri},
       DOI = {10.2307/3062125},
}

@article {Jarnik1926,
    AUTHOR = {Jarn\'{\i}k, Vojt\v{e}ch},
     TITLE = {\"{U}ber die {G}itterpunkte auf konvexen {K}urven},
   JOURNAL = {Math. Z.},
  FJOURNAL = {Mathematische Zeitschrift},
    VOLUME = {24},
      YEAR = {1926},
    NUMBER = {1},
     PAGES = {500--518},
      ISSN = {0025-5874,1432-1823},
   MRCLASS = {99-04},
  MRNUMBER = {1544776},
       DOI = {10.1007/BF01216795},
}

@article {CHS2009,
    AUTHOR = {Jim\'enez Calvo, I. and Herranz, J. and S\'aez, G.},
     TITLE = {A new algorithm to search for small nonzero {$|x^3-y^2|$}
              values},
   JOURNAL = {Math. Comp.},
  FJOURNAL = {Mathematics of Computation},
    VOLUME = {78},
      YEAR = {2009},
    NUMBER = {268},
     PAGES = {2435--2444},
      ISSN = {0025-5718,1088-6842},
   MRCLASS = {11Y50 (11D25)},
  MRNUMBER = {2521296},
MRREVIEWER = {Konstantinos\ Draziotis},
       DOI = {10.1090/S0025-5718-09-02240-6},
}

@ARTICLE{Johansson2017,
 AUTHOR = {Johansson, Fredrik},
     TITLE = {Arb: efficient arbitrary-precision midpoint-radius interval
              arithmetic},
   JOURNAL = {IEEE Trans. Comput.},
  FJOURNAL = {Institute of Electrical and Electronics Engineers.
              Transactions on Computers},
    VOLUME = {66},
      YEAR = {2017},
    NUMBER = {8},
     PAGES = {1281--1292},
      ISSN = {0018-9340},
   MRCLASS = {Expansion},
       DOI = {10.1109/TC.2017.2690633},
}

@article {KheVou2011,
    AUTHOR = {Kh{\'e}mira, Samy and Voutier, Paul},
     TITLE = {Approximation diophantienne et approximants de
              {H}ermite-{P}ad{\'e} de type {I} de fonctions exponentielles},
   JOURNAL = {Ann. Sci. Math. Qu{\'e}bec},
  FJOURNAL = {Annales des Sciences Math{\'e}matiques du Qu{\'e}bec},
    VOLUME = {35},
      YEAR = {2011},
    NUMBER = {1},
     PAGES = {85--116},
      ISSN = {0707-9109},
   MRCLASS = {11J82 (33B10 41A21)},
MRREVIEWER = {Michel Waldschmidt},
}

@article {Knill2014,
    AUTHOR = {Knill, Oliver},
     TITLE = {Cauchy-{B}inet for pseudo-determinants},
   JOURNAL = {Linear Algebra Appl.},
  FJOURNAL = {Linear Algebra and its Applications},
    VOLUME = {459},
      YEAR = {2014},
     PAGES = {522--547},
      ISSN = {0024-3795},
   MRCLASS = {15A15},
MRREVIEWER = {Ravindra B. Bapat},
       DOI = {10.1016/j.laa.2014.07.013},
}

@article{LLL1982,
     AUTHOR = {Lenstra, A. K. and Lenstra, Jr., H. W. and Lov\'{a}sz, L.},
     TITLE = {Factoring polynomials with rational coefficients},
   JOURNAL = {Math. Ann.},
  FJOURNAL = {Mathematische Annalen},
    VOLUME = {261},
      YEAR = {1982},
    NUMBER = {4},
     PAGES = {515--534},
      ISSN = {0025-5831},
   MRCLASS = {12-04 (12A20 68C20 68C25)},
MRREVIEWER = {Daniel Lazard},
       DOI = {10.1007/BF01457454},
}

@article{NW95,
    author={Y.~V. Nesterenko and M.~Waldschmidt},
    title={On the approximation of the values of exponential function and logarithm
        by algebraic numbers (in {R}ussian)},
    journal={Mat. Zapiski},
    volume= 2,
    year= 1996,
    pages={23--42},
    key={NW95},
    note = {Available in English at \url{https://arxiv.org/abs/math/0002047}}
}

@article{NguyenS09,
  author    = {Phong Q. Nguyen and
               Damien Stehl{\'{e}}},
  title     = {An {LLL} Algorithm with Quadratic Complexity},
  journal   = {{SIAM} J. Comput.},
  volume    = {39},
  number    = {3},
  pages     = {874--903},
  year      = {2009},
  doi       = {10.1137/070705702},
  bibsource = {dblp computer science bibliography, https://dblp.org}
}

@article {Rivoal2007,
    AUTHOR = {Rivoal, Tanguy},
     TITLE = {Convergents and irrationality measures of logarithms},
   JOURNAL = {Rev. Mat. Iberoam.},
  FJOURNAL = {Revista Matem\'atica Iberoamericana},
    VOLUME = {23},
      YEAR = {2007},
    NUMBER = {3},
     PAGES = {931--952},
      ISSN = {0213-2230,2235-0616},
   MRCLASS = {11J82},
  MRNUMBER = {2414498},
MRREVIEWER = {Tapani\ Matala-Aho},
       DOI = {10.4171/RMI/519},
}

@article{Strang1999,
        title = {The discrete cosine transform},
       volume = {41},
       number = {1},
      JOURNAL = {SIAM Rev.},
     FJOURNAL = {SIAM Review},
       author = {Strang, G.},
         year = {1999},
        pages = {135--147}
}

@article{Xu2016,
 AUTHOR = {Xu, Kuan},
     TITLE = {The {C}hebyshev points of the first kind},
   JOURNAL = {Appl. Numer. Math.},
  FJOURNAL = {Applied Numerical Mathematics. An IMACS Journal},
    VOLUME = {102},
      YEAR = {2016},
     PAGES = {17--30},
      ISSN = {0168-9274},
   MRCLASS = {41A05 (42C05 65D05)},
MRREVIEWER = {Alexander P. Goncharov},
       DOI = {10.1016/j.apnum.2015.12.002},
}

@book{Boyd2001,
	Address = {Mineola, NY},
	Author = {Boyd, John P.},
	Edition = {Second},
	Pages = {xvi+668},
	Publisher = {Dover Publications Inc.},
	Title = {Chebyshev and {F}ourier spectral methods},
	Url = {http://www-personal.umich.edu/~jpboyd/BOOK_Spectral2000.html},
	Year = {2001}
}

@book {Cas1997,
    AUTHOR = {Cassels, J. W. S.},
     TITLE = {An introduction to the geometry of numbers},
    SERIES = {Classics in Mathematics},
      NOTE = {Corrected reprint of the 1971 edition},
 PUBLISHER = {Springer-Verlag},
   ADDRESS = {Berlin},
      YEAR = {1997},
     PAGES = {viii+344},
}

@book{Cheney1982,
    AUTHOR = {Cheney, E. W.},
     TITLE = {Introduction to approximation theory},
      NOTE = {Reprint of the second (1982) edition},
 PUBLISHER = {AMS Chelsea Publishing, Providence, RI},
      YEAR = {1998},
     PAGES = {xii+259},
      ISBN = {0-8218-1374-9},
   MRCLASS = {41-01 (41-02)},
}

@book{FoxParker1968,
    AUTHOR = {Fox, L. and Parker, I. B.},
     TITLE = {Chebyshev polynomials in numerical analysis},
 PUBLISHER = {Oxford University Press, London-New York-Toronto, Ont.},
      YEAR = {1968},
     PAGES = {ix+205},
   MRCLASS = {65.10},
MRREVIEWER = {G. N. Lance},
}

@BOOK{vzGG2013,
        AUTHOR             = {von zur Gathen, J. and Gerhard, J.},
        PUBLISHER          = {Cambridge University Press},
        TITLE              = {Modern computer algebra},
        YEAR               = {2013},
        edition = {third},
        language = {English}
}

@book {GruLek1987,
    AUTHOR = {Gruber, P. M. and Lekkerkerker, C. G.},
     TITLE = {Geometry of numbers},
    SERIES = {North-Holland Mathematical Library},
    VOLUME = {37},
   EDITION = {2nd},
 PUBLISHER = {North-Holland Publishing Co.},
   ADDRESS = {Amsterdam},
      YEAR = {1987},
     PAGES = {xvi+732},
}

@book{Cohen1993,
title={A course in computational algebraic number theory},
booktitle={A course in computational algebraic number theory},
author={H.~Cohen},
publisher={Springer-Verlag},
address={Berlin},
series={Graduate Texts in Mathematics},
volume={138},
year={1993}
}

@book {Lov1986,
    AUTHOR = {L.~Lov{\'a}sz},
     TITLE = {An algorithmic theory of numbers, graphs and convexity},
    SERIES = {CBMS-NSF Regional Conference Series in Applied Mathematics},
    VOLUME = {50},
 PUBLISHER = {Society for Industrial and Applied Mathematics (SIAM)},
   ADDRESS = {Philadelphia, PA},
      YEAR = {1986},
     PAGES = {iv+91},
}

@book{Mason2002,
  title={Chebyshev polynomials},
  author={Mason, John C. and Handscomb, David C.},
  year={2002},
  publisher={CRC Press}
}

@book{MullerEtAl2010,
         title = {\href{https://www.springer.com/birkhauser/mathematics/book/978-0-8176-4704-9}{Handbook of Floating-Point Arithmetic}},
        author = {Jean-Michel Muller and Nicolas Brisebarre and Florent de Dinechin and Claude-Pierre Jeannerod and Vincent Lef{\`e}vre and Guillaume Melquiond and Nathalie Revol and Damien Stehl{\'e} and Serge Torres},
	pages = {572},
	publisher = {Birkh{\"a}user},
        year = {2010}, 
        language = {English},
        key = "MullerEtAl2010",
}

@book {PPST2018,
    AUTHOR = {Plonka, Gerlind and Potts, Daniel and Steidl, Gabriele and
              Tasche, Manfred},
     TITLE = {Numerical {F}ourier analysis},
    SERIES = {Applied and Numerical Harmonic Analysis},
 PUBLISHER = {Birkh\"{a}user/Springer, Cham},
      YEAR = {2018},
     PAGES = {xvi+168},
      ISBN = {978-3-030-04305-6; 978-3-030-04306-3},
   MRCLASS = {65-02 (42-02 65Txx 94A12)},
  MRNUMBER = {3890075},
       DOI = {10.1007/978-3-030-04306-3},
}

@book{Powell1981,
  title={Approximation theory and methods},
  author={Powell, Michael James David},
  year={1981},
  publisher={Cambridge University Press}
}

@book{Rivlin1974,
    AUTHOR = {Rivlin, Theodore J.},
     TITLE = {The {C}hebyshev polynomials},
      NOTE = {Pure and Applied Mathematics},
 PUBLISHER = {Wiley-Interscience [John Wiley \& Sons], New
              York-London-Sydney},
      YEAR = {1974},
     PAGES = {vi+186},
}

@book{Trefethen2013,
	Author = {Trefethen, Lloyd Nicholas},
	Publisher = {SIAM},
	Title = {\href{https://www.chebfun.org/ATAP/}{Approximation {T}heory and {A}pproximation {P}ractice}},
	Url = {https://www.chebfun.org/ATAP/},
	Year = {2013}}

@book{Tucker2011,
	Address = {Princeton, NJ},
	Author = {Tucker, Warwick},
	Isbn = {978-0-691-14781-9},
	Mrclass = {65-02 (65G30)},
	Mrreviewer = {G. Alefeld},
	Note = {A short introduction to rigorous computations},
	Pages = {xii+138},
	Publisher = {Princeton University Press},
	Title = {Validated numerics, A short introduction to rigorous computations},
	Year = {2011}
}

@book{LLL+25, 
  editor    = {Phong Q. Nguyen and
               Brigitte Vall{\'{e}}e},
  title     = {The {LLL} Algorithm - Survey and Applications},
  series    = {Information Security and Cryptography},
  publisher = {Springer},
  year      = {2010},
  doi       = {10.1007/978-3-642-02295-1},
  isbn      = {978-3-642-02294-4},
  bibsource = {dblp computer science bibliography, http://dblp.org}
}

@inproceedings{AkhaviStehle2008,
  author    = {Akhavi, Ali  and Stehl\'e, Damien},
  title     = {{S}peeding-up {L}attice {R}eduction with {R}andom {P}rojections},
  booktitle = {{LATIN} 2008: Theoretical Informatics, 8th Latin American Symposium,
  B{\'{u}}zios, Brazil, April 7-11, 2008, Proceedings},
  pages     = {293--305},
  year      = {2008},
  howpublished = {\url{https://hal.archives-ouvertes.fr/hal-00550984}},
  bibsource = {dblp computer science bibliography, http://dblp.org},
  language = {English}
}

@inproceedings{BrisebarreChevillard2007,
 author = {N. Brisebarre and S. Chevillard},
 title = {Efficient polynomial {$L^\infty$} approximations},
 booktitle = {ARITH '07: Proceedings of the 18th IEEE Symposium on Computer Arithmetic},
 year = {2007},
 isbn = {0-7695-2854-6},
 pages = {169--176},
 doi = {http://dx.doi.org/10.1109/ARITH.2007.17},
 publisher = {IEEE Computer Society},
 address = {Washington, DC},
 }

@inproceedings{BriJol10, 
  author       = {Nicolas Brisebarre and
                  Mioara Joldes},
  editor       = {Wolfram Koepf},
  title        = {Chebyshev interpolation polynomial-based tools for rigorous computing},
  booktitle    = {Symbolic and Algebraic Computation, International Symposium, {ISSAC}
                  2010, Munich, Germany, July 25-28, 2010, Proceedings},
  pages        = {147--154},
  publisher    = {{ACM}},
  year         = {2010},
  url          = {https://doi.org/10.1145/1837934.1837966},
  doi          = {10.1145/1837934.1837966}
}

@inproceedings{Coppersmith2001,
    author={Don Coppersmith},
    title={Finding Small Solutions to Small Degree Polynomials},
    booktitle={Proceedings of Cryptography and Lattices {(CaLC})},
    year= 2001,
    publisher={Springer-Verlag, Berlin},
    editor={J.~H. Silverman},
    series={Lecture Notes in Computer Science},
    pages = {20--31},
volume = 2146
}

@inproceedings{Elkies2000,
author= {Elkies, N.~D.},
title={Rational points near curves and small nonzero $|x^3-y^2|$ via lattice
reduction},
booktitle={Proceedings of the 4th {A}lgorithmic {N}umber {T}heory
{S}ymposium ({ANTS IV})},
publisher= {Springer-Verlag, Berlin},
editor    = {Wieb Bosma},
series= {Lecture Notes in Computer Science},
year= 2000,
pages= {33--63},
volume =1838}

@inproceedings{LefevreMuller2001a,
    author={V. Lef{\`e}vre and J.-M.~Muller},
    title={Worst Cases for Correct Rounding of the Elementary Functions in Double Precision},
     booktitle={Proceedings of the 15th {IEEE} Symposium on Computer Arithmetic (ARITH-15)},
editor = {N.~Burgess and L.~Ciminiera},
    month = jun,
    year = {2001},
    address = {Vail, CO},
    pages={111--118},
    key={LefevreMuller2001a},
 doi = {http://doi.ieeecomputersociety.org/10.1109/ARITH.2001.930110}
}

@inproceedings{Ajtai1998,
  author    = {M.~Ajtai},
  title     = {The {S}hortest {V}ector {P}roblem in {L}$_{\mbox{2}}$ is {NP}-hard
               for {R}andomized {R}eductions ({E}xtended {A}bstract)},
  booktitle = {Proceedings of the 30th {ACM} symposium on Theory of computing (STOC)},
  year      = {1998},
  pages     = {10--19},
  ee        = {http://doi.acm.org/10.1145/276698.276705},
  bibsource = {DBLP, http://dblp.uni-trier.de}
}

@inproceedings{SZG2022,
  author       = {Alexei Sibidanov and
                  Paul Zimmermann and
                  St{\'{e}}phane Glondu},
  title        = {The {CORE-MATH} Project},
  booktitle    = {29th {IEEE} Symposium on Computer Arithmetic, {ARITH} 2022, Lyon,
                  France, September 12-14, 2022},
  pages        = {26--34},
  publisher    = {{IEEE}},
  year         = {2022},
  url          = {https://doi.org/10.1109/ARITH54963.2022.00014},
  doi          = {10.1109/ARITH54963.2022.00014},
  bibsource    = {dblp computer science bibliography, https://dblp.org}
}

@inproceedings{Stehle2006,
author= {Stehl{\'e}, D.},
title={On the Randomness of Bits Generated by Sufficiently Smooth Functions},
booktitle={Proceedings of the 7th Algorithmic Number Theory Symposium,
                  ANTS VII},
publisher= {Springer-Verlag, Berlin},
editor    = {F.~Hess and S.~Pauli and M.~E.~Pohst},
series={Lecture Notes in Computer Science},
year= 2006,
pages= {257--274},
volume =4076}

@inproceedings{Lefevre2005,
    author={V.~Lef{\`e}vre},
    title={New Results on the Distance Between a Segment and $\mathbb{Z}^2$. {A}pplication to the Exact Rounding},
    booktitle={Proceedings of the 17th {IEEE} Symposium on Computer Arithmetic (ARITH-17)},
    year={2005},
    month= jun,
    pages={68--75},
    publisher={{IEEE} Computer Society Press, Los Alamitos, CA},
    key={Lefevre2005},
    location={Cape Cod, MA}
}

@misc{crlibm,
        author = {The {Ar\'enaire} Project},
        title = "{CR}libm, {C}orrectly {R}ounded mathematical library",
        month = {July},
        year = 2006,
        note = "\url{https://github.com/taschini/crlibm}"
}

@PhdThesis{ChevillardPhDThesis,
  key = "ChevillardPhDThesis",
  author = 	 {Chevillard, Sylvain},
  title = 	 {\href{https://dumas.ccsd.cnrs.fr/THESES-ENS-LYON/tel-00460776}{{\'Evaluation efficace de fonctions num\'eriques 
              -- Outils et exemples}}},
  school = 	 {\'Ecole normale sup\'erieure de {L}yon -- {U}niversit\'e de {L}yon},
  year = 	 {2009},
  howpublished = {\url{http://www-sop.inria.fr/members/Sylvain.Chevillard/download/papers/TheseSylvainChevillard.pdf}},
  url = {https://theses.hal.science/tel-00460776},
   language = {French},
   note = {In French},
}

@PhDThesis{KhemiraPhD,
  author = {Kh{\'e}mira, Samy},
  title = {\href{https://hal.archives-ouvertes.fr/tel-00009653}
          {Approximants de {H}ermite-{P}adé, déterminants d’interpolation et approximation diophantienne}},
  school = {{U}niversit{\'e} {P}aris 6,  {P}aris, 
            {F}rance},
  year = {2005},
  howpublished = {\url{https://hal.archives-ouvertes.fr/tel-00009653}},
  url = {https://tel.archives-ouvertes.fr/tel-00657843},
 language = {French}
}

@phdthesis{Lauter2008,
    author={C. Q. Lauter},
    title={\href{https://www.christoph-lauter.org/these.pdf}{Arrondi Correct de Fonctions Math{\'e}matiques}},
    school={{\'Ecole Normale Sup{\'e}rieure de Lyon, Lyon, France}},
    year= 2008,
    month= oct,
    url={https://www.christoph-lauter.org/these.pdf},
    language={French}
}

@PhDThesis{TorresPhD,
  key = "torresPhd", 
  author = {Torres, Serge},
  title = {\href{https://tel.archives-ouvertes.fr/tel-01396907}{Tools for the {D}esign of {R}eliable and {E}fficient {F}unctions {E}valuation {L}ibraries}},
  school = {\'Ecole normale sup\'erieure de Lyon -- {U}niversit\'e de {L}yon, {L}yon, {F}rance},
  year = {2016},
  howpublished = {\url{https://tel.archives-ouvertes.fr/tel-01396907}},
  url = {https://theses.hal.science/tel-01396907},
}

@phdthesis{Lef2000,
    author={V.~Lef{\`e}vre},
     title={\href{https://www.vinc17.net/research/papers/these.pdf}{Moyens Arithm{\'e}tiques Pour un Calcul Fiable}},
    school={{\'E}cole Normale Sup{\'e}rieure de Lyon, Lyon, France},
    year= 2000,
    note= {\url{https://www.vinc17.net/research/papers/these.pdf}},
   howpublished = {\url{https://www.vinc17.net/research/papers/these.pdf}},
 language = {French},
}

@techreport{CRLIBM2006,
    author={C. Daramy-Loirat and D. Defour and F. de Dinechin and M. Gallet and N. Gast
        and C. Q. Lauter and J.-M. Muller},
    title={{CR-LIBM}, A library of correctly-rounded elementary functions in double-precision},
    institution={LIP Laboratory, Arenaire team},
    year={2006},
    month= dec,
    address={Available at \url{https://ens-lyon.hal.science/ensl-01529804/file/crlibm.pdf}}
}

@Unpublished{BH2023,
  author      = {Brisebarre, Nicolas and Hanrot, Guillaume},
  title       = {{Integer points close to a transcendental curve and correctly-rounded evaluation of a function}},
  note        = {working paper or preprint},
  year        = {2023},
  hal_id      = {hal-03240179},
  hal_version = {v4},
  url         = {https://hal.science/hal-03240179},
}

\appendix

\section{Proofs of facts regarding Chebyshev polynomials}\label{app:proofscheby}

In this section,  $E_\rho$, ${\mathcal E}_\rho$, $E_{\rho_1,\rho_2}$ and ${\mathcal E}_{\rho_1,\rho_2}$ respectively denote $E_{\rho,-1,1}$, ${\mathcal E}_{\rho,-1,1}$, $E_{\rho_1,-1,1,\rho_2,-1,1}$  and ${\mathcal E}_{\rho_1,-1,1,\rho_2,-1,1}$ for all $\rho,\rho_1,\rho_2 > 1$.

First, 
 we recall the aliasing phenomenon in the case of Chebyshev nodes of the first kind.
\begin{proposition}\label{prop:aliasing}
  For all $N\ge 1$,
  \begin{itemize}
  \item for $k = 0, \ldots, N-1$, the polynomials $T_k, - T_{2N -k}, - T_{2N +k},   T_{4N -k}, $ $  T_{4N +k}, \ldots$ take the same values at the $\mu_{\ell,N-1}$, $\ell=0,\ldots, N-1$,
  \item for $j \ge 0$, let 
\[
m_{j,N} = | (j+N-1) (\bmod 2N) - N+1 | \textrm{ and } s_{j,N} =\left  \lfloor \frac{N+j}{2N} \right \rfloor, 
\]
the polynomials $T_j$ and $(-1)^{s_{j,N}}T_{m_{j,N}}$ take the same values at the $\mu_{\ell,N-1}$, $\ell=0,\ldots, N-1$.
  \end{itemize}
\end{proposition}
\begin{proof}
  These are Theorems 1 and 2 of~\cite{Xu2016}.
\end{proof}  
 Let $f$ be Lipschitz continuous over $[-1,1]$, we know~\cite[Thm. 3.1]{Trefethen2013} that $f$ admits a series expansion $\sumprime_{n\ge 0} a_n T_n(x)$ in $L_2 \left([-1,1],(1-x^2)^{-1/2}\right)$ which converges uniformly and absolutely to $f$. We now state a straightforward consequence of  Proposition~\ref{prop:aliasing}.
\begin{corollary}\label{cor:aliasing}
  Let $N\in \nat, N \ge 1$. Let $f$ be Lipschitz continuous over $[-1,1]$,
  its Chebyshev coefficients $(a_k)_{k\ge 0}$ and the coefficients $c_{k}, k =0, \ldots, N-1,$ of the interpolation polynomial $p_{N-1} = \sumprime_{0 \le n \le N-1} c_n T_n(x)$  of $f$ at the Chebyshev nodes of the first kind satisfy, for $k = 0, \ldots, N-1$,
\begin{equation}\label{eq:aliasing}
  c_{k} = \sum_{j=0}^{+\infty}(-1)^j  a_{2jN+k} +  \sum_{j=1}^{+\infty}(-1)^j  a_{2jN-k}.
\end{equation}
\end{corollary}
Let  $N\in \nat$, $N \ge 1$, we also define 
\[
\gamma_{\rho,0,N-1} = 1 \textrm{ and }  \gamma_{\rho,k,N-1} =   \frac{1}{1-\rho^{-2N}} \left (  1 +  \frac{1}{\rho^{2(N-k)}} \right) \textrm{ for } k = 1, \ldots, N-1.
\]
\begin{proposition} \label{prop:ineqcauchy_app}
Let $\rho > 1$, let  $N\in \nat$, $N\ge 1$, $f$ be a function analytic in a neighborhood of ${E_\rho}$, the coefficients $(c_{k})_{k =0, \ldots, N-1}$ of the interpolation polynomial of $f$ at the Chebyshev nodes of the first kind satisfy
  \[
  |c_{k} | \le  2 \frac{M_\rho (f)}{\rho^k}  \gamma_{\rho,k,N-1}, k = 0, \ldots, N-1,
  \]
  where $M_\rho (f) = \max_{z \in \mathcal{E}_\rho} |f(z)|$. 
Moreover, we have
\[
\norminfmunpun{ f - p_{N-1} } \le \frac{4 M_{\rho}(f)}{\rho^{N-1} (\rho -1)}.
\]
\end{proposition}
\begin{proof}
First, we use the following consequence of~\cite[Thm 8.1]{Trefethen2013}: the Chebyshev coefficients satisfy
\begin{equation}\label{eq:cauchyL2}
|a_k| \le 2 \frac{M_\rho (f)}{\rho^k} \textrm{ for all } k \geq 0. 
\end{equation}
Then, we combine Eq.~\eqref{eq:aliasing} and Inequalities~\eqref{eq:cauchyL2} to obtain, for  $k = 1, \ldots, N-1$,
\begin{align*}
\left | c_{k} \right | & \le  | a_{k}| +  \sum_{j=1}^{+\infty} (|a_{2jN-k}| +  | a_{2jN+k}|)\\ 
& \le  2 \frac{M_\rho (f)}{\rho^k} \left (  1 + \frac{1}{\rho^{2N-2k}} + \frac{1}{\rho^{2N}} + \frac{1}{\rho^{4N-2k}} + \frac{1}{\rho^{4N}}+ \cdots \right) \\
& \le  2 \frac{M_\rho (f)}{\rho^k}  \frac{1}{1-\rho^{-2N}} \left (  1 + \frac{1}{\rho^{2(N-k)}} \right). 
\end{align*}
Moreover, recall that $c_{0} =  \frac{2}{N} \sum_{0 \le \ell \le N-1} f(\mu_{\ell})$, hence $\left|c_{0}\right| \leq  2 \max_{x \in [-1,1]} | f(x)| \le 2 M_\rho (f)$ by the maximum principle. 

Now, we turn to the estimate on the remainder. Corollary~\ref{cor:aliasing} yields, for any $x \in [-1,1]$, 
\[
f(x) - p_{N-1}(x) = \sum_{k \ge N} a_k (T_k(x) - \mathcal{A}_N (T_k)(x))
\]
where
\[
 \mathcal{A}_N (T_k) = \begin{cases} 0 \textrm{ if } k \equiv N \pmod{2N}, \\  (-1)^{s_{k,N}} T_{m_{k,N}} \textrm{ otherwise}. \end{cases}  
\]  Hence, we have, for any $x \in [-1,1]$, 
\begin{multline*}
| f(x) - p_{N-1}(x) | \le \sum_{k \ge N} |a_k | | T_k(x) - \mathcal{A}_N(T_k)(x)|\\ \le 2 \sum_{k \ge N} |a_k | \le 4 M_\rho (f) \sum_{k \ge N} \rho^{-k} = \frac{4 M_{\rho}(f)}{\rho^{N-1} (\rho -1)}. 
\end{multline*}
\end{proof}
Regarding the two variable case, we start by establishing results analogous to~\cite[Thms 8.1 and 8.2]{Trefethen2013}.
Let  $f$ in $L_2 \left([-1,1]\times [-1,1],(1-x^2)^{-1/2}(1-y^2)^{-1/2}\right)$, we denote by $\sumprime_{n_1 \ge 0} \sumprime_{n_2 \ge 0}  a_{n_1,n_2} T_{n_1}(x) T_{n_2}(y)$ its series expansion.
\begin{proposition}\label{prop:ineqcauchy2D_app}
Let $\rho_1, \rho_2 > 1$, $f$ be a function analytic in a neighborhood of ${E_{\rho_1,\rho_2}}$, the coefficients $a_{n_1,n_2}, n_1, n_2 \geq 0,$ of the Chebyshev series of $f$ satisfy, for all $n_1$  and $n_2 \in \NN$,
\begin{equation}\label{eq:cauchyL22D}
  |a_{n_1,n_2}| \le 4 \frac{M_{\rho_1,\rho_2} (f)}{\rho_1^{n_1} \rho_2^{n_2}}, 
\end{equation}
where $M_{\rho_1,\rho_2} (f) = \max_{z \in \mathcal{E}_{\rho_1,\rho_2}} |f(z)|$. Moreover, we have, for all $n_1$  and $n_2 \in \NN$, 
\begin{multline*}
\norminfmunpund{f(x,y) - \sumprime_{k_1 = 0}^{n_1} \sumprime_{k_2 = 0}^{n_2}  a_{k_1,k_2} T_{k_1}(x) T_{k_2}(y)} \\  \le \frac{4 \rho_1 \rho_2 M_{\rho_1,\rho_2}(f)}{(\rho_1 -1)(\rho_2 -1)} \left (\frac{1}{\rho_1^{n_1+1} } + \frac{1}{\rho_2^{n_2+1} }\right ).
\end{multline*}
\end{proposition}
\begin{proof}
For $\rho >0$, we define $ \mathcal{C}_{\rho} = \{ z \in \comp, |z| = \rho\}$.    Extending what is done in the proof of~\cite[Thm. 8.1]{Trefethen2013}, we now introduce the change of variables  $x = (z_1 + z_1^{-1})/2, y = (z_2 + z_2^{-1})/2$   where $z_1, z_2 \in \mathcal{C}_1$ and the function
  \[
  f(x,y) = F(z_1,z_2) = \sumprime_{n_1 \ge 0} \sumprime_{n_2 \ge 0}  a_{n_1,n_2} \frac{z_1^{n_1}+z_1^{-n_1}}{2} \frac{z_2^{n_2}+z_2^{-n_2}}{2} .
  \]
 Now we use Cauchy's integral formula in two variables: for all $n_1, n_2 \in \nat$, 
 \[
  \frac{1}{  (2i\pi)^2}   \int_{\mathcal{C}_1 \times \mathcal{C}_1} \frac{F(z_1,z_2)}{z_1^{n_1+1}z_2^{n_2+1}} \, \mathrm{d}z_1 \mathrm{d}z_2
 = \frac{a_{n_1,n_2}}{4}.
 \]
  If $J :  (z_1,z_2) \mapsto ( (z_1 + z_1^{-1})/2, (z_2 + z_2^{-1})/2)$, the domain ${E_{\rho_1,\rho_2}}$ contains the image of $ \mathcal{R}_{\rho_1} \times \mathcal{R}_{\rho_2}$, where $\mathcal{R}_\rho = \{ z \in \comp, \rho^{-1} < |z| < \rho \}$, via the map $J$. Note that, since $F = f \circ J $, $F$ is analytic in a neighborhood of $ { \mathcal{R}_{\rho_1} } \times {\mathcal{R}_{\rho_2}}$ since it is the composition of two analytic functions, hence, for all  $n_1, n_2 \in \nat$,
  \[ 
  a_{n_1,n_2} = \frac{1}{(i\pi)^2} \int_{\mathcal{C}_{\rho_1} \times \mathcal{C}_{\rho_2}} \frac{F(z_1,z_2)}{z_1^{n_1+1}z_2^{n_2+1}} \, \mathrm{d}z_1 \mathrm{d}z_2,
  \]
  from which follows
   \[
  |a_{n_1,n_2}| \leq  \frac{(2\pi)^2\rho_1\rho_2}{\pi^2} \frac{\max_{(z_1,z_2) \in \mathcal{C}_{\rho_1} \times \mathcal{C}_{\rho_2}} |F(z_1,z_2)|}{\rho_1^{n_1+1}\rho_2^{n_2+1}} =  \frac{4 M_{\rho_1,\rho_2} (f)}{\rho_1^{n_1} \rho_2^{n_2}} \textrm{ for all } n_1, n_2 \in \nat.
  \]
  As for the remainder, for all $x, y \in [-1,1]$, for all $n_1$  and $n_2 \in \NN$, we have 
\begin{multline*}
  f(x,y) - \sumprime_{k_1 = 0}^{n_1} \sumprime_{k_2 = 0}^{n_2}  a_{k_1,k_2} T_{k_1}(x) T_{k_2}(y)  \\ =
    \sum_{k_1 \ge  n_1 + 1} \sumprime_{k_2 \ge 0}  a_{k_1,k_2} T_{k_1}(x) T_{k_2}(y) +
  \sumprime_{k_1 = 0}^{ n_1 } \sum_{k_2 \ge  n_2 +1}  a_{k_1,k_2} T_{k_1}(x) T_{k_2}(y), 
\end{multline*}
  hence,  
\begin{multline*}  
\norminfmunpund{f(x,y) - \sumprime_{k_1 = 0}^{n_1} \sumprime_{k_2 = 0}^{n_2}  a_{k_1,k_2} T_{k_1}(x) T_{k_2}(y)}  \\ \leq \sum_{k_1 \ge  n_1 + 1} \sumprime_{k_2 \ge 0}  |a_{k_1,k_2}| +   \sumprime_{k_1 = 0}^{ n_1 } \sum_{k_2 \ge  n_2 +1}  |a_{k_1,k_2}|  
 \\  \le  4   M_{\rho_1,\rho_2} (f) \left ( \frac{1}{\rho_1^{n_1+1}} + \frac{1}{\rho_2^{n_2+1} } \right ) \frac{\rho_1 \rho_2}{(\rho_1 -1)(\rho_2 -1)}.
  \end{multline*}
\end{proof}
Now we can prove:
\begin{proposition}\label{prop:ineqcauchy2D_int_app}
  Let $\rho_1, \rho_2 > 1$, let  $N_1, N_2 \in \nat, N_1, N_2 \ge 2$, $f$ be a function analytic in a neighborhood of ${E_{\rho_1,\rho_2}}$, the coefficients $c_{k_1,k_2}, k_1 =0, \ldots, N_1 -1, k_2 =0, \ldots, N_2 -1$ of the interpolation polynomial $P_{N_1-1,N_2-1}$ of $f$ at pairs of Chebyshev nodes of the first kind satisfy, for $k_1 = 0, \ldots, N_1 -1$, $k_2 = 0, \ldots, N_2 -1$,
  \[
  \left | c_{k_1,k_2} \right | \le   4 \frac{M_{\rho_1,\rho_2} (f)}{\rho_1^{k_1}\rho_2^{k_2}}  \gamma_{\rho_1,k_1,N_1-1}  \gamma_{\rho_2,k_2,N_2-1},  
  \]
where $M_{\rho_1,\rho_2} (f) = \max_{z \in \mathcal{E}_{\rho_1,\rho_2}} |f(z)|$. 
 Moreover, we have 
\[
\norminfmunpund{ f - P_{N_1-1,N_2-1} } \le \frac{8 \rho_1 \rho_2 M_{\rho_1,\rho_2}(f)}{(\rho_1 -1)(\rho_2 -1)} \left (\frac{1}{\rho_1^{N_1} } + \frac{1}{\rho_2^{N_2} }\right ).
\]
\end{proposition}
\begin{proof}
Let $\sumprime_{n_1 \ge 0} \sumprime_{n_2 \ge 0}  a_{n_1,n_2} T_{n_1}(x) T_{n_2}(y)$ be the series expansion of $f$, 
 the aliasing phenomenon presented above still exists: for $k_1 = 0, \ldots, N_1 -1$, $k_2 = 0, \ldots, N_2 -1$,
\begin{multline}\label{eq:aliasing2D}
  c_{k_1,k_2} = \sum_{p_1=0}^{+ \infty} \sum_{p_2=0}^{+ \infty}(-1)^{p_1+p_2} a_{2p_1N_1+k_1,2p_2N_2+k_2} \\ + \sum_{p_1=0}^{+ \infty} \sum_{p_2=1}^{+ \infty}(-1)^{p_1+p_2} a_{2p_1N_1+k_1,2p_2N_2-k_2} +  \sum_{p_1=1}^{+ \infty} \sum_{p_2=0}^{+ \infty}(-1)^{p_1+p_2} a_{2p_1N_1-k_1,2p_2N_2+k_2} \\ + \sum_{p_1=1}^{+ \infty} \sum_{p_2=1}^{+ \infty}(-1)^{p_1+p_2} a_{2p_1N_1-k_1,2p_2N_2-k_2}.
\end{multline}
We now combine Equation~\eqref{eq:aliasing2D} and Inequalities~\eqref{eq:cauchyL22D} to obtain, for  $k_1 = 1, \ldots, N_1 -1, k_2 = 1, \ldots, N_2 -1$,
\begin{align*}
\left | c_{k_1,k_2} \right | \le &\, \sum_{p_1=0}^{+ \infty} \sum_{p_2=0}^{+ \infty} | a_{2p_1N_1+k_1,2p_2N_2+k_2} | +  \sum_{p_1=0}^{+ \infty} \sum_{p_2=1}^{+ \infty} | a_{2p_1N_1+k_1,2p_2N_2-k_2}|  \\
&\,  + \sum_{p_1=1}^{+ \infty} \sum_{p_2=0}^{+ \infty} |a_{2p_1N_1-k_1,2p_2N_2+k_2}| + \sum_{p_1=1}^{+ \infty} \sum_{p_2=1}^{+ \infty} | a_{2p_1N_1-k_1,2p_2N_2-k_2} | \\
\le &\, \sum_{p_1=0}^{+ \infty} \sum_{p_2=0}^{+ \infty} 4 \frac{M_{\rho_1,\rho_2} (f)}{\rho_1^{2p_1N_1+k_1}\rho_2^{2p_2N_2+k_2}} +  \sum_{p_1=0}^{+ \infty} \sum_{p_2=1}^{+ \infty}  4 \frac{M_{\rho_1,\rho_2} (f)}{\rho_1^{2p_1N_1+k_1}\rho_2^{2p_2N_2-k_2}}   \\
&\,  + \sum_{p_1=1}^{+ \infty} \sum_{p_2=0}^{+ \infty} 4 \frac{M_{\rho_1,\rho_2} (f)}{\rho_1^{2p_1N_1-k_1}\rho_2^{2p_2N_2+k_2}} + \sum_{p_1=1}^{+ \infty} \sum_{p_2=1}^{+ \infty}  4 \frac{M_{\rho_1,\rho_2} (f)}{\rho_1^{2p_1N_1-k_1}\rho_2^{2p_2N_2-k_2}}  \\
\le &\,   4 \frac{M_{\rho_1,\rho_2} (f)}{\rho_1^{k_1}\rho_2^{k_2}}  \frac{1}{1-\rho_1^{-2N_1}}  \frac{1}{1-\rho_2^{-2N_2}} \\ & \left ( 1 +  \frac{1}{\rho_1^{2(N_1-k_1)}} +  \frac{1}{\rho_2^{2(N_2-k_2)}} +  \frac{1}{\rho_1^{2(N_1-k_1)}\rho_2^{2(N_2-k_2)}}  \right). 
\end{align*}
If we use Equation~\eqref{eq:coeffs-interpol2D}, we get
\begin{align*}
\left  |c_{0,0} \right| & \le 4 M_{\rho_1,\rho_2} (f) \textrm{ thanks to the maximum principle,} \\
\left  |c_{k_1,0} \right| & \le 4 \frac{M_{\rho_1,\rho_2} (f)}{\rho_1^{k_1}} \frac{1}{1-\rho_1^{-2N_1}} \left (  1 + \frac{1}{\rho_1^{2(N_1-k_1)}} \right) \textrm{ for } k_1 = 1, \ldots, N_1 -1, \\
\left  |c_{0,k_2} \right| & \le 4 \frac{M_{\rho_1,\rho_2} (f)}{\rho_2^{k_2}}  \frac{1}{1-\rho_2^{-2N_2}} \left (  1 + \frac{1}{\rho_2^{2(N_2-k_2)}} \right) \textrm{ for } k_2 = 1, \ldots, N_2 -1.
\end{align*}
The last two inequalities are consequences of Proposition~\ref{prop:ineqcauchy_app}. 

  As for the remainder, for all $x, y \in [-1,1]$, we have thanks to the aliasing phenomenon
\begin{multline} \label{eq:reste2D} 
  f(x,y) - \sumprime_{k_1 = 0}^{N_1-1} \sumprime_{k_2 = 0}^{N_2-1}  c_{k_1,k_2} T_{k_1}(x) T_{k_2}(y)  \\ =  \sum_{k_1 \ge  N_1} \sumprime_{k_2 \ge 0}  a_{k_1,k_2} ( T_{k_1}(x) T_{k_2}(y)  - {\mathcal A}_{N_1} (T_{k_1})(x) {\mathcal A}_{N_2} (T_{k_2})(y)) \\ +
  \sumprime_{k_1 = 0}^{ N_1 -1} \sum_{k_2 \ge  N_2}  a_{k_1,k_2} ( T_{k_1}(x) T_{k_2}(y) - {\mathcal A}_{N_1} (T_{k_1})(x) {\mathcal A}_{N_2} (T_{k_2})(y)  ).
\end{multline}
Hence, 
\begin{multline*}  
\norminfmunpund{f(x,y) - P_{N_1 -1,N_2-1} (x,y)}  \leq \sum_{k_1 \ge  N_1} \sumprime_{k_2 \ge 0} 2 |a_{k_1,k_2}| \\ +   \sumprime_{k_1 = 0}^{ N_1 -1 } \sum_{k_2 \ge  N_2} 2 |a_{k_1,k_2}| 
   \le  8   M_{\rho_1,\rho_2} (f) \left ( \frac{1}{\rho_1^{N_1}} + \frac{1}{\rho_2^{N_2} } \right ) \frac{\rho_1 \rho_2}{(\rho_1 -1)(\rho_2 -1)},
\end{multline*}
thanks to Proposition~\ref{prop:ineqcauchy2D_app}.
\end{proof}
The next lemma eases the computations. Recall that we introduced in~\Cref{subsec:unifapprox} $\eta_{\rho,0} = 1$ and $\eta_{\rho,k} =  (\rho^2+1)/(\rho^{2}-1)$ for $k = 1, \ldots, N-1$. 
\begin{lemma} \label{lem:simpligamma} Let $\rho > 1$, $N \ge 2$, for $k =0, \ldots, N-1$, we have
  \[
  \gamma_{\rho,k,N-1} \le \eta_{\rho,k}. 
  \]
In particular, if $\rho \ge 2$,  $\eta_{\rho,k} \le 2$ for $k =1, \ldots, N-1$.
\end{lemma}  
\begin{proof}
The case $k=0$ is straightforward. For $k = 1, \ldots, N-1$, we have
\[
\gamma_{\rho,k,N-1} =  \frac{1}{1-\rho^{-2N}} \left (1 + \frac{1}{\rho^{2(N-k)}} \right )  =  \frac{1}{\rho^{2N}-1} \left (\rho^{2N} + \rho^{2k}  \right ) \le \frac{\rho^{2N}}{\rho^{2N}-1} \left (1 + \rho^{-2}  \right ).
\]
The function $u \mapsto u/(u-1)$ is strictly decreasing over $(1,+\infty)$, hence
\[
\gamma_{\rho,k,N-1} \le  \frac{\rho^{2N}}{\rho^{2N}-1} \left (1 + \rho^{-2}  \right ) \le  \frac{\rho^{2}}{\rho^{2}-1} \left (1 + \rho^{-2}  \right ) = \frac{\rho^2 + 1}{\rho^2 - 1}.
\]
The last statement is obvious.
\end{proof}

\textit{Proof of  Proposition~\ref{prop:ineqcauchy2D}}. We introduce
\[
  \widehat{f}: \widehat{f}(z_1,z_2) = f \left ( z_1 \frac{b_1-a_1}{2} + \frac{a_1+b_1}{2}, z_2 \frac{b_2-a_2}{2} + \frac{a_2+b_2}{2} \right )
\]
  for any $(z_1,z_2)$ in a suitable neighborhood of $E_{\rho_1,\rho_2}$.  The coefficients $c_{k_1,k_2}$ are also the coefficients of the interpolation polynomial in $\ree_{N_1-1,N_2-1} [x,y]$ of $\widehat{f}$ at pairs of Chebyshev nodes of the first kind.  Hence, it suffices to apply Proposition~\ref{prop:ineqcauchy2D_int_app} to  $\widehat{f}$ and Lemma~\ref{lem:simpligamma}. \qed

Now we study the behavior of Chebyshev approximation outside of its
range of validity (namely, $[-1, 1]$ in our notations).
\begin{lemma} \label{lem:majTn}
  For $x\in \mathbb{R}$, $|x| \ge 1$, $n\ge 0$, we have
  \[
  |T_n(x)| \le \left (|x| + \sqrt{x^2-1} \right )^n.
  \]
\end{lemma}
\begin{proof}
  We write $|x| = \cosh(y)$ for $y  = \log(|x| + \sqrt{x^2-1})$. 
  As $T_{n}(\cosh(y)) = \cosh(ny)$, we have
  $|T_n(x)|=|T_n(|x|)| = \frac{1}{2} \left((|x| + \sqrt{x^2-1})^n + (|x| + \sqrt{x^2-1})^{-n}\right)
  \le (|x| + \sqrt{x^2-1})^n.$
\end{proof}

Note that $E_\rho$ can also be defined as $\{ z \in \comp: \max(| z + \sqrt{z^2- 1}|, |z-\sqrt{z^2-1}|) \leq \rho  \}$. As a corollary of Lemma~\ref{lem:majTn}, we see that for all $x \geq 1$, $x + \sqrt{x^2-1} < \rho$, the series $S(x) = \sumprime_{n\ge 0} a_n T_n(x)$ is convergent;
as $f = S$ over $[-1, 1]$, we see that $f = S$ over the domain of
convergence of $S$, which contains the interior of $E_\rho$, so at least on $[-1, x]$.

Thus, for all $x$ such that $x \geq 1, x + \sqrt{x^2-1} < \rho$, following the
proof of Proposition~\ref{prop:ineqcauchy_app}, the inequality
\[
|f(x) - p_{N-1}(x)| \le \frac{4M_{\rho}(f)}{\rho^{N-1} (\rho - 1)} 
\]
becomes
\[
|f(x) - p_{N-1}(x)|\le \frac{4M_{\rho}(f)\rho}{\rho - (x + \sqrt{x^2-1})} \left(\frac{x + \sqrt{x^2-1}}{\rho}\right)^N.
\]

In particular, for $x \le 3$, assuming $\rho \ge 8$, so that $\rho/(\rho - (x + \sqrt{x^2-1}))\le \rho/(\rho - (3 + 2 \sqrt{2})) \le 4$, we find
\[
|f(x) - p_{N-1}(x)| \le 16 M_{\rho}(f) \left(\frac{3 +
  2\sqrt{2}}{\rho}\right)^{N}.
\]

These statements extend to the case of bivariate functions under the
following form; the meaning of this Lemma is that the quality
of the Chebyshev interpolation over $[-1, 1]$ as an approximation
over a larger interval degrades, but not too quickly. 
\begin{lemma}\label{le:cheb_extrapol}
 Let $\rho_1, \rho_2 > 1$, let  $N_1, N_2 \in \nat, N_1, N_2 \ge 2$, $f$ be a function analytic in a neighborhood of ${E_{\rho_1,\rho_2}}$, let $P_{N_1-1, N_2-1}$ be the bivariate interpolant of order $(N_1-1, N_2 -1)$ of
$f$ over $[-1, 1]\times [-1, 1]$. Then,  for $x \in \ree,  1 \leq x  \leq 3$, $8 \le \rho_1$, $y\in [-1, 1]$, 
we have
\begin{multline*}
  \left|f(x,y) - P_{N_1-1, N_2-1} (x,y) \right| \\ \le  8 M_{\rho_1, \rho_2}(f) \left ( \left(\frac{3+2\sqrt{2}}{\rho_1}\right)^{N_1} + \frac{1}{\rho_2^{N_2}} \right ) \frac{\rho_1\rho_2}{(\rho_1 - (3+2\sqrt{2}))(\rho_2 -1)}.
\end{multline*}
\end{lemma}
\begin{proof}
We use again \eqref{eq:reste2D}, \eqref{eq:cauchyL22D} and Lemma~\ref{lem:majTn} to derive, for $x \in [1,3]$, $8 \le \rho_1$, $y\in [-1, 1]$, 
\begin{multline*}  
  |f(x,y) - P_{N_1 -1,N_2-1} (x,y)| \\ \leq 8 M_{\rho_1, \rho_2}(f) 
\left (  \sum_{k_1 \ge  N_1} \sumprime_{k_2 \ge 0} \frac{(3+2\sqrt{2})^{k_1}}{\rho_1^{k_1}\rho_2^{k_2}} \right.  + \left.  \sumprime_{k_1 = 0}^{N_1 -1} \sum_{k_2 \ge  N_2}  \frac{(3+2\sqrt{2})^{k_1}}{\rho_1^{k_1}\rho_2^{k_2}} \right ), 
\end{multline*}
hence the result.
\end{proof}

  \section{Some technical lemmata for the estimate of volumes}
  \label{app:phipsi}

We start with a convenient combinatorial lemma. 
\begin{lemma}\label{lem:bnd_using_ks}
Let $\gamma \in \ree, \gamma \geq 1$, $N, N_1, N_2$ positive integers. Consider the multiset\footnote{By sum of multisets, we mean that the multiplicity of an element of the union is the sum of its multiplicities in the multisets.} $S = \{ k + \gamma k', (k, k')\in [0, N_1 - 1]\times [0, N_2 - 1] \cap \Z^2 \} + \{ \underbrace{N_1,\ldots,N_1}_{N \textrm{ times}}\}$, and order the elements of $S$ as $\sigma_0 \leq \dots \leq \sigma_{\card   S-1}$.

Define 
  \[
  \Omega_\gamma(N,N_1, N_2) = \sigma_0 + \dots + \sigma_{N-1},
  \]
  the sum of the $N$ smallest elements of $S$.
  
  Let now $s \leq N_1$ be a real number, let $\K_s =\{ (i, j) \in [0, N_1 - 1] \times [0, N_2 - 1] \cap \Z^2: i + \gamma j \leq s \}$.   Then,
\[
  \Omega_\gamma(N,N_1, N_2) \geq s (N - \card \K_s) + \sum_{(i,j)\in \K_s} (i + \gamma j).
\]
\end{lemma}
\begin{proof}
  Put $\M_s = \{ i + j\gamma, (i,j) \in \K_s \}$. Then, if $\card \K_s \leq N$, $ \M_s \subset \{ \sigma_0, \dots,$ $ \sigma_{N-1} \}$, and
any element in $\{ \sigma_0, \dots,$ $ \sigma_{N-1} \} \setminus \M_s $ is at
least equal to $s$.

Otherwise, we have $\{ \sigma_0, \dots, \sigma_{N-1} \} \subset \M_s$, and any element in $\M_s \setminus \{ \sigma_0, \dots, \sigma_{N-1} \}$ is at most equal to $s$. 
\end{proof}
  
We now give explicit expressions for $\card \K_s$ and $\sum_{(i,j)\in\K_s} (i + j \gamma)$.
\begin{lemma}\label{lem:estim_card_sum_Ks}
Let $s \in \ree$, $0\le s < N_1$ and $\gamma \ge N_1/N_2 \ge 1$. We have 
\begin{equation}\label{eq:cardKs}
\card \K_s = \sum_{j=0}^{\lfloor s/\gamma \rfloor} (1 + \lfloor s - j \gamma \rfloor )= ( 1 +  \lfloor s/\gamma \rfloor)  + \sum_{j=0}^{\lfloor s/\gamma \rfloor} \lfloor s - j \gamma \rfloor 
\end{equation}
and
\begin{equation}\label{eq:sommeKs}
\sum_{(i,j)\in\K_s} (i + j \gamma) =   \sum_{j=0}^{\lfloor s/\gamma \rfloor}  ( 1 + \lfloor s - j\gamma   \rfloor ) \left (  j \gamma + \frac{ \lfloor s - j \gamma \rfloor }{2} \right  ).
\end{equation}
This implies
\begin{equation}\label{eq:ineqcardKs}
( 1 +  \lfloor s/\gamma \rfloor) ( s - \gamma  \lfloor s/\gamma \rfloor/2)  \leq \card \K_s \leq    ( 1 +  \lfloor s/\gamma \rfloor) (1 + s - \gamma  \lfloor s/\gamma \rfloor/2) 
\end{equation}
and
\begin{equation}\label{eq:ineqsommeKs}
\begin{multlined}
  (1 + \lfloor s/\gamma \rfloor) \frac{6s(s-1) +  \gamma \lfloor s/\gamma \rfloor  ( 3 - \gamma - 2 \gamma  \lfloor s/\gamma \rfloor)  }{12}   \le \sum_{(i,j)\in\K_s}  (i + j \gamma)  \\
\le    (1 + \lfloor s/\gamma \rfloor) \frac{6s(s+1) +  \gamma \lfloor s/\gamma \rfloor (3 - \gamma - 2 \gamma  \lfloor s/\gamma \rfloor)}{12}.
\end{multlined}
\end{equation}
\end{lemma}
\begin{proof}
First, note that $s/\gamma < N_1/\gamma \leq N_1/(N_1/N_2) = N_2$, hence $   \lfloor s/\gamma \rfloor \le N_2 - 1$. Let $(i,j) \in \K_s$, the largest possible value of $j$ corresponds to the case $i =0$: we then have $j\gamma \le s$, that is to say $j \le  \lfloor s/\gamma \rfloor$. Now, in order to count the elements of $\K_s$, we enumerate, for $j = 0, \ldots, \lfloor s/\gamma \rfloor$, the elements of each slice $\{ i + j \gamma \le s, i \in [0,N_1-1]\}$: there are $1  + \lfloor s - j \gamma \rfloor$ such elements, which proves~\eqref{eq:cardKs}.
 
Now,  for $j = 0, \ldots, \lfloor s/\gamma \rfloor$, we sum the values $i + j\gamma$ for $i$ in the slice  $\{ i + j \gamma \le s, i \in [0,N_1-1]\}$, i.e., $i \in [0,   \lfloor s - j \gamma \rfloor ]$. Hence,  for $j = 0, \ldots, \lfloor s/\gamma \rfloor$, we sum the values $  \lfloor s - j \gamma \rfloor (1+ \lfloor s - j \gamma \rfloor)/2 +  (1 +  \lfloor s - j \gamma \rfloor)j \gamma $, from which~\eqref{eq:sommeKs} follows.

We use $s - j \gamma - 1 \le \lfloor s - j \gamma \rfloor \le s - j \gamma$ for $j=0,\ldots, \lfloor s/\gamma \rfloor$ to derive from~\eqref{eq:cardKs} 
\[
   ( 1 +  \lfloor s/\gamma \rfloor) ( s - \gamma  \lfloor s/\gamma \rfloor/2) 
 \leq \card  \K_s \leq    ( 1 +  \lfloor s/\gamma \rfloor) (1 + s - \gamma  \lfloor s/\gamma \rfloor/2),
\]
which yields~\eqref{eq:ineqcardKs}.  Likewise, we  derive from~\eqref{eq:sommeKs}
\[
 \sum_{j=0}^{\lfloor s/\gamma \rfloor}  \frac{(s -  j \gamma  )(s +
   j \gamma -1)}{2}  \le \sum_{(i,j)\in\K_s}  (i + j \gamma) 
 \le \sum_{j=0}^{\lfloor s/\gamma \rfloor}   \frac{(1+s -  j \gamma  )(s +
   j \gamma )}{2},
 \]
which yields~\eqref{eq:ineqsommeKs}.
\end{proof}

\begin{lemma}\label{lem:phi}
  Let $\phi$ be the function from $[1, +\infty)$ to $[1, +\infty)$ defined by   $\phi(x) = (1 + \lfloor x \rfloor) (x - \lfloor x \rfloor /2)$.   Then $\phi$ is continuous and strictly increasing, and defines   a bijection from $[1, +\infty)$ to $[1, +\infty)$.  For any $x \geq 1$, we have $ x(x+1)/2 \le \varphi (x) \le (x+1/2)^2/2$ and $\sqrt{2x} -1/2 \le \phi^{-1}(x) \le \sqrt{2x+1/4} - 1/2$.
\end{lemma}
\begin{proof}
  For $x \not \in \Z$, it is clear that $\phi$ is $\mathcal{C}^1$ in a neighborhood  of $x$ and that $\phi'(x) \ge 1$. 

  For $x$ in $\Z$, we have $\phi(x) = \lim_{t\rightarrow x^+} \phi(t) = (1+x)x/2$,
  whereas $\lim_{t\rightarrow x^-} \phi(t) = x(x-(x-1)/2) = x(x+1)/2$. This
  proves continuity, and the remaining assertions follow.

  Let $k \in \nat, k \ge 1, a \in \ree$ and $g_a(x)=(x+1-a)(x+a)/2$. We denote by $\phi_k$ the restriction of $\phi$ to $[k,k+1]$. 
  For all $x \in [k,k+1]$, $(\phi_k - g_a)'(x) = k+1/2-x$: the function $\phi_k - g_a$ is increasing over $[k,k+1/2]$ and decreasing over $[k+1/2,k+1]$. Now, we remark that $\phi_k (k) = g_0(k)$ and $\phi_k (k+1) = g_0(k+1)$, which yields  $\phi_k (x) \ge g_0(x) = x(x+1)/2$ for all $x\in [k,k+1]$, and  $\phi_k (k+1/2) = g_{1/2}(k+1/2)$, which yields  $\phi_k (x) \le g_{1/2}(x) = (x+1/2)^2/2$ for all $x\in [k,k+1]$.

  The proof of the remaining inequalities is straightforward.
\end{proof}

\begin{lemma}\label{lem:estim_cardKs} Let $3 \le \gamma \le N$,  $s = \gamma \phi^{-1}(N/\gamma)$,   $N_1 = \left \lfloor \sqrt{2 N \gamma}  \right \rfloor$ and $N_2 =  \left \lceil \sqrt{2N/\gamma} \, \right \rceil$. We have $\gamma \geq N_1/N_2 \geq 1$, $s < N_1$ and
$  \sqrt{2N} (\sqrt{2N} - \sqrt{3}/3)     \leq N_1 N_2 \leq  (2+ \sqrt{2}) N$.
  
  Assume that $N \rightarrow \infty$,   then $N \le \card \K_s \le N + O(s / \gamma)$. 
\end{lemma}
\begin{proof}
We have
  \[
  \sqrt{2N\gamma} - \sqrt{2N/\gamma} = \sqrt{2N\gamma} (1 - 1/\gamma) \ge 1,
  \]
  since $N \ge \gamma \ge 3$. It follows that $ N_1 = \left \lfloor \sqrt{2 N \gamma} \right \rfloor \ge \left \lceil \sqrt{2N/\gamma}\, \right \rceil = N_2$, hence $N_1/N_2 \ge 1$. Moreover, for any $\gamma \ge 3$, 
  \[
  N_1 = \left  \lfloor \sqrt{2 N \gamma} \right \rfloor \le \gamma \sqrt{2 N /\gamma} \le \gamma N_2.
  \]

  Also,  
  \[
\left ( \sqrt{2N\gamma} - 1 \right )    \sqrt{2 N/ \gamma}   \leq  N_1N_2 \leq   \sqrt{2 N \gamma}  \left (1+\sqrt{2N/\gamma} \right ),
  \]
hence
\[
\sqrt{2N} (\sqrt{2N} - \sqrt{3}/3)     \leq N_1 N_2 \leq  (2+ \sqrt{2}) N
\]
  since $N \ge \gamma \ge 3$.
  
  Finally, from Lemma~\ref{lem:phi} and the fact that $ (\sqrt{8x+1}-1)/2 < \sqrt{2x} - 3/8$ for all $x\ge 1$, we have $s =  \gamma \phi^{-1}(N/\gamma) < \gamma \left ( \sqrt{2N/\gamma} - 3/8\right ) <  \sqrt{2N\gamma} - 1 \le N_1$ since $N \ge\gamma \ge 3$.   Therefore, we can apply Lemma~\ref{lem:estim_card_sum_Ks}: we have, from~\eqref{eq:ineqcardKs},
\[
  \card \K_s = \gamma (1 + \lfloor s/\gamma \rfloor) (s/\gamma - \lfloor s/\gamma \rfloor / 2) + O(s/\gamma)
   \le \gamma \phi(s/\gamma) + O(s/\gamma)
   \le N + O(s/\gamma).
  \]
  The lower bound is a restatement of the lower bound  from~\eqref{eq:ineqcardKs}: 
\[
\card \K_s\ge \gamma \varphi(s/\gamma) = N. 
\]
\end{proof}
\begin{corollary}\label{cor:estim_sumKs}
  With the assumptions of Lemma~\ref{lem:estim_cardKs},
  put $\lambda = N/\gamma$. Then, we have
  \[
  \Omega_\gamma(N, N_1, N_2) = \psi(\lambda) N \gamma + O(N),
  \]
  where
  \begin{equation}\label{eq:defpsi}
  \psi(\lambda) = \frac{1 + \lfloor \phi^{-1}(\lambda) \rfloor}{12 \lambda}
  \left( 6\phi^{-1}(\lambda)^2 - \lfloor \phi^{-1}(\lambda) \rfloor - 2
  \lfloor \phi^{-1}(\lambda) \rfloor^2\right).
  \end{equation}
  
  Note that if $\gamma = o(N)$, we have $\lambda \rightarrow \infty$ and
  $\phi^{-1}(\lambda) = \sqrt{2\lambda} + O(1)$, so that
  $\psi(\lambda) = 2\sqrt{2\lambda}/3 + O(1)$ and 
  \[
  \Omega_\gamma(N, N_1, N_2) = \frac{2\sqrt{2}}{3} N^{3/2} \gamma^{1/2} + O(N\gamma).
  \]
\end{corollary}
\begin{proof}
  Lemma~\ref{lem:estim_cardKs} shows us that the assumptions of 
  Lemma~\ref{lem:estim_card_sum_Ks} are satisfied; we thus get 
  \[
  \sum_{(i, j)\in \K_s} (i + j\gamma)
  = \psi(\lambda) N \gamma + O(N).
  \]
  Further, note that for our value of $s$, 
  the term $s(N - \card \K_s)$ from~Lemma~\ref{lem:bnd_using_ks} is $O(s^2/\gamma)$, which is   $O(\gamma \phi^{-1}(\lambda)^2) = O(N \phi^{-1}(\lambda)^2/\lambda) = O(N)$, thanks to Lemma~\ref{lem:phi}.
  This gives $\Omega_\gamma(N, N_1, N_2) \ge \psi(\lambda)N \gamma + O(N)$. For the upper bound, as $\card \K_s \ge N$, we have
  \[
\Omega_\gamma \le    \sum_{(i, j)\in \K_s} (i + j\gamma), 
  \]
  from which the result follows. 
\end{proof}

\begin{remark}\label{rem:gamma13} We can obtain a similar result for $1 < \gamma < 3$ if we set  $N_1 = 1+ \left \lfloor \sqrt{2 N \gamma} \right \rfloor$ and $N_2 = 1 + \left \lceil \sqrt{2N/\gamma} \, \right \rceil$. 
\end{remark}  

\begin{lemma}\label{lem:encadr_psi}
    For $x\in [1, +\infty)$, we have
        \[
        0 \le \psi(x) - \frac{2\varphi^{-1}(x)}{3} + \frac{1}{6}\le  \frac{2-\sqrt{3}}{6},
        \]
        and $\psi(x) - 2 \varphi^{-1}(x)/3 + 1/6 = O(x^{-1/2})$ when $x\rightarrow \infty$.
\end{lemma}
\begin{proof}
For $x \ge 1$, we write $x = \varphi(z)$, with $z\ge 1$.
  We also define $\{z\} = z -\lfloor z\rfloor \in [0, 1).$ 
  We have 
  \begin{align*} \label{eq:psiphi}
    \psi(\varphi(z)) & = \frac{6z^2 - \lfloor z\rfloor - 2\lfloor z\rfloor ^2}{12z - 6\lfloor z\rfloor}, \\ 
    &     = \frac{2}{3} z - \frac{1}{6} + \frac{\{z\} - \{z\}^2}{6\{z\} + 3\lfloor z\rfloor}. 
  \end{align*}
  This shows that $\psi(\varphi(z)) - 2z/3 + 1/6 = O(1/z),$ which gives the asymptotic part of the Lemma, as $z = \sqrt{2x} + O(1)$; this also shows that  
  \[
 0 \le \psi(\varphi(z)) - \frac{2}{3} z + \frac{1}{6} \le \frac{\{z\} - \{z\}^2}{6\{z\} + 3}.
  \]
  as $z\ge 1$. The lemma then follows from the fact that $(u - u^2) /(6u + 3) \le (2 - \sqrt{3})/6$ for $u\in [0, 1]$.
\end{proof}

Let $k \in \nat, k \geq 1$, $\varphi_{|[k,k+1)} : x \in [k,k+1) \mapsto (1+k)(x-k/2) \in [k(k+1)/2,(k+1)(k+2)/2)$. For $y \in [k(k+1)/2,(k+1)(k+2)/2)$, we have
$ k = \lfloor \left ( \sqrt{1+8y} -1 \right )/2 \rfloor$ and 
\[
  \varphi^{-1}(y) =  \frac{y}{1+k} + \frac{k}{2}.
\]
Hence 
\[
  \varphi^{-1}(y) =  \frac{y}{1+u(y)} + \frac{u(y)}{2},
\]
where  $u(y) = \left \lfloor  \sqrt{1/4+2y} -1/2 \right \rfloor$ for all $y \geq 1$. 

Writing 
\[
g(y) = \frac{1}{3}\left( \frac{2y}{1+u(y)} + u(y)\right) - \frac{1}{6},
\]
we deduce from the Lemma that $g(y) \le \psi(y) \le g(y) + \frac{2-\sqrt{3}}{6}$ and $\psi(y) - g(y) = O(y^{-1/2}).$ We further check that $\psi(1) = g(1).$

For the sake of completeness, we conclude by deriving an explicit expression for $\psi$. For $y \in [k(k+1)/2,(k+1)(k+2)/2)$,  we have $\varphi^{-1} (y) \in [k,k+1)$. We finally obtain from~\cref{eq:defpsi} the following explicit expression for $\psi$: 
\begin{align*}
  \psi (y) & = \frac{6 \left ( \frac{y}{1+k} + \frac{k}{2} \right )^2 - k - 2k^2}{12 \left ( \frac{y}{1+k} + \frac{k}{2} \right ) - 6k}
  = \frac{12 y^2+12 y k(k+1) -k (k^3+4k^2+5k+2)}{24(k+1)y}\\
  & = \frac{12 y^2+12 y u(y)(u(y)+1) - u(y) (u(y)^3+4 u(y)^2+5u(y)+2)}{24(u(y)+1)y}.
\end{align*}

\section{Volume estimates for rigorous interpolants at the Chebyshev nodes}\label{app:extalt2D}

Let $a_1<b_1$, $a_2 < b_2$, let $i=0,\ldots, N-1$, let $f_i$ be a function defined over $[a_1,b_1]\times [a_2,b_2]$. We shall use the following results for the functions $f_i$ defined in~\eqref{eq:fi2D}. If we interpolate $f_i$ by $Q_i (x,t) \in \ree_{N_1-1,N_2-1} [x,t]$ at pairs of Chebyshev nodes $
 \mu_{k_1,N_1-1,[a_1,b_1]},\mu_{k_2,N_2-1,[a_2,b_2]}, 0 \le k_1 \le N_1 -1,  0 \le k_2 \le N_2-1$, cf. \Cref{subsec:overview:lebesgue},  the corresponding coefficients
$ ( c_{k_1,k_2,i} )_{\substack{0 \le k_1 \le N_1 -1,\\0 \le k_2 \le N_2 -1}}$ are given by
\begin{multline*}
  ( c_{k_1,k_2,i} )_{\substack{0 \le k_1 \le N_1 -1,\\0 \le k_2 \le N_2 -1}} \\\ = \frac{4}{N_1 N_2} \textrm{2D-DCT-II} \left ( (f_i(
\mu_{N_1 - 1 - \ell_1,N_1-1,[a_1,b_1]},\mu_{N_2 - 1 - \ell_2,N_2-1,[a_2,b_2]} 
  )_{\substack{0 \le \ell_1 \le N_1 -1,\\0 \le \ell_2 \le N_2 -1}}\right )
\end{multline*}
from~\eqref{eq:coeffs-interpol2D}.

Let $\rho_1, \rho_2 > 1$,  let $f_0, \ldots, f_{N-1}$ be  functions analytic in a neighborhood of ${E_{\rho_1,a_1,b_1,\rho_2,a_2,b_2}}$, cf.~\Cref{subsec:unifapprox}.
 We introduce the $N\times (N_1 N_2 + N)$ matrix $A = \left( A_1 | A_2 \right)$, defined by
\begin{equation}\label{eq:A1A2_2D}
  \begin{gathered}
    (A_1)_{i,(k_1,k_2)} = 
  \left ( \frac{c_{k_1,k_2,i}}{2^{\delta_{0k_1}+\delta_{0k_2}}} \right )_{\substack{0 \le i \le N - 1,\\ 0 \le k_1 \le N_1 -1,\, 0 \le k_2 \le N_2-1}},\\
  (A_2)_{i,j}  =   
  \, 64 \Mbiv(f_i)
  \left ( \frac{1}{\rho_1^{N_1}}+ \frac{1}{\rho_2^{N_2}}\right) \delta_{ij}, 0 \leq i,j \leq N -1.
  \end{gathered}
\end{equation}
Recall from Proposition~\ref{prop:ineqcauchy2D} that $\norminf{ f_i - Q_i  } \le A_2[i,i],$ $ i = 0, \ldots, N-1.$

The diagonal right part, $A_2$,  of the matrix will be used for controlling that the functions $P_0(ux,v(f(x)+t)), P_1(ux,v(f(x)+t))$, output by the lattice basis reduction process, are uniformly small over $[a_1,b_1]\times[a_2,b_2]$; this accounts for the presence of the $64 \Mbiv(f_i) \left ( \frac{1}{\rho_1^{N_1}}+ \frac{1}{\rho_2^{N_2}}\right)$ remainder term for the approximation of $P_i(ux, v(f(x)+t))$ by its interpolation polynomial at the 2D Chebyshev points.

The smallness of those functions is a consequence of the smallness of the lattice volumes that we now estimate.

\subsection{Volume estimates}
For the sake of simplicity, we now assume that $\rho_1, \rho_2 \ge 2$.  The next result gives an upper bound for the volume of the lattice generated by the rows of $A$:
\begin{theorem}\label{thm:majodet-ext2D}
  Let $\rho_2 \ge \rho_1 \ge 2$, $a_1<b_1$, $a_2 < b_2$.   We further assume that $\rho_1^{N_1} \le \rho_2^{N_2}$, and define $\gamma = \log \rho_2/\log \rho_1 \ge 1$.
  Let $f_0, \ldots, f_{N-1}$ be  functions analytic in a neighborhood of ${E_{\rho_1,a_1,b_1,\rho_2,a_2,b_2}}$.
  Then, we have
\[
(\det A \transp{A})^{1/2} \leq \left (2^7 \sqrt{N} \right )^N 2^{N_1N_2} \frac{\prod_{i=0}^{N-1} \Mbiv(f_i)}{\rho_1^{\Omega_\gamma(N, N_1, N_2)}}.
\] 
\end{theorem}
\begin{proof}
For $0 \le i, j \le N -1$, $0 \le k_1 \le N_1-1$, $0 \le k_2 \le N_2 -1$,  we have from Proposition~\ref{prop:ineqcauchy2D}, 
\[ 
\left | (A_1)_{i,(k_1,k_2)}  \right | \le  \left |  \frac{c_{k_1,k_2,i}}{2^{\delta_{0k_1}+\delta_{0k_2}}}  \right | \le 
   \frac{4 \Mbiv(f_i)}{ \rho_1^{k_1}  \rho_2^{k_2}} \frac{\rho_1^2 + 1}{\rho_1^2 - 1}\frac{\rho_2^2 + 1}{\rho_2^2 - 1}  \le \frac{16  \Mbiv(f_i)}{\rho_1^{k_1} \rho_2^{k_2}}. 
\]
Moreover, from the definition of $A_2$ and the assumption $\rho_1^{N_1} \le \rho_2^{N_2}$, we know that, for $0 \le i, j \le N -1$,
\[
\left | A_{2, i, j} \right | \le \frac{128 \Mbiv(f_i)}{\rho_1^{N_1}}.
\]
We now apply  Theorem~\ref{thm:stehle-ants}. We put ${\mathfrak r}_i = 128 \Mbiv(f_i)$ for $i =0, \ldots,~N-1$. Then, notice that the product of the $N$ largest elements among the ${\mathfrak c}_j$'s
\[
1, \ldots, \frac{1}{\rho_1^{k_1} \rho_2^{k_2}}, \ldots, \frac{1}{\rho_1^{N_1-1}\rho_2^{N_2-1}}, \underbrace{\frac{1}{\rho_1^{N_1}}, \ldots,  \frac{1}{\rho_1^{N_1}}}_{N \textrm{ times}},
\]
is upper bounded by $\rho_1^{-\Omega_{\gamma}(N,N_1, N_2)}$ thanks to the assumption $\rho_1^{\gamma} = \rho_2$, or equivalently, $\rho_1^{-i}\rho_2^{-j} = \rho_1^{-i-\gamma j}$.

We can now apply Theorem~\ref{thm:stehle-ants} to obtain the bound
\[
(\det A \transp{A})^{1/2} \leq \left (2^7 \sqrt{N} \right )^N 2^{\frac{N_1N_2 + N}{2}}  \frac{\prod_{i=0}^{N-1} \Mbiv(f_i)}{\rho_1^{\Omega_\gamma(N, N_1, N_2)}}, 
  \]
  from which the theorem follows, in view of the preliminary assumption
  that $N_1 N_2 \ge N$. 
\end{proof}

Corollary~\ref{cor:estim_sumKs} allows us to give asymptotic versions of Theorem~\ref{thm:majodet-ext2D}, which will be more convenient in the sequel. 
\begin{corollary}\label{cor:majodet2D}
  Let $\rho_1, \rho_2 \ge 2$ such that $\gamma = \log \rho_2/\log \rho_1 \in [3,N]$. Let $a_1<b_1$, $a_2 < b_2$,   
  $N_1 =\left  \lfloor \sqrt{2 N \gamma} \right \rfloor$ and $N_2 =  \left \lceil \sqrt{2N/\gamma} \right \rceil$.  Let $f_0, \ldots, f_{N-1}$ be  functions analytic in a neighborhood of ${E_{\rho_1,a_1,b_1,\rho_2,a_2,b_2}}$.

Assuming $N \rightarrow \infty$,  we  obtain 
\[
  (\det A \transp{A})^{1/2} \leq 2^{O(N \log N)} \frac{\prod_{i=0}^{N-1} \Mbiv(f_i)}{\rho_1^{\psi(N/\gamma) N \gamma + O(N)}}.
\]
\end{corollary}
We now specialize this statement to our situation, where we shall use the ordered list of functions 
\begin{multline}\label{eq:fi2D}
[ f_i, 0 \le i \le \underbrace{(d+1)(d+2)/2}_{=:N} - 1 ] = [  (x,t) \mapsto f_{k,\ell} (x,t),  0 \le \ell \le d, 0 \le k \le d - \ell ]  \\ := [  (x, t) \mapsto u^k x^k v^\ell (f(x)+t)^\ell,  0 \le \ell \le d, 0 \le k \le d - \ell ]. 
\end{multline}
In this context, it is convenient to replace, in the definition of $A$ and in Theorem~\ref{thm:majodet-ext2D}, $\Mbiv(f_i)$ by its upper bound $u^k \Muniv(x)^k v^\ell \Mbiv(f(x)+t)^\ell$; the proof carries over to that setting without any modification.

\begin{remark}\label{rem:hatA2D}
  By construction, $\left | \hat{A} [i,j]  \right | \leq  \left | A [i,j]  \right |$ for all $i,j$. Hence, Theorem~\ref{thm:majodet-ext2D} and its corollaries, which proceed by upper bounding the coefficients of $A$ and applying~Theorem~\ref{thm:stehle-ants}, also hold for  $(\det \hat{A} \transp{\hat{A}})^{1/2}$.
  \end{remark}
 
\begin{corollary}\label{cor:bnd_with_alpha2D}
  Let $d \ge 1$, $N =(d+1)(d+2)/2$, $u, v \in \nat \setminus \{ 0 \}$.  Let $\rho_1, \rho_2 \ge 2$ such that $\gamma = \log \rho_2/\log \rho_1 \in [3,N]$. Let $a_1<b_1$, $a_2 <b_2$,    $N_1 = \left \lfloor \sqrt{2 N \gamma} \right \rfloor$ and $N_2 = \left \lceil \sqrt{2N/\gamma}\, \right \rceil$. Let $f$ be a function analytic in a   neighborhood of ${E_{\rho_1,a_1,b_1}}$. Define
  the matrices $A_1$, $A_2$, $A = (A_1 | A_2)$ as in \eqref{eq:A1A2_2D}, and   the quantity   $\Delta_{N,N_1, N_2,[a_1,b_1],[a_2, b_2], \rho_1, \rho_2} := (\det A \transp{A})^{1/2}$.
  We have, as $d \rightarrow + \infty$,
\begin{multline*}
  \Delta_{N,N_1, N_2,[a_1,b_1],[a_2, b_2], \rho_1, \rho_2}^{1/(N-1)} \leq  2^{O(1)} \sqrt{N}^{1 + o(1)} \frac{(uv)^{d/3 + O(1)}}{\rho_1^{\psi(N/\gamma) \gamma + O(1)}}  \\
\left ( \frac{b_1-a_1}{2} \rho_1 + \left | \frac{b_1+a_1}{2}\right |\right )^{d/3 + O(1)}    \left (\Muniv(f) +  \frac{b_2-a_2}{2} \rho_2 + \left | \frac{b_2+a_2}{2}\right | \right )^{d/3 + O(1)}.
\end{multline*}
\end{corollary}
\begin{proof}
Note that each $f_{k,\ell}$ is analytic in a neighborhood of ${E_{\rho_1,a_1,b_1,\rho_2,a_2,b_2}}$. Also, since $\sum_{0 \leq k + \ell \leq d} k = \sum_{0 \leq k + \ell \leq d} \ell =  dN/3$, the exponent of $uv$, $\frac{b_1-a_1}{2} \left ( \frac{\rho_1 + \rho_1^{-1}}{2} \right ) + \left | \frac{b_1+a_1}{2}\right | \le \frac{b_1 - a_1}{2} \rho_1  + \left| \frac{b_1+a_1}{2}\right|$ and $\Muniv(f) +  \frac{b_2-a_2}{2} \left ( \frac{\rho_2 + \rho_2^{-1}}{2} \right ) + \left | \frac{b_2+a_2}{2}\right | \le  \Muniv(f) +  \frac{b_2-a_2}{2}\rho_2 + \left | \frac{b_2+a_2}{2}\right |$
  is $\frac{dN}{3(N-1)} = \frac{d}{3} +O(1)$.
\end{proof}

\section{Practical remarks} \label{app:practical}

Algorithms~\ref{algo:build_two_var} and~\ref{algo:2variables} have been written with readability in mind. We now add some practical clarifications. 

\subsection{Efficiency}\label{sssec:eff2D}
The number of DCT calls (see Step~\ref{step:dct2D} in
\Cref{algo:build_two_var}) can be reduced from $O(N)$ to
$O(d)$ by noticing that, for a given function $F$ (we use $F(x) =
\varphi(x, t)$ for fixed $t$, cf. Step \ref{step:varphi2D} of \Cref{algo:build_two_var}), the DCT $\delta'$ of the vector $u'$
associated with $xF(x)$ can be deduced from the DCT $\delta$ of the
vector $u$ associated with $F(x)$ via the following formulas, which are
easily deduced from the recurrence relation $2x T_n(x) = T_{n+1}(x) +
T_{n-1}(x)$: 
  \begin{itemize}
  \item $\delta'[0] = (b_1-a_1) \delta[1]/2 + (b_1+a_1) \delta[0]/2$,
  \item $\delta'[k] = (b_1-a_1) (\delta[k-1] + \delta[k+1])/4 + (b_1+a_1)\delta[k]/2$, $1\leq k\leq n-2$,
  \item $\delta'[n-1] = (b_1-a_1) \delta[n-2]/4 + (b_1+a_1)\delta[n-1]/2$.
  \end{itemize}

  This gives, in practice, a significant speedup in the construction   of the matrix for large $d$.

A similar strategy applies for the computation of the remainder matrix $M_r$.
  
\subsection{Overestimation issues}\label{sssec:overst2D}
  We now discuss the fact that $B_f \ge \Muniv(f)$ is expected as an input to Algorithm~\ref{algo:build_two_var}. For some functions such as $\exp$ or $\Gamma$, we can take advantage of a closed expression, or a very simple upper bound, for this maximum. Otherwise, either we develop a dedicated routine to derive a tight estimate of this value, or we can use interval or ball arithmetic~\cite{Tucker2011,Johansson2017} to quickly obtain an upper bound, which then may raise overestimation issues. So far, our experiments, which use Arb\footnote{\url{https://flintlib.org/}}, resp. MPFI\footnote{\url{https://gitlab.inria.fr/mpfi}}, for every ball, resp. interval, arithmetic based computation, did not show any problematic overestimation -- we thus did not have to develop dedicated routines.  

This remark leads to the fact that we may overestimate $B_f$, thus $M_r$, in implementations of the Algorithm. Still, if the condition stated at Step~\ref{step:cond_lll2D} of Algorithm~\ref{algo:2variables} is satisfied with these (possibly overestimated) computed values, then this condition is also satisfied for the actual values and Theorem~\ref{thm:alternants_interpolants2D} and Corollary~\ref{cor:main2D} apply. This overestimation may however increase the computational cost.

  \subsection{Choice of norms}\label{sssec:choice_norms2D}
  
  Besides the heuristic coprimality condition, the success condition of Algorithm~\ref{algo:2variables} is expressed at Step~\ref{step:cond_lll2D} in terms of the Euclidean norm of the vectors. This is a  mere convenience related to the fact that the bounds on the  LLL algorithm are expressed in terms of this norm, making it more  tractable in our proofs of correctness / complexity analysis. 
  
Alternatively, one may make the choice of a success condition expressed in terms of the 1-norm, namely
\begin{multline}\label{eq:norm1_2D}
  \max_{i = 0, 1} \bigg (  \| ( M_{LLL}[i,j] )_{0\le j \le N + N_1 N_2 -1} \|_1   \\
   + (N+N_1N_2)\frac{ \| ( M_{LLL}[i,j] )_{N_1 N_2\le j \le N + N_1 N_2 -1} \|_1}{4N}  \bigg )< 2^{\Prec}.
\end{multline}
Indeed, as we have seen in the proof of Theorem~\ref{thm:alternants_interpolants2D}, this condition means that from the vector we can derive a polynomial   $P$ such that $P(ux, v(f(x)+t)) = \sum_{i_1=0}^{N_1 -1}\sum_{i_2=0}^{N_2 -1} c_{i_1,i_2} T_{i_1,[a_1,b_1]}(x)T_{i_2,[a_2,b_2]}(t) + R(x,t)$, with $\sum_{i_1=0}^{N_1-1}\sum_{i_2=0}^{N_2-1}  |c_{i_1,i_2}| +$  $\norminf{R(x,t)} $ $ < 1$. Since   $\norminf{T_{i, [a_k,b_k]}} = 1$ for all $i$ and $k=1,2$, we see that \eqref{eq:norm1_2D}  guarantees that the polynomial $P$ verifies $\norminf{P(ux, v(f(x)+t))} < 1$
  over $[a_1, b_1] \times [a_2,b_2]$, which is the key criterion for the success of the algorithm. 
  As this condition is slightly more efficient in practice, we recommend  using it in any implementation of our algorithm.

\subsection{Rounding issues}\label{sssec:rounding2D}

In Algorithm~\ref{algo:build_two_var}, the computation of $\Prec$ at Step~\ref{step:prec2D}, as well as those performed at Steps~\ref{step:mc2D} and~\ref{step:mr2D}, as written, require correct rounding, and may raise issues such as those this paper aims at solving. We show that this is not necessary: it suffices to round approximations whose error is less than
$1/2$ and whose direction is known. This changes each rounded quantity by at most one unit, in a controlled direction, and only requires a harmless adaptation of the proofs.

\paragraph{\textbf{Convention.}} Throughout this subsection, $A = (A_1 \mid A_2)$
denotes the matrix of~\Cref{ssec:algo1} in which the
diagonal entries of $A_2$ are the quantities actually used by
Algorithm~\ref{algo:build_two_var}, namely
\[
  A_2[i,i] = R_0\, (u B_x)^k \bigl(v (B_f + B_t)\bigr)^\ell ,
  \qquad i \leftrightarrow (k,\ell),
\]
where $B_x$ and $B_t$ are defined at Step~\ref{step:BxBt} and $B_f$ is
any certified upper bound for $\max_{z \in \mathcal{E}_{\rho_1, a_1,
b_1}} |f(z)|$ (see~\Cref{sssec:overst2D}). Since $B_x \ge
\Muniv(x)$ and $B_f + B_t \ge \Mbiv(f(x)+t)$, we still have $\|f_i -
Q_i\|_\infty \le A_2[i,i]$ for all $i$, which is the only property of
$A_2$ used in the proof of Theorem~\ref{thm:alternants_interpolants2D}.
Likewise, Theorem~\ref{thm:majodet-ext2D} holds with $\Mbiv(f_i)$
replaced by $u^k v^\ell B_x^k (B_f + B_t)^\ell$.

We will use the following obvious statements.
\begin{lemma}\label{lem:rounding2D}
  Let $x, \tilde{x} \in \ree$.
  \begin{enumerate}
  \item If $x \le \tilde{x} \le x + 1/2$, then $\lceil \tilde{x} \rceil
    - \lceil x \rceil \in \{0, 1\}$.
  \item If $\tilde{x} \le x \le \tilde{x} + 1/2$, then $\lfloor x
    \rfloor - \lfloor \tilde{x} \rfloor \in \{0, 1\}$ and $0 \le x -
    \lfloor \tilde{x} \rfloor < 3/2$.
  \item If $x \tilde{x} \ge 0$ and $|\tilde{x}| \le |x| \le |\tilde{x}|
    + 1/2$, then $[x]_0 - [\tilde{x}]_0 \in \{0, \sgn(x)\}$, $\bigl|
    [\tilde{x}]_0 \bigr| \le |x|$ and $\bigl| x - [\tilde{x}]_0 \bigr| <
    3/2$.
  \end{enumerate}
\end{lemma}
Approximations satisfying these one-sided conditions are easy to compute, either with floating-point arithmetic using a sufficient intermediate precision and the suitable directed rounding mode for each operation, or with ball arithmetic as provided, for instance, by Arb. For the third item, if a ball $[m \pm r]$ with $r \le 1/4$ contains $x$, one may take $\tilde{x} = 0$ if $|m| \le r$ and $\tilde{x} = \sgn(m)(|m| - r)$ otherwise. The intermediate precision required is discussed in \Cref{app:prec2D}.

\paragraph{\textbf{Computed quantities.}} In the sequel, we denote by $\PracPrec$, $\PracMc$, $\PracMr$ the quantities computed in this way. Note that $M_c$ and $M_r$ are then defined with $\PracPrec$ instead of $\Prec$.

 Let $\tau = -\log_2 (R_0) + \log_2 (N)$, so that $\Prec = \lceil \tau \rceil + 2$. In practice, we compute:
\begin{itemize}
\item $\PracPrec = \lceil \tilde{\tau} \rceil + 3$, where
  $\tau \le \tilde{\tau} \le \tau + 1/2$;
  by Lemma~\ref{lem:rounding2D}, $\PracPrec \in \{\Prec + 1, \Prec +
  2\}$;
\item $\PracMc[i,j] = [\tilde{x}]_0$, where $\tilde{x}$ approximates $x
  = 2^{\PracPrec} A_1[i,j]$ in the sense of the third item of
  Lemma~\ref{lem:rounding2D}, for $0 \le i \le N-1$ and $0 \le j \le N_1
  N_2 - 1$;
\item $\PracMr[i,i] = \lfloor \tilde{x} \rfloor$, where $\tilde{x}$
  approximates $x = 2^{\PracPrec} A_2[i,i]$ in the sense of the second
  item of Lemma~\ref{lem:rounding2D}, for $0 \le i \le N-1$.
\end{itemize}
The extra bit of precision in $\PracPrec$ compensates for the possible
loss of one unit in $\PracMr$. It only replaces the bounds of
Lemma~\ref{lem:prec2D} by $\frac{2}{N} (\rho_1^{-N_1} + \rho_2^{-N_2})
\le 2^{-\PracPrec} \le \frac{8}{N} (\rho_1^{-N_1} + \rho_2^{-N_2})$,
which does not affect the asymptotic analysis of
\Cref{sec:two-var}.

We let $\PrachatA = 2^{-\PracPrec} (\PracMc \mid \PracMr)$ denote the
matrix actually computed, and $\hat{A}_{2,\mathsf{comp}} = 2^{-\PracPrec} \PracMr$ its
right block. Algorithm~\ref{algo:2variables} is run on $(\PracMc \mid
\PracMr)$, and $\Prec$ is replaced by $\PracPrec$ in
Step~\ref{step:cond_lll2D}.

\begin{lemma}\label{lem:computedA2D}
Assume that $\min(u\Muniv(x), v\Mbiv(f(x)+t)) \ge 1$. Then, 
  for all $i, j$, we have
  \begin{enumerate}
  \item $|\PrachatA[i,j]| \le |A[i,j]|$;
  \item $|A[i,j] - \PrachatA[i,j]| < 2^{1 - \PracPrec}$;
  \item $\hat{A}_{2,\mathsf{comp}}[i,i] \ge 2^{2 - \PracPrec} N$.
  \end{enumerate}
\end{lemma}

\begin{proof}
  The first two items follow directly from Lemma~\ref{lem:rounding2D}:
  for the left block, apply the third item to $x = 2^{\PracPrec}
  A_1[i,j]$; for the right block, apply the second item to $x =
  2^{\PracPrec} A_2[i,i] > 0$. For the third item, the assumption
  $\min(u\Muniv(x), v\Mbiv(f(x)+t)) \ge 1$ gives $u B_x \ge 1$ and $v(B_f
  + B_t) \ge 1$, hence $A_2[i,i] \ge R_0$. As $\PracPrec \ge \Prec + 1
  \ge \tau + 3$, we get $2^{\PracPrec} A_2[i,i] \ge
  2^{\PracPrec} R_0 \ge 8N$. It follows that $\PracMr[i,i] \ge 8N - 1 \ge
  4N$.
\end{proof}

\begin{remark}
  By the first item of Lemma~\ref{lem:computedA2D}, the argument of
  Remark~\ref{rem:hatA2D} applies verbatim to $\PrachatA$:
  Theorem~\ref{thm:majodet-ext2D} and its corollaries hold for
  $\PrachatA$.
\end{remark}

\begin{remark}\label{rem:adapt2D}
  Theorem~\ref{thm:alternants_interpolants2D} holds with $\hat{A}$
  replaced by $\PrachatA$, up to the following adaptation of its
  proof. Let $\Lambda \in \Z^N$ be such that $\normeucl{\Lambda
  \PrachatA} \le 1/(N + N_1N_2)$. By the second and third items of
  Lemma~\ref{lem:computedA2D}, for $j = 0, \ldots, N + N_1N_2 - 1$,
  \begin{multline*}
    \bigl| (\Lambda A)[j] - (\Lambda \PrachatA)[j] \bigr|
    \le \sum_{0 \le k + \ell \le d} |\lambda_{k,\ell}|\, 2^{1 - \PracPrec} \\
    \le \frac{1}{2N} \sum_{0 \le k + \ell \le d} |\lambda_{k,\ell}|
    \min_i \hat{A}_{2,\mathsf{comp}}[i,i]
    \le \frac{\normun{\Lambda \hat{A}_{2,\mathsf{comp}}}}{2N}.
  \end{multline*}
  Hence
  \[
    \normun{\Lambda A}
    \le \normun{\Lambda \PrachatA} + (N + N_1N_2)
      \frac{\normun{\Lambda \hat{A}_{2,\mathsf{comp}}}}{2N}
    \le \frac{1}{\sqrt{N + N_1N_2}} + \frac{1}{2\sqrt{N}} < 1
  \]
  for $N \ge 3$ and $N_1, N_2 \ge 2$, and the rest of the proof is
  unchanged. Accordingly, the success condition~\eqref{eq:norm1_2D}
  becomes
  \begin{multline*}
    \max_{i = 0, 1} \biggl( \normun{( M_{LLL}[i,j] )_{0 \le j \le N +
    N_1N_2 - 1}} \\
    + (N + N_1N_2) \frac{\normun{( M_{LLL}[i,j] )_{N_1N_2 \le j \le N +
    N_1N_2 - 1}}}{2N} \biggr) < 2^{\PracPrec}.
  \end{multline*}
\end{remark}

\subsection{Newton polynomials}\label{sssec:newton2D}

In practice, we replace the monomial functions   $\{u^k x^k\}_{0 \le k \le d}$ with Newton polynomial functions $\{ u x (ux - 1) \cdots (ux - k +1)/k! \}_{0 \le k \le d}$. In both cases, the substitution $x = X/u$ yields integer values - respectively   $\{X^k\}_{0 \le k \le d}$ and  $\{ X (X - 1) \cdots (X - k +1)/k! \}_{0 \le k \le d}$. Likewise, we replace the ``monomials'' $\{v^k (f(x)+t)^k\}_{0 \le k \le d}$ with ``Newton polynomials''  $\{ v (f(x)+t) (v (f(x)+t) - 1) \cdots (v (f(x)+t) - k +1)/k! \}_{0 \le k \le d}$. Hence, the changes to operate are:
  \begin{itemize}
  \item Step~\ref{step:varphi2D}, Algorithm~\ref{algo:build_two_var},

    $\varphi \leftarrow  \left  ((x,t) \mapsto  \left(\prod_{j = 1}^{k} (u x -j + 1)/j \right) \left(\prod_{j = 1}^{\ell}  (v (f(x)+t) - j + 1)/j \right )\right)$,
  \item Step~\ref{step:mr2D}, Algorithm~\ref{algo:build_two_var},

    $M_r [i,i] \leftarrow \left  \lfloor 2^{\Prec} R_0 \prod_{j = 1}^{k} \frac{u B_x + j - 1 }{j}  \prod_{j = 1}^{\ell} \frac{v (B_f+B_t)+j -1}{j} \right \rfloor$,

  \item Step~\ref{step:monomials2D}, Algorithm~\ref{algo:2variables},

      $L_m \leftarrow \left  [ \prod_{j = 1}^{k} \frac{X_1 -j +1}{j}  \prod_{j = 1}^{\ell} \frac{X_2  -j + 1}{j} \textrm{ for } k = 0 \textrm{ to } d-\ell \textrm{ for } \ell = 0  \textrm{ to } d \right ]$.
  \end{itemize} 
  The use of Newton polynomials leads to smaller uniform norms, hence
  makes it possible to tackle larger intervals for the same $d$.  This
  improvement is asymptotically negligible (it contributes to a lower
  order term) but is quite significant in practice.

  Note that the optimizations described in \Cref{sssec:eff2D}
  can easily be adapted to the case of Newton polynomials.

\section{Complexity of Algorithms~\ref{algo:build_two_var} and \ref{algo:2variables}} \label{app:compalgo}

We now prove the statements from \Cref{sssec:basicplx}.

\subsection{Precision issues}\label{app:prec2D}
In order to estimate the complexity of \Cref{algo:build_two_var}, we need to evaluate the required precision for those algorithms. We choose to refer to the extended version of our work~\cite{BH2023} rather than including a long and technical precision analysis in this paper: its impact is solely on the complexity estimate of this section -- the correctness of our implementation is ensured by the use of interval arithmetic.

We summarize our results on this point in the following:
\begin{lemma}\label{lem:prec_comp}
  Put $\cM = \max (u, v, |a_1|, |b_1|, |a_2|, |b_2|, \rho_1, \rho_2, B_f,   \max_{[a_1,b_1]} |f'(x)|)$. Assume that $\rho_1, \rho_2$ are
  integers\footnote{This assumption is purely technical, made in order to avoid discussing the representation of $\rho_1, \rho_2$ and its size}. Then, the computations of Algorithm~\ref{algo:build_two_var} on
  input $a_1, b_1, a_2, b_2, f, d, N_1, N_2,$ $ u,$ $ v, \rho_1, \rho_2, B_f$ can be made
  in floating-point precision ${\InterPrec} = \max(N_1 \log \rho_1, N_2 \log \rho_2) + O(d \log \cM)$.
\end{lemma}

In the sequel, we assume  that interval evaluation (that we use in practice) at precision ${\InterPrec}$ of a function $f$ uses $O(1)$ evaluations of $f$ at precision ${\InterPrec}$. In our implementation, we used the Arb library~\cite{Johansson2017}.

\subsection{Proof of the complexity of Algorithm~\ref{algo:build_two_var}}\label{app:proofalgo1}

The most costly steps of \Cref{algo:build_two_var} appear in the loop \ref{step:for_ell2D}-\ref{step:end_for_ell2D} which can be performed using $O(d^2)$ evaluations of $f$ at precision $O({\InterPrec})$, and $O(d^2 N_1 N_2)$ multiplications of real numbers in   precision $O({\InterPrec})$, plus $O(d^2)$ DCTs of size $N_1 N_2$ in precision ${\InterPrec}$, for a total cost $O(d^2 (N_1 N_2)^2  \Mult({\InterPrec})  + d^2 C_{f, {\InterPrec}})$.

\subsection{Proof of the complexity of Algorithm~\ref{algo:2variables}}\label{app:proofalgo2}

  The main steps of \Cref{algo:2variables} are:
  \begin{itemize}
  \item A call to Algorithm~\ref{algo:build_two_var}, the cost of which is $\tO(d^6{\InterPrec})$ by Proposition~\ref{prop:cplxalgo1};
  \item A call to LLL on a lattice of dimension $N$ with entries of size
    $O({\InterPrec})$. 
  \end{itemize}
  For the second part, we use the $L^2$ algorithm~\cite{NguyenS09} on a lattice
  of dimension $N = O(d^2)$, embedded into $\R^{N+N_1N_2}$; we thus have complexity $O(d^4 N_1 N_2 \Mult(d^2) (d^2 + {\InterPrec}){\InterPrec})$. 
   This cost dominates the cost of  Algorithm~\ref{algo:2variables}.

\end{document}